\documentclass[a4paper,10pt,reqno, english]{amsart} 

\usepackage[utf8]{inputenc}
\usepackage{url,paralist}
\usepackage{caption}
\usepackage{subcaption}
\usepackage{makecell} 
\usepackage[colorlinks=true,urlcolor=blue,citecolor=magenta]{hyperref}
\usepackage{color}
\usepackage{float}
\usepackage{mathtools}
\usepackage{dsfont}
\usepackage{tikz}
\usepackage{xcolor}
\usepackage{amsfonts,amssymb,enumerate}
\usepackage{graphicx}
\usepackage{amsmath,bm,amsthm}
\usepackage{mathtools}
\usetikzlibrary{arrows.meta}
\usepackage[arrow,curve,matrix,tips,2cell]{xy}  
  \SelectTips{eu}{10} \UseTips
  \UseAllTwocells

\theoremstyle{plain}
\newtheorem{theorem}{Theorem}[section]
\newtheorem{lemma}[theorem]{Lemma}
\newtheorem{claim}[theorem]{Claim}
\newtheorem{corollary}[theorem]{Corollary}
\newtheorem{proposition}[theorem]{Proposition}

\newtheorem*{theorem*}{Theorem}

\newtheorem*{r-conjecture}{The Ramos conjecture}

\newtheorem*{claim*}{Claim}

\theoremstyle{definition}

\newtheorem{example}[theorem]{Example}

\newtheorem{remark}[theorem]{Remark}

\newcommand{\R}{\mathbb{R}}

\newcommand{\N}{\mathbb{N}}
\newcommand{\Z}{\mathbb{Z}}
\newcommand{\F}{\mathbb{F}}
\newcommand{\B}{\mathrm{B}}
\newcommand{\E}{\mathrm{V}}
\newcommand{\PP}{\mathbb{P}}
\newcommand{\pt}{\mathrm{pt}}

\newcommand{\M}{\mathcal{M}}
\newcommand{\HH}{\mathcal{H}}
\newcommand{\II}{\mathcal{I}}
\newcommand{\JJ}{\mathcal{J}}
\newcommand\Sym{\mathfrak{S}}

\newcommand{\flag}{\operatorname{Flag}}

\newcommand{\id}{\operatorname{id}}

\newcommand{\spann}{\operatorname{span}}
\newcommand{\card}{\operatorname{card}}

\newcommand{\map}{\operatorname{Map}}
\newcommand{\e}{\operatorname{e}}
\newcommand{\OO}{\operatorname{O}}

\newcommand{\RP}[1]{\mathbb{R}\mathrm{P}^{#1}} % real projective space 
\newcommand{\Gr}{\mathrm{G}}   % Grassmannian manifold 
\newcounter{commentcounter}
\newcommand{\commentJ}[1] % Comment by Jaime
{\stepcounter{commentcounter} {  
\textcolor{blue}{ \bf Comment~\arabic{commentcounter}(by Jaime):}  { \ttfamily #1}}}

\def\gg{\mathfrak{g}}
\def\hh{\mathfrak{h}}
\def\GL{{\rm GL}}
\def\Hom{{\rm Hom}}
\def\End{{\rm End}}
\def\Ff{\mathbb{F}}
\title[Many partitions of  mass assignments]
{Many partitions of  mass assignments}

\author[Blagojevi\'c]{Pavle V. M. Blagojevi\'{c}} 
\thanks{The research by Pavle V. M. Blagojevi\'{c} leading to these results has
        received funding from the Serbian Ministry of Science, Technological development and Innovations.}
\address{Inst. Math., FU Berlin, Arnimallee 2, 14195 Berlin, Germany\hfill\break
\mbox{\hspace{4mm}}Mat. Institut SANU, Knez Mihailova 36, 11001 Beograd, Serbia}
\email{blagojevic@math.fu-berlin.de} 
\author[Crabb]{Michael C. Crabb} 
\address{Institute of Mathematics, University of Aberdeen, Aberdeen AB24 3UE, UK}
\email{m.crabb@abdn.ac.uk}
\begin{document}

%------------------------------------  TITLE ------------------------------------%
\date{}
\dedicatory{Dedicated to the memory of  Frederick R. Cohen an exceptional mathematician\\ and an amazing human being}
\maketitle
%\tableofcontents
%------------------------------------ ABSTRACT ------------------------------------%
\begin{abstract}
In this paper, extending the recent work of authors with Calles Loperena \& Dimitrijevi\'c Blagojevi\'c, we give a general and complete treatment of problems of partition of mass assignments with prescribed arrangements of hyperplanes on Euclidean vector bundles.
Using a new configuration test map scheme, as well as an alternative topological framework, we are able to reprove known results, extend them to arbitrary bundles as well as to put various types of constraints on the solutions.
Moreover, the developed topological methods allow us to give new proofs and extend results of Guth \& Katz, Schnider and Sober\'on \& Takahashi.
In this way we place all these results under one ``roof''.

\end{abstract}
%--------------------------------------------------------------------------------------%

%--------------------------------------------------------------------------------------%
%--------------------------------------------------------------------------------------%
%--------------------------------------------------------------------------------------%
%--------------------------------------------------------------------------------------%
\section{Introduction}
\label{sec : Introduction and statement of main results}
%--------------------------------------------------------------------------------------%
%--------------------------------------------------------------------------------------%
%--------------------------------------------------------------------------------------%
%--------------------------------------------------------------------------------------%

Problems of the existence of mass partitions by affine hyperplanes in a Euclidean space have a long and exciting history since the 1930's ham-sandwich theorem of Hugo Steinhaus  and Karol Borsuk  \cite[Problem 123]{Mauldin-1981}.
The ham-sandwich theorem claims the existence of a hyperplane which equiparts $d$ given masses in a $d$-dimensional Euclidean space. 
For more details about the history on the ham-sandwich theorem and its interconnection with the non-existence of antipodal maps between spheres  consult  \cite{Beyer-Zardecki-2004} or \cite{Matousek-2003}.
The followup work by Branko Gr\"unbaum \cite{Grunbaum-1960}, Hugo Hadwiger \cite{Hadwiger-1966}, David Avis \cite{Avis-1984}, and a bit later by Edgar Ramos \cite{Ramos-1996}, demonstrated how an increase in complexity of mass partition questions naturally creates even more complicated problems related to the non-existence of equivariant maps. 
Topological challenges of the Gr\"unbaum--Hadwiger--Ramos hyperplane mass partition problems were discussed recently in \cite{Blagojevic-Frick-Haase-Ziegler-2018}.
For more information about various types of mass partition problems see the recent survey by Edgardo Rold\'{a}n-Pensado and Pablo Sober\'{o}n \cite{RoldanPensadoSoberon-2022}.

\medskip
In order to  motivate a study of partitions of mass assignments over Euclidean vector bundles as a natural extension of classical studies we first briefly recall the original problem.
For brevities  sake, from now on, we write ``GHR'' for ``Gr\"unbaum--Hadwiger--Ramos''.

\subsection{What is the GHR problem for masses?}
A {\em mass} in a Euclidean space $\E$ is assumed to be a finite Borel measure on $\E$  which vanishes on every affine hyperplane.  

\medskip
An {\em oriented affine hyperplane} $H(u;a)$ in $\E$ is given by
\begin{compactitem}[\quad --]
\item  a unit vector $u\in \E$, the unit normal to the associated affine hyperplane $H_{u;a}$, which in addition determines the orientation of the hyperplane, and by
\item a scalar $a\in\R$ which determines the distance of the associated affine hyperplane $H_{u;a}$ from the origin in direction $u$.
\end{compactitem}
The associated affine hyperplane is defined by $H_{u;a}:=\{x\in\E \colon \langle x,u\rangle =a\}$. 
Furthermore, the oriented affine hyperplane $H(u;a)$ defines two closed half-spaces by 
\[
H_{u;a}^{u}:=\{x\in \E \colon\langle x,u\rangle -a \geq 0\}\quad\text{and}\quad
H_{u;a}^{-u}:=\{x\in\E \colon \langle x,-u\rangle +a\geq 0\}.
\]
In other words, an oriented affine hyperplane is a triple $H(u;a)=(H_{u;a},H_{u;a}^{u},H_{u;a}^{-u})$. 

\medskip
An {\em arrangement} of $k$ (oriented) {\em affine hyperplanes}  $\HH$ in $\E$ is an ordered collection $\HH=\big( H(u_1;a_1), \ldots, H(u_k;a_k) \big)$ of $k$ oriented affine hyperplanes in $\E$.
Such an arrangement $\HH$ and a collection of unit normal vectors $(v_1,\dots,v_k)\in \{u_1,-u_1\}\times\dots\times \{u_k,-u_k\}$ to the elements of the arrangement $\HH$  determine an \emph{orthant}  as the intersection of the corresponding closed half-spaces:
\[
\mathcal{O}_{(v_1,\dots,v_k)}^{\HH}:=H_{u_1;a_1}^{v_{1}}\cap\cdots\cap H_{u_k;a_k}^{v_{k}}.
\]
There are $2^k=\card\big( \{u_1,-u_1\}\times\dots\times \{u_k,-u_k\}\big)$ orthants determined by the arrangement $\HH$.
The orthants are not necessarily distinct or non-empty. 
The arrangement of hyperplanes $\HH=\big( H(u_1;a_1), \ldots, H(u_k;a_k) \big)$ is {\em orthogonal} if $u_r\perp u_s$ for every $1\leq r<s\leq k$. 

\begin{figure}
\includegraphics[scale=0.8]{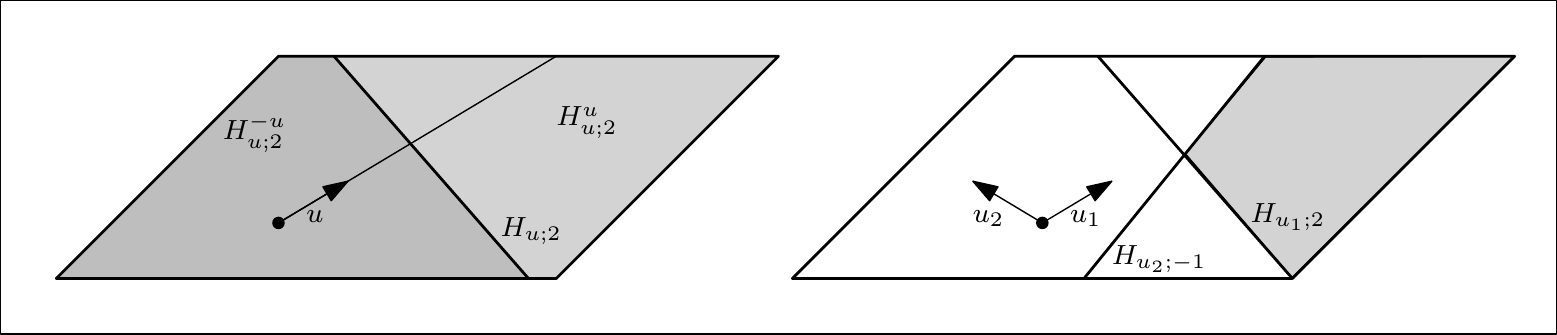}
\caption{\small An illustration of an oriented affine hyperplane, associated half-spaces, arrangement of two affine hyperplanes, and an orthant $\mathcal{O}_{(u_1,-u_2)}^{\HH}$ where $\HH=(H(u_1,2),H(u_2,-1))$.}	
\end{figure}
 
\medskip
Now, we say that an arrangement $\HH=\big( H(u_1;a_1), \ldots, H(u_k;a_k) \big)$ in $\E$ \emph{equiparts} a collection of masses $\M$ in $\E$ if and only if for every mass $\mu\in\M$ and every $(v_1,\dots,v_k)\in \{u_1,-u_1\}\times\dots\times \{u_k,-u_k\}$ holds:
\[
\mu(\mathcal{O}_{(v_1,\dots,v_k)}^{\HH})=\frac{1}{2^{k}}\,\mu(V).
\]
Furthermore, a collection of masses $\M$ in $\E$ {\em can be equiparted by an arrangement} of $k$ affine hyperplanes if there exists an arrangement $\HH$ of $k$ oriented affine hyperplanes in $\E$ which equiparts $\M$.

\medskip
The GHR problem for masses asks for {\em the minimal dimension $d=\Delta (j,k)$ of a Euclidean space $\E$ in which every collection $\M$ of $j$ masses can be equiparted by an arrangement of $k$  affine hyperplanes}. 

\medskip
The first few values of the function $\Delta$ can be derived from classical results.
Indeed, the ham-sandwich theorem \cite{Beyer-Zardecki-2004} implies that  $\Delta(d,1)=d$, an argument of Gr\"unbaum \cite{Grunbaum-1960} says that $\Delta(1,2)=2$, while the seminal work of Hadwiger \cite{Hadwiger-1966} yields that $\Delta(2,2)=3$, $\Delta(1,3)=3$.
Furthermore, Avis and Ramos showed that $\tfrac{2^k-1}{k}j\leq \Delta(j,k)$, while Peter Mani-Levitska, Sini\v{s}a Vre\'{c}ica \& Rade \v{Z}ivaljevi\'c in \cite[Thm.\,39]{ManiLevitska-Vrecica-Zivaljevic-2006} proved that $ \Delta(j,k)\ \le\ j + (2^{k-1}-1)2^{\lfloor\log_2j\rfloor}$.
The list of known values of the function $\Delta$ is given in \cite{Blagojevic-Frick-Haase-Ziegler-2016}.

\medskip
In our recent paper with Calles Loperena and Dimitrijevi\'c Blagojevi\'c \cite{BlagojevicCallesCrabbDimitrijevic}, motivated by the work of Patrick Schnider \cite{Schnider-2019} and  Ilani Axelrod-Freed \& Sober\'on \cite{Soberon}, we studied an extension of the GHR problem for masses to the problem for mass assignments over Grassmann manifolds.

%--------------------------------------------------------------------------------------%
%--------------------------------------------------------------------------------------%
\subsection{What is the GHR problem for mass assignments?} 
\label{subsec : What is the GHR problem for mass assignments?}
%--------------------------------------------------------------------------------------%
%--------------------------------------------------------------------------------------%
Let $M_+(X)$ be the space of all finite Borel measures on a topological space $X$ equipped with the weak topology.
That is the minimal topology on $M_+(X)$ with the property that for every bounded and upper semi-continuous function $f \colon X \longrightarrow \R$, the induced function $M_+(X)\longrightarrow\R$ given by $\nu \longmapsto \int f d\nu$, is upper semi-continuous.
For $X=\E$, a Euclidean space, the subspace of all masses is denoted by $M_+'(\E)\subseteq M_+(\E)$.

\medskip
Let $E$ be a Euclidean vector bundle over a path-connected space $B$ with fibre $E_b$ at $b\in B$.
Consider the associated fibre bundle
\begin{equation}
\label{eq: mass_assignment}	
\xymatrix{
 M_+'(E) := \{ (b, \nu) \mid b \in B ,\;   \nu \in M_+'(E_b) \}\ar[r] & B,\quad &(b,\nu)\ar@{|->}[r] &b.
}
\end{equation}
The topology on $M_+'(E)$ is defined using the local triviality of $E$ and the topology on fibres we chose. 
Any cross-section $\mu\colon B\longrightarrow M_+'(E)$ of the fibre bundle \eqref{eq: mass_assignment} is called {\em mass assignment} on the Euclidean vector bundle $E$.
In particular, $\mu(b)$ is a mass on $E_b$ for every $b\in B$.

\medskip
More generally, let us now write $M_+(E)\longrightarrow B$ for the locally trivial bundle with fibre at $b\in B$ the space $M_+(E_b)$ of finite Borel measures on the sphere $S(E_b)$.
A continuous section $\mu$ will be called a {\it family of (probability) measures} on $E$ if $\mu_b\in M_+(E_b)$ is a (probability)
measure for each $b\in B$. 
In the following we give an illustrative example of a family of probability measures on $E$.

\begin{example}
Let $E$ be a Euclidean vector bundle over a path-connected space $B$.
Suppose that $X\longrightarrow B$ is a finite cover embedded fibrewise in $E$, and suppose that $p: X\to [0,1]$
is a continuous function such that, for each $b\in B$,
$\sum_{x\in X_b} p(x)=1$.
For a Borel subset $A\subseteq E_b$, define $\mu_b(A):=\sum_{x\in A\cap X_b} p(x)$.
Then $\mu$ defines a family of probability measures on $S(E)$.	
\end{example}

\medskip
The GHR problem for mass assignments on a Euclidean vector bundle $E$ over $B$ asks {\em for all pairs of positive integers $(j,k)$ with the property that for every collection of $j$ mass assignments $\M=( \mu_1, \ldots , \mu_j )$ on $E$ there exists a point $b\in B$ such that the collection of $j$ masses $\M(b):=( \mu_1(b), \ldots , \mu_j(b) )$ on $E_b$ can be equiparted by an arrangement of $k$ affine hyperplanes in $E_b$}.
If we denote by $\Delta(E)$ the set of such pairs $(j,k)$, then the  GHR problem for mass assignments on $E$ is a question of describing the set $\Delta(E)\subseteq\N^2$.

\medskip
Recently, with Calles Loperena and Dimitrijevi\'c Blagojevi\'c \cite{BlagojevicCallesCrabbDimitrijevic}, we studied the GHR problem for mass assignment over tautological vector bundles over Grassmann manifolds.
In particular, with appropriate reformulation, the result \cite[Thm.\,1.5]{BlagojevicCallesCrabbDimitrijevic} can be stated as follows.

\begin{theorem}
\label{th : original result reformulated}
Let $d\geq 2$ and $\ell\geq 1$ be integers where $1\leq\ell\leq d$, and let $E_{\ell}^{d}$ be the tautological vector bundle over the Grassmann manifold $\Gr_{\ell}(\R^d)$ of all $\ell$-dimensional linear subspaces in $\R^d$.
Then
\[
\big\{(j,k)\in\N^2\, :\, 1\leq k\leq\ell,\, 2^{\lfloor\log_2j\rfloor}(2^{k-1}-1)+j\leq d\big\}\subseteq \Delta(E_{\ell}^{d}).
\]
\end{theorem}

\medskip
It is important to observe that the result of Theorem \ref{th : original result reformulated} does not really depend on the value of the parameter $\ell$.
In particular, for $\ell=d$ it recovers the upper bound of Mani-Levitska--Vre\'{c}ica--\v{Z}ivaljevi\'c  \cite[Thm.\,39]{ManiLevitska-Vrecica-Zivaljevic-2006} for the function $\Delta$.	

\medskip
In this paper, following the ideas of B\'ar\'any and Matou\v{s}ek \cite{BaranyMatousek2001} and Crabb \cite{Crabb2020}, we extend mass assignment partition problems in a Euclidean space by affine hyperplane arrangements to mass assignment partition problems on the unit Euclidean sphere by arrangements of equatorial spheres.
Additionally, we will restrict, and therefore simplify, our notions of mass and mass assignment.

\subsection{What are the GHR problems on spheres and sphere bundles?}\label{Sec 1.3}
First, we show how the GHR problem for masses in $\R^d$ induces the corresponding mass partition problem on the unit sphere in $\R^{d+1}$.

\medskip
Let $d\geq 1$ be an integer.
Embed $\R^d$ into $\R^{d+1}$ via the embedding $x\longmapsto (x,-1)$.
In this way $\R^d$ coincides with the tangent space to the unit sphere $S(\R^{d+1})\cong S^d$ at the point $y_0:=(0,\dots,0,-1)$.
Let $p\colon \R^d\longrightarrow \Lambda$ be the homeomorphism, between $\R^d$ and the open lower hemisphere $\Lambda:=\{y\in S(\R^{d+1}) \,:\, \langle y,y_0\rangle >0\}$ of the sphere $S(\R^{d+1})$, given by
$
x\longmapsto \frac{1}{\sqrt{\|x\|^2+1}}(x,-1)
$
for $x\in\R^d$.

\begin{figure}
\includegraphics[scale=0.8]{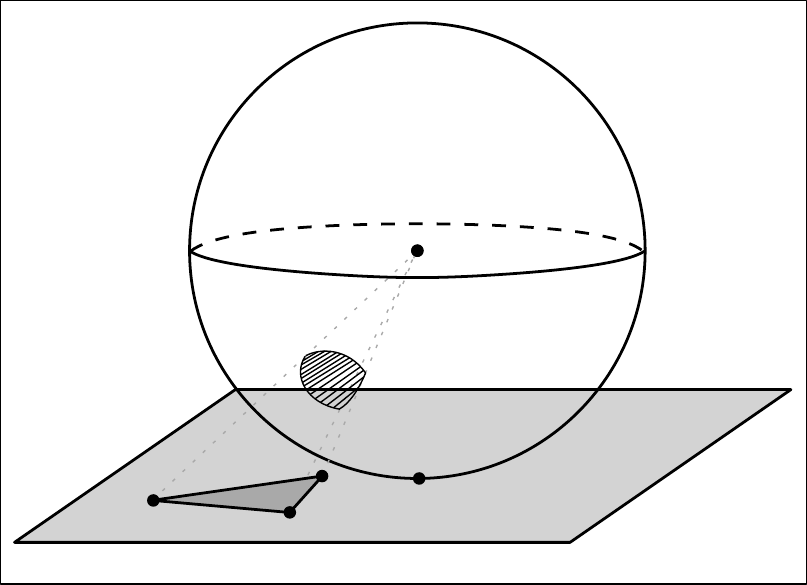}
\caption{\small An illustration of a transition of a mass on $\R^2$ into a measure on $S^2$.}	
\end{figure}

\medskip
Now, every mass $\mu$ on the Euclidean space $\R^d$ induces a measure (mass) $\mu'$ on $S(\R^{d+1})$ defined by $\mu'(A):=\mu(p^{-1}(A\cap\Lambda))$, where $A\subseteq S(\R^{d+1})$ is an element of the Borel $\sigma$-algebra on  $S(\R^{d+1})$.
In particular, measure $\mu'$ vanishes on each equatorial sphere of $S(\R^{d+1})$.
Here, an equatorial sphere of $S(\R^{d+1})$ can be always presented as an intersection of $S(\R^{d+1})$ and a unique linear hyperplane in $\R^{d+1}$.

\medskip
Furthermore, every affine hyperplane $H$ in $\R^d$ is mapped via $p$ to a part of an equatorial sphere of $S(\R^{d+1})$.
More precisely, 
\[
p(H)=\spann (H)\cap\Lambda=\{\lambda\cdot(x,-1) \,:\, \lambda\in\R,\,x\in H \}\cap\Lambda,
\] 
where $\spann$ denotes linear span in $\R^{d+1}$.

\medskip
Using the transition of masses on $\R^d$ into measures on $S^d$, and affine hyperplanes in  $\R^d$ into equatorial spheres on $S^d$, we can formulate the GHR problem for masses on sphere as follows:
{\em determine the minimal dimension $d=\Delta_S(j,k)$ of a unit Euclidean sphere $S^d$ in which every collection of $j$ masses can be equiparted by an arrangement of $k$ equatorial spheres.}
Here, the notions of masses and  equipartition of masses are naturally extended from the affine to the spherical setup. 
 
\medskip
Motivated by this spherical extension of the classical problem and with a desire to simplify the treatment of the mass assignments, we restate the GHR problem for mass assignments in the following way.

\medskip
Let $E$ be a Euclidean vector bundle over path-connected space $B$, and let $S(E)$ denote the unit sphere bundle associated to $E$.
Now, {\em we are looking for all pairs of positive integers $(j,k)$ with the property that for every collection of $j$ continuous real valued functions $\varphi_1,\dots,\varphi_j\colon S(E)\longrightarrow\R$, there exists a point $b\in B$ and an arrangement $\HH^b=(H_1^b,\dots,H_k^b)$ of $k$ linear hyperplanes in the fibre $E_b$ of $E$  such that for every pair of connected components $(\mathcal{O}',\mathcal{O}'')$ of the arrangement complement $E_b-(H_1^b\cup\dots\cup H_k^b)$ the following statement holds}
\[
\int_{\mathcal{O}'\cap S(E_b)}\varphi_1=\int_{\mathcal{O}''\cap S(E_b)}\varphi_1
\quad , \dots , \quad 
\int_{\mathcal{O}'\cap S(E_b)}\varphi_j=\int_{\mathcal{O}''\cap S(E_b)}\varphi_j.
\]
Here integration is assumed to be with the respect to the measure on the sphere $S(E_b)$ induced by the metric.
Once again, $\Delta_S(E)$ denotes the set of all such pairs $(j,k)$.
Since Euclidean and spherical partition problems are tightly related, we will not make a particular distinction between them. 
From now on instead of a mass assignment we consider a real valued continuous function from the sphere bundle, and instead of an affine hyperplane we take linear hyperplane which induces an equatorial sphere.

%--------------------------------------------------------------------------------------%
%--------------------------------------------------------------------------------------%
%--------------------------------------------------------------------------------------%
%--------------------------------------------------------------------------------------%
\section{Statements of the main results}
\label{sec : Statements of the main results}
%--------------------------------------------------------------------------------------%
%--------------------------------------------------------------------------------------%
%--------------------------------------------------------------------------------------%
%--------------------------------------------------------------------------------------%

After collecting the first family of results for tautological bundles (Theorem \ref{th : original result reformulated}) it is natural to ask various followup questions:
\begin{compactitem}[\quad --]
\item Why not consider partitions of mass assignments on arbitrary vector bundles instead of only tautological vector bundles?
\item Can we constrain our choice of desired partitions on the given vector bundle by forcing normals of hyperplanes into chosen fixed vector subbundles of the ambient vector bundle?
\item What about partitions with pairwise orthogonal hyperplanes, as it was considered in classical case?	
\item  And finally, how can we fit all these questions into a common framework?
\end{compactitem}
In the following we present multiple answers to the question we just asked. 
Interconnection between the main results of the paper is given in Figure \ref{diagram}. 

\begin{figure}[h]
\includegraphics[scale=0.615]{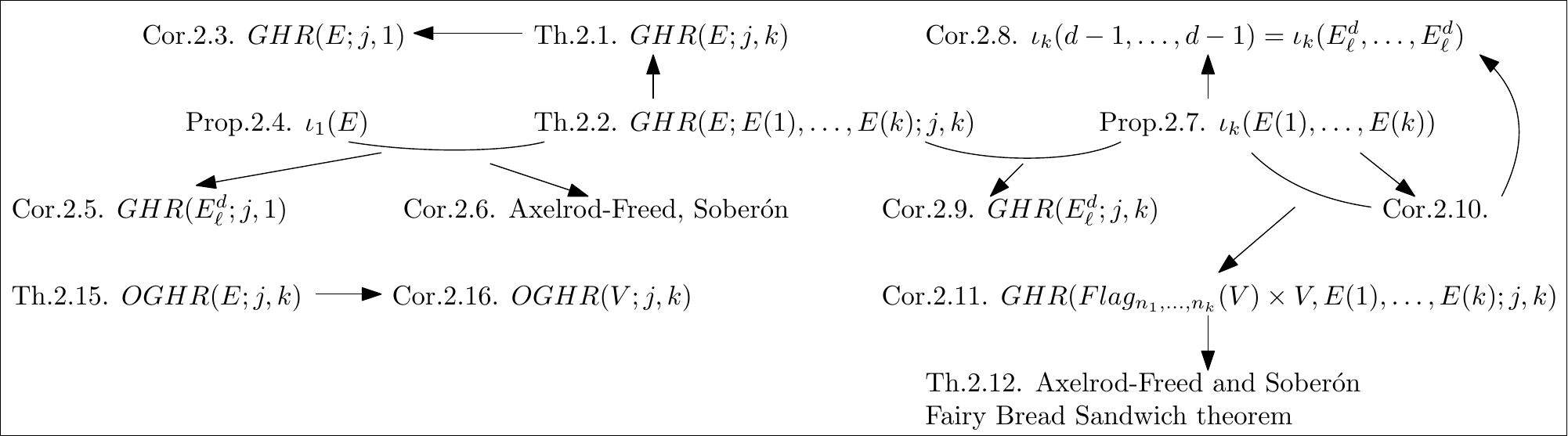}
\caption{\small The main results of the paper and connections between them.} 
\label{diagram}
\end{figure}

\medskip
We begin the list of our results with the full generalisation of \cite[Thm.\,1.1]{BlagojevicCallesCrabbDimitrijevic}.
In other words, the old result becomes a special case of the next theorem in the case of tautological vector bundles. 
For an $n$-dimensional  Euclidean vector bundle $E$ over a compact and connected ENR\footnote{Euclidean Neighbourhoods Retract} $B$ and an integer $k\geq 1$ we denote by 
\[
R_k(B):=H^*(B;\F_2)[x_1,\dots,x_k]
\]
the ring of polynomials in $k$ variables $x_1,\dots,x_k$ of degree $1$ with coefficients in the cohomology ring of the base space $H^*(B;\F_2)$.
Note that by definition an ENR is locally path-connected and so the assumption of connectedness for ENR is equivalent with the assumption of being path-connected.
Classically, we denote by   $w_i(E)$, $i\geq 0$, the Stiefel--Whitney classes of the vector bundle $E$.
In addition, we consider the ideal 
\[
\II_{k}(E):=\Big( \sum_{s = 0}^{n}   w_{n-s}(E)\,x_{r}^{s}  	\: : \:	1 \leq r \leq k \Big) \ \subseteq \ R_{k}(B),
\]
and the element
\[
e_{k}(B):=\prod_{(\alpha_{1},\dots, \alpha_{k})\in\F_2^k-\{ 0\}} (\alpha_{1}x_1 + \cdots + \alpha_{k}x_k) \ \in \ R_{k}(B).
\] 
Now a generalisation of Theorem \ref{th : original result reformulated}, which is proved in Section \ref{subsec : proof of main result 01}, can be stated as follows.

\begin{theorem}
\label{th : main result 01}	
Let $E$ be a Euclidean vector bundle of dimension $n$ over a compact and connected ENR $B$, and let $k\geq 1$ and $j\geq 1$ be integers.

\smallskip\noindent
If the element $e_{k}(B)^j$ does not belong to the ideal $\II_{k}(E)$, then $(j,k)\in \Delta_S(E)$, or in other words, for every collection of $j$ continuous real valued functions $\varphi_1,\dots,\varphi_j\colon\allowbreak S(E)\longrightarrow\R$, there exists a point $b\in B$ and an arrangement $\HH^b=(H_1^b,\dots,H_k^b)$ of $k$ linear hyperplanes in the fibre $E_b$ of $E$ such that for every pair of connected components $(\mathcal{O}',\mathcal{O}'')$ of the arrangement complement $E_b-(H_1^b\cup\dots\cup H_k^b)$ the following statement holds
\[
\int_{\mathcal{O}'\cap S(E_b)}\varphi_1=\int_{\mathcal{O}''\cap S(E_b)}\varphi_1
\quad , \dots , \quad 
\int_{\mathcal{O}'\cap S(E_b)}\varphi_j=\int_{\mathcal{O}''\cap S(E_b)}\varphi_j.
\]
\end{theorem}

\medskip
The first generalisation of Theorem \ref{th : main result 01} is obtained by a restriction of the family of the arrangements in which we are looking for our partition.  
Concretely, we ask for $i$-th hyperplane in the arrangement to have its normal vector in a specific vector subbundle.
For that we modify our setup as follows.

\medskip
Let $k\geq 1$ be an integer, and let  $E(1),\dots,E(k)$ be Euclidean vector bundles over a compact and connected ENR $B$.
Denote by $n_i$ the dimension of the vector bundle $E(i)$ for $1\leq i\leq k$. 
We consider the ideal in $R_{k}(B)$:
\[
\II_k(E(1),\dots,E(k)):=\Big( \sum_{s = 0}^{n_r}   w_{n_r-s}(E(r))\,x_{r}^{s}  	\: : \:	1 \leq r \leq k \Big) \ \subseteq \ R_{k}(B),
\]
and set
\[
\iota_k (E(1),\dots,E(k)):=\max\big\{j :  e_{k}(B)^j\notin \II_k(E(1),\dots,E(k))\big\}.
\]
Finally, we say that an arrangement of $k$ linear hyperplanes $\HH^b=(H_1^b,\dots,H_k^b)$ in the fibre  $E_b$ is determined by the collection of vector subbundles $E(1),\dots,E(k)$ if a unit normal of the linear hyperplane $H_i^b$ belongs to the fibre $E(i)_b$, for every $1\leq i\leq k$.

\medskip
Now, the generalisation, proved in Section \ref{subsec : proof of main result 02}, says the following.

\begin{theorem}
\label{th : main result 02}	
Let $E$ be a Euclidean vector bundle of dimension $n$ over a compact and connected ENR $B$, $k\geq 1$ and $j\geq 1$ integers, and let $E(1),\dots,E(k)$ be vector subbundles of $E$ of dimensions $n_1,\dots,n_k$, respectively.

\smallskip\noindent
If $j\leq \iota_k (E(1),\dots,E(k))$, then for every collection of $j$ continuous real valued functions $\varphi_1,\dots,\varphi_j\colon S(E)\longrightarrow\R$, there exists a point $b\in B$ and an arrangement $\HH^b=(H_1^b,\dots,H_k^b)$ of $k$ linear hyperplanes in fibre $E_b$ of $E$ determined by the collection of vector subbundles $E(1),\dots,E(k)$ such that for every pair of connected components $(\mathcal{O}',\mathcal{O}'')$ of the arrangement complement $E_b-(H_1^b\cup\dots\cup H_k^b)$ the following statement holds
\[
\int_{\mathcal{O}'\cap S(E_b)}\varphi_1=\int_{\mathcal{O}''\cap S(E_b)}\varphi_1
\quad , \dots , \quad 
\int_{\mathcal{O}'\cap S(E_b)}\varphi_j=\int_{\mathcal{O}''\cap S(E_b)}\varphi_j.
\]
\end{theorem}

\medskip
After a generalisation and an extension of \cite[Thm.\,1.1]{BlagojevicCallesCrabbDimitrijevic}, it is natural to ask whether the algebraic criteria from Theorem \ref{th : main result 01} and Theorem  \ref{th : main result 02} can be substituted by appropriate numerical criteria.
In other words, is there an appropriate  generalisation of Theorem \ref{th : original result reformulated} in the case of an arbitrary vector bundle.
We start our discussion from the case $k=1$, the ham-sandwich. 

\medskip
Let $E$ be a Euclidean vector bundle of dimension $n$ over a compact and connected ENR $B$ and let $k=1$. 
Since the ideal 
$
\II_{1}(E)=\Big( \sum_{s = 0}^{n}   w_{n-s}(E)\,x_{1}^{s}  	 \Big)
$
and $e_1(B)^{n-1}=x_1^{n-1}\notin \II_{1}(E)$ we conclude that 
\[\iota_1(E)=\max\big\{j:x_1^j\notin \II_{1}(E)\big\}\geq n-1.\]
The equality $\iota_1(E)=n-1$ is attained in the case when the base space $B$ is a point. 
Indeed, when $B=\pt$ then the vector bundle $E$ is a trivial, $w(E)=1$, and so $\II_{1}(E)=(x_1^n)$ implying that 
$\iota_1(E)=\max\big\{j:x_1^j\notin (x_1^n)\big\}=n-1$.
We just proved the following ham-sandwich type result for Euclidean vector bundles.

\begin{corollary}
Let $E$ be a Euclidean vector bundle of dimension $n$ over a compact and connected ENR $B$.

\smallskip\noindent
If $j=n-1$ then for every collection $\varphi_1,\dots,\varphi_j\colon S(E)\longrightarrow\R$ of $j$ continuous real valued functions, there exists a point $b\in B$ and a hyperplane $H^b$ in $E_b$ such that for the connected components $\mathcal{O}'$ and $\mathcal{O}''$ of the  complement $E_b-H^b$ the following statement holds
\[
\int_{\mathcal{O}'\cap S(E_b)}\varphi_1=\int_{\mathcal{O}''\cap S(E_b)}\varphi_1
\quad , \dots , \quad 
\int_{\mathcal{O}'\cap S(E_b)}\varphi_j=\int_{\mathcal{O}''\cap S(E_b)}\varphi_j.
\]	
\end{corollary}

\medskip
The previous result is general and holds for all vector bundles and therefore rather crude because it must contain the classical ham-sandwich theorem.
It is natural to ask how the topology of the vector bundle $E$ affects the upper bound for the number of functions we can equipart in a fibre.
In other words can we say more about the number $\iota_1(E)$.

\medskip
Indeed, the following proposition, proved in Section \ref{subsec : prop : ham-sandwich}, explains a connection between the topology of $E$ and the number $\iota_1(E)$.

\begin{proposition}
\label{prop : ham-sandwich}
Let $E$ be a Euclidean vector bundle of dimension $n$ over a compact and connected ENR $B$.
Then
\[
 \iota_1(E)=\max\big\{j : 0\neq w_{j-n+1}(-E)\in H^{j-n+1}(B;\F_2)\big\}.
\]	
\end{proposition}

\medskip
Here, $-E$ denotes an inverse vector bundle of $E$.
This is a vector bundle  $E'$, over the base space of $E$, having the property that the Whitney sum $E\oplus E'$ is a trivial vector bundle over $B$.
In particular, the inverse vector bundle is not uniquely defined.
On the other hand the Stiefel--Whitney classes of all inverse vector bundles, of a given vector bundle, do coincide. 
Finally, for example, the compactness of the base space guarantees the existence of an inverse bundle of a given vector bundle. 
(Alternatively, we can take $-E$ to be a virtual bundle representing the negative of the class of $E$ in the Grothendieck $K$-group $KO^0(B)$. 
Precisely, a virtual bundle is a pair $(E_0,E_1)$ of vector bundles over $B$, in an appropriate category, and $-E$ is the pair $(0,E)$. 
The set of isomorphism classes of virtual bundles is precisely the Grothendieck group $KO^0(B)$.
Hence, the dimension $\dim (-E) =- \dim E$, and in the $K$-group $[-E]=-[E] = [E'] - \dim (E\oplus E')$.)

\medskip
As a special case of the previous result we recover the ham-sandwich result for the tautological vector bundle \cite[Cor.\,1.2]{BlagojevicCallesCrabbDimitrijevic}.

\begin{corollary}\label{cor : ham-sandwich}
Let $d\geq 2$ and $\ell\geq 1$ be integers where $1\leq\ell\leq d$, and let $E_{\ell}^{d}$ be the tautological vector bundle over the Grassmann manifold $\Gr_{\ell}(\R^d)$ of all $\ell$-dimensional linear subspaces in $\R^d$.
Then
\[
\iota_1 (E_{\ell}^{d})=d-1.
\]
\end{corollary}
\begin{proof}
For the proof we use the fact that the inverse bundle $-E_{\ell}^{d}$  can be realised as the orthogonal complement vector bundle $(E_{\ell}^{d})^{\perp}$.
In particular, we have that 
\[
w(-E_{\ell}^{d})=1+w_1((E_{\ell}^{d})^{\perp})+\dots+w_{d-\ell}((E_{\ell}^{d})^{\perp}),
\]
where the orthogonal complement is considered inside the trivial vector bundle $\Gr_{\ell}(\R^d)\times\R^d$.
Since $w_{d-\ell}(-E_{\ell}^{d})=w_{d-\ell}((E_{\ell}^{d})^{\perp})\neq 0$ and $w_{i}(-E_{\ell}^{d})=w_{i}((E_{\ell}^{d})^{\perp})=0$ for $i\geq d-\ell+1$, consult for example \cite[p.\,523]{Hiller-1980}, we have from Proposition \ref{prop : ham-sandwich} that 
\[
\iota_1 (E_{\ell}^{d})=\max\big\{j : 0\neq w_{j-\ell+1}(-E_{\ell}^{d})\big\}= d-1.
\]
\end{proof}

\medskip
The following spherical version of the result of Axelrod-Freed and Sober\'on \cite[Thm.\,1.3]{Soberon}, which was previously conjectured by Schnider \cite[Conj.\,2.4]{Schnider-2020}, is a direct consequence of our Theorem \ref{th : main result 02} and Corollary \ref{cor : ham-sandwich}.

\begin{corollary}\label{cor : Axelrod-Freed-Soberon-2}
Let $d\geq 2$ and $\ell\geq 1$ be integers where $1\leq\ell\leq d$, and let $W$ be an arbitrary $(\ell-1)$-dimensional vector subspace of $\R^d$.

\smallskip\noindent
If $j= d-1$ then for any collection of continuous functions $\varphi_1,\dots,\varphi_j\colon S(E^d_l)\longrightarrow\R$, there exists:
\begin{compactenum}[\quad --]
	\item $V\in G_{\ell}(\R^d)$ which contains $W$, and
	\item $U\in G_{\ell-1}(\R^d)$, which is contained in $V$
\end{compactenum}
such that for the connected components $\mathcal{O}'$ and $\mathcal{O}''$ of the  complement $V-U$ the following statement holds
\[
\int_{\mathcal{O}'\cap S(V)}\varphi_1=\int_{\mathcal{O}''\cap S(V)}\varphi_1
\quad , \dots , \quad 
\int_{\mathcal{O}'\cap S(V)}\varphi_j=\int_{\mathcal{O}''\cap S(V)}\varphi_j.
\]		
\end{corollary}
\begin{proof}
Consider the vector bundle $E=E(1)=H(W^{\perp})\oplus\underline{W}$ over $\PP(W^{\perp})$ where $\underline{W}=\PP(W^{\perp})\times W$ is the trivial vector bundle over $\PP(W^{\perp})$.
Recall that here $H(W^{\perp})$ is the canonical Hopf line bundle over the projective space $\PP(W^{\perp})$.
According to Theorem \ref{th : main result 02} in the case $k=1$ we have: if $j\leq \iota_1(E)$, then for any $j$ continuous functions $\varphi_1,\dots,\varphi_j\colon S(E)\longrightarrow\R$ there exists a line $L\in \PP(W^{\perp})$ and a linear hyperplane $U$ in $V:=L\oplus W$ such that for the connected components $\mathcal{O}'$ and $\mathcal{O}''$ of the  complement $V-U$ the following equalities  hold:
\[
\int_{\mathcal{O}'\cap S(V)}\varphi_1=\int_{\mathcal{O}''\cap S(V)}\varphi_1
\quad , \dots , \quad 
\int_{\mathcal{O}'\cap S(V)}\varphi_j=\int_{\mathcal{O}''\cap S(V)}\varphi_j.
\]		

\medskip
Since $w(E)=w(H(W^{\perp})\oplus\underline{W})=w(H(W^{\perp}))$ and $H(W^{\perp})\cong E_1^{d-\ell+1}$ the Corollary \ref{cor : ham-sandwich} implies that $\iota_1(E)=\iota_1(E_1^{d-\ell+1})=d-1$.
With the assumption $j=d-1\leq  \iota_1(E)$ we conclude the proof of the corollary.
\end{proof}

\medskip
Further on, if $E(1)$ is an $n_1$ dimensional vector subbundle of the vector bundle $E$ then
\[
\iota_1(E(1))\leq \iota_1(E).
\]
Indeed, if $E(1)^{\perp}$ is the orthogonal complement vector bundle of $E(1)$ in $E$ then
\begin{multline*}
	x_1^n+w_1(E)x_1^{n-1}+\dots +w_n(E)=\\
	\big(x_1^{n_1}+w_1(E(1))x_1^{n_1-1}+\dots +w_{n_1}(E(1))\big)\\ 
	\big(x_1^{n-n_1}+w_1(E(1)^{\perp})x_1^{n-n_1-1}+\dots +w_{n-n_1}(E(1)^{\perp})\big).
\end{multline*}
Consequently, $x_1^j\notin\II_1(E(1))$ implies $x_1^j\notin\II_1(E)$.

 \medskip
Recall, that
\[
e_k(\pt)=\prod_{(\alpha_{1},\dots, \alpha_{k})\in\F_2^k-\{ 0\}} (\alpha_{1}x_1 + \cdots + \alpha_{k}x_k) \ \in \ R_{k}(\pt)\cong \F_2[x_1,\dots,x_k].
\]
Now, for positive integers $m_1,\dots,m_k$ we define
\[
\iota_k(m_1,\dots,m_k):=\max\big\{ j : e_k(\pt)^j\notin (x_1^{m_1},\dots,x_k^{m_k})\big\}.
\]

\medskip
For example, if $E=\underline{\R^{n}}$ is a trivial $n$ dimensional real vector bundle over $B=\pt$, then
\[
\iota_k(\underline{\R^{n_1}},\dots, \underline{\R^{n_k}})=\iota_k(n_1,\dots,n_k).
\]
Notice that the equality holds for all integers  $n\geq\max\{n_1,\dots,n_k\}$.
Indeed, since $w(\underline{\R^{n_1}})=\dots=w(\underline{\R^{n_k}})=1$, it follows that 
\[(x_1^{n_1},\dots,x_k^{n_k})=\II_k( \underline{\R^{n_1}},\dots, \underline{\R^{n_k}}).\]
 
\medskip
In general, the following inequality always holds
\[
\iota_k(n_1,\dots,n_k)\leq \iota_k(E(1),\dots,E(k)).
\] 
In fact, the condition $e_k(\pt)^j\notin (x_1^{n_1},\dots,x_k^{n_k})$, for some integer $j$, implies the existence of a monomial $x_1^{m_1}\cdots x_k^{m_k}$, in the additive presentation of $e_k(\pt)^j$ with respect to the monomial base of $\F_2[x_1,\dots,x_k]$, with the property that $m_1\leq n_1-1,\dots, m_k\leq n_k-1$.
Since the ideal $\II_k(E(1),\dots,E(k))$ is generated by polynomials $x_1^{n_i}+w_1(E(i))x_1^{n_i-1}+\dots +w_{n_i}(E(i))$, $1\leq i\leq k$, the existence of the monomial $x_1^{m_1}\cdots x_k^{m_k}$ in the presentation of $e_k(\pt)^j$ implies that $e_k(B)^j\notin \II_k(E(1),\dots,E(k))$.

\medskip
Actually, we can say more, as the following proposition illustrates.
For the proof see Section \ref{subsec : prop : for k>1}.

\begin{proposition}
\label{prop : for k >1}
Let $k\geq 1$ be an integer, and let  $E(1),\dots,E(k)$ be Euclidean vector bundles over a compact and connected ENR $B$.
Denote by $n_i$ the dimension of the vector bundle $E(i)$ for $1\leq i\leq k$. 
If
\[
0\neq  w_{\iota_1(E(1))-n_1+1}(-E(1))\cdots  w_{\iota_1(E(k))-n_k+1}(-E(k))\in H^*(B;\F_2),
\]
then 
\[
\iota_k(\iota_1(E(1))+1,\dots,\iota_1(E(k))+1)=\iota_k(E(1),\dots,E(k)).
\]	
\end{proposition}

\medskip
A direct consequence of the previous proposition, in the case when $E$ is a tautological vector bundle, is the following corollary \cite[Lem.\,4.1]{BlagojevicCallesCrabbDimitrijevic}.
For a proof see Section \ref{sub : proof of cor: A}.

\begin{corollary}\label{cor : A}
Let $d\geq 2$, $k\geq 1$, and $\ell\geq 1$ be integers where $1\leq k\leq\ell\leq d$, and let $E_{\ell}^{d}$ be the tautological vector bundle over the Grassmann manifold $\Gr_{\ell}(\R^d)$ of all $\ell$-dimensional linear subspaces in $\R^d$. 	
Then
\[
\iota_k(d,\dots,d)=\iota_k(E_{\ell}^{d},\dots,E_{\ell}^{d}) .
\]	
\end{corollary}

\medskip
The next corollary is a spherical version of \cite[Thm.\,1.4]{BlagojevicCallesCrabbDimitrijevic}. 

\begin{corollary}\label{cor : D}
Let $d\geq 2$, $k\geq 1$, and $\ell\geq 1$ be integers where $1\leq k\leq\ell\leq d$, and let $E=E_{\ell}^{d}$ be the tautological vector bundle over the Grassmann manifold $\Gr_{\ell}(\R^d)$ of all $\ell$-dimensional linear subspaces in $\R^d$. 

\smallskip\noindent
If $j=2^t+r$ where $0\leq r\leq 2^t-1$ and $d\geq 2^{t+k-1}+r$, then $(j,k)\in \Delta_S(E_{\ell}^{d})$.

\smallskip\noindent
In other words, if $j=2^t+r$ where $0\leq r\leq 2^t-1$ and $d\geq 2^{t+k-1}+r$, 
then for every collection of $j$ continuous real valued functions $\varphi_1,\dots,\varphi_j\colon S(E)\longrightarrow\R$, there exists a point $b\in B$ and an arrangement $\HH^b=(H_1^b,\dots,H_k^b)$ of $k$ linear hyperplanes in the fibre $E_b$ of $E$ such that for every pair of connected components $(\mathcal{O}',\mathcal{O}'')$ of the arrangement complement $E_b-(H_1^b\cup\dots\cup H_k^b)$ the following statement holds
\[
\int_{\mathcal{O}'\cap S(E_b)}\varphi_1=\int_{\mathcal{O}''\cap S(E_b)}\varphi_1
\quad , \dots , \quad 
\int_{\mathcal{O}'\cap S(E_b)}\varphi_j=\int_{\mathcal{O}''\cap S(E_b)}\varphi_j.
\]	
 
\end{corollary}
\begin{proof}
From Theorem \ref{th : main result 01}	we have that $(j,k)\in \Delta_S(E_{\ell}^{d})$	if
$
e_{k}(B)^j\notin \II_{k}(E_{\ell}^{d})=\II_{k}(E_{\ell}^{d},\dots, E_{\ell}^{d})
$.
Stated differently $(j,k)\in \Delta_S(E_{\ell}^{d})$	if
\[
	j\leq \iota_k(E_{\ell}^{d},\dots,E_{\ell}^{d})= \iota_k(d,\dots,d)= 
	\max\big\{ j' : e_k(\pt)^{j'}\notin (x_1^{d},\dots,x_k^{d})\big\}.
\]
Here the first equality comes from Corollary \ref{cor : A} while the second one is just the definition of $\iota_k(d,\dots,d)$.

\medskip
Since $j=2^t+r$ where $0\leq r\leq 2^t-1$ and $d\geq 2^{t+k-1}+r$, then according to \cite[Lem.\,4.2]{BlagojevicCallesCrabbDimitrijevic} we have that $e_k(\pt)^{j}\notin (x_1^{d},\dots,x_k^{d})$. 
Thus, indeed $j\leq  \iota_k(E_{\ell}^{d},\dots,E_{\ell}^{d})$ and the proof of the corollary is complete. 
\end{proof}

\medskip
We proceed with the next consequence of Proposition \ref{prop : for k >1}.
In this case the base space of the vector bundle will be the real flag manifold, and so the following statement is an extension of Corollary \ref{cor : A}.
For the relevant background on the real flag manifold, associated canonical vector bundles, and a proof of the corollary see Section \ref{subsec : cor : flag}.

\begin{corollary} \label{cor : B}
	Let $k\geq 1$ and $d\geq 2$ be integers, and let $0=n_0< n_1<\dots<n_{k-1}<n_k<n_{k+1}= d$ be a strictly increasing sequence of integers.
	For a real $d$-dimensional vector space $V=\R^d$ let $E_1,\dots, E_{k+1}$ denote the canonical vector bundles over the flag manifold $\flag_{n_1,\dots,n_k} (V)$, with $\dim(E_i)=n_i-n_{i-1}$ for $1\leq i\leq k+1$.
	Set $E(i):=\bigoplus_{1\leq r\leq i}E_r$ for all $1\leq i\leq k$.
	Then
	\[
	\iota_k(d,\dots,d)=\iota_k(E(1),\dots,E(k)).
	\] 
\end{corollary} 

\medskip
The previous corollary, in the language of GHR problem for the mass assignments, with the help of Theorem \ref{th : main result 02} and the proof of Corollary \ref{cor : D}, gives the following consequence. 
For a proof see Section \ref{subsec : cor : flag 02}.

\begin{corollary}
	\label{cor : E}
	Let $k\geq 1$ and $d\geq 2$ be integers, let $0=n_0< n_1<\dots<n_{k-1}<n_k<n_{k+1}= d$ be a strictly increasing sequence of integers, and let $V=\R^d$ be a real $d$-dimensional vector space.
	Let $E_1,\dots, E_{k+1}$ be canonical vector bundles over the flag manifold $\flag_{n_1,\dots,n_k} (V)$, let $E(i):=\bigoplus_{1\leq r\leq i}E_r$ for all $1\leq i\leq k$, and let $E:=E(k)$.

\smallskip\noindent
Assume that $j=2^t+r$ is an integer with $0\leq r\leq 2^t-1$ and $d\geq 2^{t+k-1}+r$. 
Then for every collection of $j$ continuous real valued functions $\varphi_1,\dots,\varphi_j\colon S(E)\longrightarrow\R$, there exists a point $b:=(W_1,\dots,W_{k+1})\in \flag_{n_1,\dots,n_k} (V)$ and an arrangement  $\HH^b=(H_1^b,\dots,H_k^b)$ of $k$ linear hyperplanes in $E_b=\bigoplus_{1\leq r\leq k}W_r=W_{k+1}^{\perp}$ such that for every pair of connected components $(\mathcal{O}',\mathcal{O}'')$ of the arrangement complement $E_b-(H_1^b\cup\dots\cup H_k^b)$ the following statements hold
\[
\int_{\mathcal{O}'\cap S(E_b)}\varphi_1=\int_{\mathcal{O}''\cap S(E_b)}\varphi_1
\quad , \dots , \quad 
\int_{\mathcal{O}'\cap S(E_b)}\varphi_j=\int_{\mathcal{O}''\cap S(E_b)}\varphi_j,
\]	
and in addition
\[
H_1^b\supseteq \bigoplus_{2\leq r\leq k+1}W_r,\ 
H_2^b\supseteq \bigoplus_{3\leq r\leq k+1}W_r, \ \dots \ ,
H_k^b\supseteq\bigoplus_{k+1\leq r\leq k+1} W_{k+1}.
\]
\end{corollary}

\medskip
Here $(W_1,\dots,W_{k+1})\in \flag_{n_1,\dots,n_k} (V)$ means that $\dim W_i=n_i-n_{i-1}$ for $1\leq i\leq k+1$, and $ W_{i'}\perp W_{i''}$ for all $1\leq i'<i''\leq k+1$. 
For more details on flag manifolds see Section \ref{sec : proof flag}. 

\medskip
It should be noticed that once again the numerical assumptions on the parameters $(d,j,k)$ in Corollary \ref{cor : E} coincide with the upper bound of Mani-Levitska--Vre\'{c}ica--\v{Z}ivaljevi\'c  \cite[Thm.\,39]{ManiLevitska-Vrecica-Zivaljevic-2006} for the function $\Delta$, which can be phrased as the inequality $\Delta(2^t+r)\leq 2^{t+k-1}+r$, for $j=2^t+r$ and $0\leq r\leq 2^t-1$.

 \medskip
We conclude our collection of  results related to flags inside a real vector space  with the spherical version of a result by Axelrod-Freed  and Sober\'on \cite[Thm.\,1.2]{Soberon}.
For the so called Fairy Bread Sandwich theorem we give a new proof in Section \ref{subsec : cor : flag 03} based on the CS\,/\,TM scheme presented in Section \ref{subsec : CSTM for Fairy Bread Sandwich theorem}.

\begin{theorem}
\label{cor : F}
Let $d\geq 1$ and $k\geq 1$ be integers with $d\geq k$, and let $V=\R^{d+1}$ be a real vector space.
Fix a permutation $(j_k,\dots,j_d)$ of the set  $\{k,\dots, d\}$, and take an arbitrary collections of functions $\varphi_{a,b}\colon S(E_{a+1}^{d+1})\longrightarrow\R$, $k\leq a\leq d$, $1\leq b\leq j_a$, from the sphere bundle of the tautological vector bundle  $E_{a+1}^{d+1}$ over the Grassmann manifold $G_{a+1}(V)$ to the real numbers.

\smallskip\noindent
There exists a flag $(V_{k},\dots,V_{d})\in \flag_{k,\dots,d}(V)$ such that for every $k\leq a \leq d$ and every $1\leq b\leq j_a$ the following statement holds
\[
\int_{\{v\in V_{a+1} : \langle v,u_a\rangle\geq 0\}\cap S(V_{a+1})}\varphi_{a,b}
=
\int_{\{v\in V_{a+1} : \langle v,u_a\rangle\leq 0\}\cap S(V_{a+1})}\varphi_{a,b}.
\]	
Here the unit vectors $u_{k}, \dots, u_{d}$ are determined, up to a sign, by the equality $V_{r}=\{v\in V_{r+1} : \langle v,u_{r}\rangle=0\}$, $k\leq r\leq d$, and with $V_{d+1}=V$, this means that $u_{r}$ is a unit normal vector to $V_{r}$, considered as a hyperplane inside $V_{r+1}$.
\end{theorem}

\medskip
Returning back to Proposition  \ref{prop : for k >1} we observe that the numbers $\iota_k(m_1,\dots,m_k)$, in many cases, imply the existence of equipartitions of mass assignments. 
Hence, we collect several properties of these numbers with proofs given in Section \ref{subsec : prop : numbers}.

\begin{proposition}
\label{prop : numbers}
Let $k\geq 1$ be an integer and let $m_1,\dots,m_k$ be a sequence of positive integers.
\begin{compactenum}[\rm \quad (1)]
\item If $\iota_{k-1}(	m_1,\dots,m_{k-1})\geq m$ and $m_k\geq 2^{k-1}m+1$, then $\iota_k(m_1,\dots,m_k)\geq m$.
\item If $m_i\geq 2^{i-1}m+1$ for all $1\leq i\leq k$, then $\iota_k(m_1,\dots,m_k)\geq m$.
\item If $m\geq 1$, then $\iota_k(m+1,2m+1,2^2m+1\dots,2^{k-1}m+1)= m$.
\item Let  $m\geq 1$ and $1\leq r\leq k-1$ be integers. If $\iota_{k-r}(m_1,\dots,m_{k-r})\geq m$ and $\iota_{r}(m_{k-r+1},\dots,m_{k})\geq 2^{k-r}m$,  then $\iota_{k}(m_1,\dots,m_{k})\geq m$.
\item If $\iota_{k-1}(m_1,\dots,m_{k-1})\geq 2m$ and $m_k\geq m+1$, then  $\iota_{k}(m_1,\dots,m_{k})\geq m$.
\item Let $k=2$. $m\leq \iota_2(m_1,m_2)$ if and only if there is an integer $i$ such that $0\leq i\leq m$, ${m \choose i}=1\mod 2$, and $2m-m_2+1\leq i\leq m_1-m-1$.
\item If $1\leq r\leq 2^t$, then $\iota_2(2^t+2r,2^{t+1}+r)\geq 2^t+r-1$.
\end{compactenum}
	
\end{proposition}
 
\medskip
Using the fact that $e_k(\pt)$ is the top Dickson polynomial in variables $x_1,\dots,x_k$ we can prove even more. 
For a proof of the proposition which follows see Section \ref{subsec : prop : numbers2}.

\begin{proposition}
\label{prop : numbers2}
Let $k\geq 1$ be an integer and let $m_1,\dots,m_k$ be positive integers.
\begin{compactenum}[\rm \quad (1)]
\item If $0\leq r\leq 2^t-1$, $\iota_{k-1}(m_1,\dots,m_{k-1})\geq 2^t+2r$ and $m_k\geq 2^{t+k-1}+r+1$, then  $\iota_{k}(m_1,\dots,m_{k})\geq 2^t+r$.
\item If $0\leq r\leq 2^t-1$,  $\iota_{k-1}(m_1,\dots,m_{k-1})\geq 2^{t+1}+r$ and $m_k\geq 2^{t+k-1}+r+1$, then  $\iota_{k}(m_1,\dots,m_{k})\geq 2^t+r$.
\item  If $0\leq r\leq 2^t-1$,  $m_i\geq 2^{t+k-1}+r+1$ for all $1\leq i\leq k$, then  $\iota_{k}(m_1,\dots,m_{k})\geq 2^t+r$.
\item If  $\iota_{k}(m_1,\dots,m_{k})\geq m$, then  $\iota_{k}(2m_1,\dots,2m_{k})\geq 2m$.
\end{compactenum}	
\end{proposition}

\medskip
The statement (3) in the previous proposition is equivalent to \cite[Lem.\,4.2]{BlagojevicCallesCrabbDimitrijevic}.

\medskip
We continue with results on partitions by orthogonal arrangements --- the orthogonal GHR problem for mass assignments.

\medskip
First, let us recall the best known results on the orthogonal GHR problem for masses, or more precisely its generalisation, the so called generalised Makeev problem.
The question was formulated by Blagojevi\'c and Roman Karasev in \cite[Sec.\,1.2]{BlagojevicKarasev2012}.
For integer parameters $j\geq 1$  and $1\leq \ell\leq k$, the minimal dimension $d:=\Delta(j,\ell:k)$, or $d^{\perp}:=\Delta^{\perp}(j,\ell:k)$, of a Euclidean space $\E$ such that for every collection $\M$ of $j$ masses in $\E$ there exists an arrangement of $k$ affine hyperplanes, or pairwise orthogonal $k$ affine hyperplanes in $\E$, with the property that every subarrangement of $\ell$ hyperplanes equiparts $\M$.
In particular, $\Delta(j,\ell:k)=\Delta(j,k:k)$.
Blagojevi\'c and Karasev gave an algebraic constrains on the parameters $j$, $\ell$, $k$ and the dimensions $\Delta(j,\ell:k)$ and $\Delta^{\perp}(j,\ell:k)$, see \cite[Thm.\,2.1]{BlagojevicKarasev2012}.
The state of the art results on the generalised Makeev problem are due to Steven Simon \cite[Thm.\,1.1]{Simon2019} and Andres Mejlia, Simon and Jialin Zhang \cite[Thm.\,1.3 and Thm.\,1.5]{mejia2024generalizedmakeevproblemrevisited}.
For example, Simon in \cite[Thm.\,1.1]{Simon2019} showed that
\begin{center}
\begin{tabular}{ l l l }
$\Delta^{\perp}(2^{q+1},2:2)=3\cdot 2^q+1$, &  & $\Delta^{\perp}(2^{q+1}-1,2:2)=3\cdot 2^q-1$, \\ 
$\Delta^{\perp}(2^{q+2}-2,2:2)=3\cdot 2^{q+1}-2$, &   & $\Delta^{\perp}(1,3:3)=4$. 
\end{tabular}
\end{center}
 
\medskip
Coming back to the mass assignments, let $E$ be a Euclidean vector bundle of dimension $n$ over a compact and connected ENR $B$, and let $k\geq 1$ be an integer.
Recall that we denoted by $R_k(B)$ the cohomology ring $H^*(B;\F_2)[x_1,\dots,x_k]$, and by $e_{k}(B)$ the cohomology class $\prod_{(\alpha_{1},\dots, \alpha_{k})\in\F_2^k-\{ 0\}} (\alpha_{1}x_1 + \cdots + \alpha_{k}x_k)$.
We consider the following ideals in $R_k(B)$
\[
\mathcal{J}_k(E):= (f_1,\dots,f_k)
\qquad\text{and}\qquad
\mathcal{J}_k'(E):= (\overline{f}_1,\dots,\overline{f}_k)
\]
where
\[
f_i:=\sum_{0\leq r_1+\dots+r_i\leq n-i+1}w_{n-i+1-(r_1+\dots+r_i)}(E)\, x_1^{r_1}\cdots x_i^{r_i},
\]
and
\[
\overline{f}_i:=\sum_{0\leq r_1+\dots+r_k\leq n-i+1}w_{n-i+1-(r_1+\dots+r_k)}(E)\, x_1^{r_1}\cdots x_k^{r_k},
\]
for $1\leq i\leq k$.

\medskip
The first result on orthogonal partitions is an  analogue of Theorem \ref{th : main result 01} and Theorem \ref{th : main result 02}. 
For the proof see Section \ref{subsec : proof of main result 03}.

\begin{theorem}
\label{th : main result 03}	
Let $E$ be a Euclidean vector bundle of dimension $n$ over a compact and connected ENR $B$, and let $k\geq 1$ and $j\geq 1$ be integers.
Then the following statements are true:

\begin{compactenum}[\rm \quad (1)]
\item $\mathcal{J}_k(E)=\mathcal{J}_k'(E)$.
\item If the element $e_{k}(B)^j$ does not belong to the ideal $\mathcal{J}_k(E)=\mathcal{J}_k'(E)$, then  for every collection of $j$ continuous real valued functions $\varphi_1,\dots,\varphi_j\colon S(E)\longrightarrow\R$, there exists a point $b\in B$ and an orthogonal arrangement $\HH^b=(H_1^b,\dots,H_k^b)$ of $k$ linear hyperplanes in the fibre $E_b$ of $E$ such that for every pair of connected components $(\mathcal{O}',\mathcal{O}'')$ of the arrangement complement $E_b-(H_1^b\cup\dots\cup H_k^b)$ the following equalities hold
\[
\int_{\mathcal{O}'\cap S(E_b)}\varphi_1=\int_{\mathcal{O}''\cap S(E_b)}\varphi_1
\quad , \dots , \quad 
\int_{\mathcal{O}'\cap S(E_b)}\varphi_j=\int_{\mathcal{O}''\cap S(E_b)}\varphi_j.
\]
\end{compactenum}
\end{theorem}

\medskip
The implication \cite[Thm.\,5.2]{Simon2019} of Simon, which says that 
\[
\Delta(j,l)\leq d \quad  \Longrightarrow\quad  \Delta^{\perp}(d-1;j)\leq d-1,
\]
has an analogue in the mass assignment world. 

\begin{proposition}\label{prop : main result 03}	
Let $E$ be a Euclidean vector bundle of dimension $n$ over a compact and connected ENR $B$, and let $k\geq 1$ and $j\geq 1$ be integers.
Then, if the element $e_k(B)^{j+1}$ does not belong to the ideal $\II_k(E\oplus\underline{\R})$, the  element $e_k(B)^j$ does not belong to the ideal $\mathcal{J}_k(E)$.	
\end{proposition}  

\medskip
In the previous proposition $\underline{\R}$ denotes the trivial line bundle $B\times\R$.
The proof of the statement is postponed to Section \ref{subsec : proof of main result 03}.

\medskip
 In the case  when $B=\pt$ the previous theorem  implies directly the result of Blagojevi\'c \& Roman Karasev \cite[Thm.\,2.1 and Prop.\,3.4]{BlagojevicKarasev2012} with a better description of the set of generators of the relevant ideal.
 
\begin{corollary}
\label{cor : orthogonal 01}
Let $V$ be a Euclidean vector space of dimension $n$, and let $k\geq 1$ and $j\geq 1$ be integers.	
If
\begin{multline*}
e_{k}(\pt):=\prod_{(\alpha_{1},\dots, \alpha_{k})\in\F_2^k-\{ 0\}} (\alpha_{1}x_1 + \cdots + \alpha_{k}x_k)\not\in \\
\Big( \sum_{ r_1+\dots+r_i= n-i+1}  x_1^{r_1}\cdots x_i^{r_i} \ : \ 1\leq i \leq k\Big)=\\
\Big( \sum_{ r_1+\dots+r_k= n-i+1}  x_1^{r_1}\cdots x_k^{r_k} \ : \ 1\leq i \leq k\Big),
\end{multline*}
then  for every collection of $j$ continuous functions $\varphi_1,\dots,\varphi_j\colon S(V)\longrightarrow\R$, there exists  an orthogonal arrangement $\HH=(H_1,\dots,H_k)$ of $k$ linear hyperplanes in $V$ such that for every pair of connected components $(\mathcal{O}',\mathcal{O}'')$ of the arrangement complement $V-(H_1\cup\dots\cup H_k)$ the following statement holds
\[
\int_{\mathcal{O}'\cap S(V)}\varphi_1=\int_{\mathcal{O}''\cap S(V)}\varphi_1
\quad , \dots , \quad 
\int_{\mathcal{O}'\cap S(V)}\varphi_j=\int_{\mathcal{O}''\cap S(V)}\varphi_j.
\]
\end{corollary}

\medskip
In the case of a vector space we collect some numerical results.
For that we denote by
\[
\omega_k(n):=\max\Big\{ j : e_k(\pt)^j\notin \Big( \sum_{ r_1+\dots+r_k= n-i+1}  x_1^{r_1}\cdots x_k^{r_k} \ : \ 1\leq i \leq k\Big)\Big\} .
\]
Using a computer algebra system, like \texttt{Wolfram Mathematica}, we collect some concrete values of $\omega_k(n)$:
\begin{center}
\begin{tabular}{ l | r c c c c c c c c }
	$\omega_k(n)$ &  ${\bf n}$ & $3$ & $4$ & $5$  & $6$ & $7$ & $8$ & $9$ & $10$  \\ \hline
	${\bf k}$ &\\
$2$   &   & $0$ & $1$ &  $2$   &   $2$   & $3$    &  $4$   &  $4$   &    $5$      \\	
$3$ & &  $0$   &  $0$    & $0$ &  $1$   &  $1$   &   $2$  & $2$    &  $3$ \\
$4$ & &  $0$  &  $0$   &   $0$  &  $0$    &   $0$    & $1$   &  $1$   &  $1$ \\
\end{tabular}
\end{center}

\medskip
Using the result of Simon \cite[Thm.\,5.2]{Simon2019} or alternatively our extension, Proposition \ref{prop : main result 03} we get the following corollary.

\begin{corollary}
For all integers $k\geq 1$ and  $n\geq 1$ it holds that
\[
\omega_k(n)\geq \iota_k(n+1,\ldots ,n+1)-1.
\]
\end{corollary}

\medskip
For example, If $0\leq r\leq 2^t-1$, $n+1\geq 2^{t+k-1}+r+1$, then
$\omega_k(n)\geq 2^t+r-1$.

%--------------------------------------------------------------------------------------%
%--------------------------------------------------------------------------------------%
%--------------------------------------------------------------------------------------%
\section{From a partition problem to a topological question:\\ The CS\,/\,TM schemes}
\label{sec : From a partition problem to a topology question}
%--------------------------------------------------------------------------------------%
%--------------------------------------------------------------------------------------%
%--------------------------------------------------------------------------------------%

In this section, based on the work of Crabb \cite{Crabb2020}, we develop an alternative configuration test map scheme (CS\,/\,TM) to the one presented in \cite[Sec.\,2]{BlagojevicCallesCrabbDimitrijevic}.
This will be done in two steps, first for the classical GHR mass partition problem, and then for the mass assignment partition problem.
The new approach allows us a systematic study of mass assignment partition questions even with addition of constrains.

%--------------------------------------------------------------------------------------%
\subsection{The GHR problem for masses}
\label{subsec : CSTM classical}
%--------------------------------------------------------------------------------------%
In this part we reformulate the typical product CS\,/\,TM  scheme for the classical GHR problem.
The reformulation of the scheme naturally gives rise to a convenient CS\,/\,TM scheme for the GHR problem for mass assignments.

\medskip
Let $\E$ be a Euclidean vector space of dimension $d\geq 1$.
The unit sphere of the vector space $\E$ will be denoted by $S(\E):=\{v\in \E : \|v\| =1\}$ and the corresponding real projective space by $\PP(\E)$.
The associated Hopf line bundle is $H(V):=\{(L,v)\in \PP(\E)\times \E : v\in L\}$. 
In particular, $S(\E)\cong S^{d-1}$ is the space of all oriented $1$-dimensional vector subspaces of $V$ and $\PP(\E)\cong\RP{d-1}$ is the space of all  $1$-dimensional vector subspaces of $V$. 
The canonical homeomorphism $\PP(\E)=\Gr_1(\E)\cong \Gr_{d-1}(\E)$, $L\longmapsto L^{\perp}$, identifies the projective space $\PP(\E)$ with the space of all linear hyperplanes in $V$, the Grassmann manifold $\Gr_{d-1}(\E)$.

\medskip
The space of all arrangements  of $k$ linear hyperplanes in $\E$ can be identified with the product space $\PP(\E)^{\times k}=\PP(\E)\times\dots\times \PP(\E)$.
On the other hand, the space of all arrangements of $k$ oriented  linear hyperplanes in $\E$ is the $2^k$-fold covering $S(H(\E))^{\times k}=S(H(\E))\times\cdots\times S(H(\E))$ of $\PP(\E)^{\times k}$, whose total space, in particular, is just the product of spheres $S(\E)^{\times k}=S(\E)\times\dots\times S(\E)$.
In other words, we have a fibre bundle $S(H(\E))^{\times k}\longrightarrow  \PP(\E)^{\times k}$ with a discrete fibre 
\[
\big(S(H(\E))^{\times k}\big)_{(L_1,\dots,L_k)}= S(L_1)\times\dots\times S(L_k)
\]
at $(L_1,\dots,L_k)\in\PP(\E)^{\times k}$.
Here, $S(H(\E))$ denotes the sphere bundle of the Hopf line bundle $H(\E)$ with fibres homeomorphic to a zero dimensional sphere.  

\medskip
We denote by $A_k(\E)$ the  $2^k$-dimensional real vector bundle over $\PP(\E)^{\times k}$ with fibre at $(L_1,\dots,L_k)\in \PP(\E)^{\times k}$ defined to be the vector space  $\map\big(\prod_{i=1}^kS(L_i),\R\big)$  of all maps $\prod_{i=1}^kS(L_i)\longrightarrow \R$.   
Each vector space $\map \big(\prod_{i=1}^kS(L_i),\R\big)$ is equipped with the natural $(\Z/2)^k$-action given by the antipodal actions on the $0$-dimensional spheres $S(L_1),\dots, S(L_k)$.
The vector bundle $A_k(\E)$ is isomorphic to the vector bundle
\[
q_1^*\big(H(\E)\oplus\underline{\R}\big)
\otimes\dots\otimes
q_k^*\big(H(\E)\oplus\underline{\R}\big),
\]
where $q_i\colon \PP(\E)^{\times k}\longrightarrow \PP(\E)$ is the projection on the $i$-th factor, $\underline{\R}$  denotes the trivial line bundle, in this case, over $\PP(\E)$, and $q_i^*\big(H(\E)\oplus\underline{\R}\big)$ is the pullback vector bundle.
In particular, the vector bundle
\[
A_k(\E)\cong 
q_1^*\big(H(\E)\oplus\underline{\R}\big)
\otimes\dots\otimes
q_k^*\big(H(\E)\oplus\underline{\R}\big)
\]
has a trivial line subbundle given by all constant maps $\prod_{i=1}^kS(L_i)\longrightarrow \R$, which we also denote by $\underline{\R}$. 

\medskip
Next, let us consider a continuous function $\varphi\colon S(\E)\longrightarrow\R$ on the sphere $S(\E)$.
It induces a section $s_{\varphi}\colon\PP(\E)^{\times k}\longrightarrow A_k(\E)$ of the vector bundle $A_k(\E)$ which is given by 
\[
\xymatrix{
(L_1,\dots,L_k) \ar@{|->}[r]  & \ \big( s_{\varphi} (L_1,\dots,L_k)\colon \prod_{i=1}^kS(L_i) \longrightarrow \R  \big)
}
\]
for $(L_1,\dots,L_k)\in\PP(V)$, where
\[
s_{\varphi}  (L_1,\dots,L_k)(v_1,\dots,v_k):=\int_{\mathcal{O}_{v_1,\dots,v_k}\cap S(\E)} \varphi
\]
for $(v_1,\dots,v_k)\in \prod_{i=1}^kS(L_i)$. 
Here, $\mathcal{O}_{v_1,\dots,v_k}$ denotes the following intersection of open half-spaces in $\E$:
\[
\mathcal{O}_{v_1,\cdots,v_k}:=\{u\in \E\,:\, \langle u, v_1\rangle> 0\}\cap\dots\cap\{u\in \E\,:\, \langle u, v_k\rangle> 0\}.
\]
Here the integration is with the respect to the measure on the sphere $S(\E)$ induced by the metric.
Observe that each subset $\mathcal{O}_{v_1,\cdots,v_k}$ is actually a (path) connected component of the arrangement complement $\E-(L_1^{\perp}\cup\dots\cup L_k^{\perp})$.

\medskip
We have introduced all necessary notions to state and prove the CS\,/\,TM scheme theorem for the spherical version of the classical GHR problem.
This theorem relates to the similar results in    \cite[Prop.\,6]{ManiLevitska-Vrecica-Zivaljevic-2006}, \cite[Prop.\,2.2]{Blagojevic-Ziegler-2011}, \cite[Prop.\,2.1]{Blagojevic-Frick-Haase-Ziegler-2018}.

\begin{theorem}
\label{th : CSTM for classical problem}	
Let $\E$ be a Euclidean vector space, and let $k\geq 1$ and $j\geq 1$ be integers. 
If the Euler class of the vector bundle $\big(A_k(\E)/\underline{\R}\big)^{\oplus j}$ does not vanish, then for every collection of $j$ continuous functions $\varphi_1,\dots,\varphi_j\colon S(\E)\longrightarrow\R$ there exists an arrangement of $k$ linear hyperplanes $H_1, \dots, H_k$ in $\E$ with the property that for every pair of connected components $(\mathcal{O}',\mathcal{O}'')$ of the arrangement complement $\E-(H_1\cup\dots\cup H_k)$ the following statement holds
\[
\int_{\mathcal{O}'\cap S(\E)}\varphi_1=\int_{\mathcal{O}''\cap S(\E)}\varphi_1
\quad , \dots , \quad 
\int_{\mathcal{O}'\cap S(\E)}\varphi_j=\int_{\mathcal{O}''\cap S(\E)}\varphi_j.
\]
In other words,
\[
 \e \big(\big( A_k(\E)/\underline{\R}\big)^{\oplus j} \big)\neq 0
\quad\Longrightarrow\quad \Delta_S(j,k)\leq \dim(\E).
\]
\end{theorem}
\begin{proof}
Let us assume that the Euler class the vector bundle $\big(A_k(\E)/\underline{\R}\big)^{\oplus j}$ does not vanish.
Then, in particular, every section of the vector bundle $\big( A_k(\E)/\underline{\R}\big)^{\oplus j}$ has a zero.

\medskip
Let $\varphi_1,\dots,\varphi_j\colon S(\E)\longrightarrow\R$ be an arbitrary collection of $j$ continuous functions on the sphere $S(\E)$.
Such a collection induces a section $s\colon \PP(\E)^{\times k}\longrightarrow A_k(\E)^{\oplus j}$ of the vector bundle $A_k(\E)^{\oplus j}$ defined by
\[
\xymatrix{
(L_1,\dots,L_k) \ar@{|->}[r]  & \ \big( s_{\varphi_r} (L_1,\dots,L_k)\colon \prod_{i=1}^kS(L_i) \longrightarrow \R  \big)_{1\leq r\leq j}
}.
\]
Recall that we have already defined functions $s_{\varphi_r}$, for $1\leq r\leq j$,  by 
\[
s_{\varphi_r} (L_1,\dots,L_k)(v_1,\dots,v_k)=\int_{\mathcal{O}_{v_1,\dots,v_k}\cap S(\E)} \varphi_r
\]
for $(v_1,\dots,v_k)\in \prod_{i=1}^kS(L_i)$.

\medskip
Let $\Pi\colon A_k(\E)^{\oplus j}\longrightarrow \big( A_k(\E)/\underline{\R}\big)^{\oplus j}$ denote the map of vector bundles induced by the canonical projection(s).
Then the section $\Pi\circ s$ of the vector bundle $\big(A_k(\E)/\underline{\R}\big)^{\oplus j}$ has a zero.
Hence, there is a point $(L_1,\dots,L_k)\in \PP(\E)^{\times k}$ in the base space with the property that $s(L_1,\dots,L_k)$ belongs to the trivial subbundle $\underline{\R}^{\oplus j}$ of the bundle $A_k(\E)^{\oplus j}$.
In other words 
\[
\int_{\mathcal{O}'\cap S(\E)}\varphi_1=\int_{\mathcal{O}''\cap S(\E)}\varphi_1
\quad , \dots , \quad 
\int_{\mathcal{O}'\cap S(\E)}\varphi_j=\int_{\mathcal{O}''\cap S(\E)}\varphi_j.
\]
for all pairs of the connected components $(\mathcal{O}',\mathcal{O}'')$ of the arrangement complement $\E-(L_1^{\perp}\cup\dots\cup L_k^{\perp})$.
This completes the proof of the theorem.
\end{proof}

\medskip
The non-vanishing of the Euler class $\big( A_k(\E)/\underline{\R}\big)^{\oplus j}$ mod $2$ was studied over the years by many authors.
For example,  Mani-Levitska, Vre\'{c}ica \& \v{Z}ivaljevi\'c \cite[Thm.\,39]{ManiLevitska-Vrecica-Zivaljevic-2006} gave a sufficient condition for the non-vanishing of the mod $2$ Euler class of $\big( A_k(\E)/\underline{\R}\big)^{\oplus j}$, with a complete proof of this result given only now in \cite[Lem.\,4.3]{BlagojevicCallesCrabbDimitrijevic}.
It says that: {\em If $ \dim(\E)\ \le\ j + (2^{k-1}-1)2^{\lfloor\log_2j\rfloor}$, then the top Stiefel--Whitney class of the vector bundle $\big( A_k(\E)/\underline{\R}\big)^{\oplus j}$ does not vanish.}

\medskip
Now we focus our attention to the partition problems of mass assignments and the corresponding solution schemes.

%--------------------------------------------------------------------------------------%
\subsection{The GHR problem for mass assignments}
\label{subsec : CSTM for mass assignements}
%--------------------------------------------------------------------------------------%
The scheme we give in this section is derived from the scheme for the spherical version of the classical problem presented in Section \ref{subsec : CSTM classical}. 
Due to a transition from a Euclidean space to a sphere, the new scheme differs from the one used in \cite[Sec.\,2]{BlagojevicCallesCrabbDimitrijevic}.

\medskip
Let $E$ be a Euclidean vector bundle over a compact and connected ENR base space $B$.
The associated unit sphere bundle of $E$ is 
\[
S(E)=\{(b,v) : b\in B,\, v\in S(E_b)\}.
\]
Next, let $\PP(E)$ denote the projective bundle of $E$, that is
\[
\PP(E)=\{(b,L) : b\in B,\, L\in \PP(E_b)\}.
\]
In particular, $S(E)/(\Z/2)\cong \PP(E)$.
Here the fibrewise antipodal action of the sphere bundle is assumed.
Further on, let $H(E)$ be the Hopf bundle associated to the vector bundle $E$.
That is the line bundle  
\[
H(E):=\{(b,L,v) : b\in B,\, L\in \PP(E_b),\,v\in L\}
\]
over the projective bundle $\PP(E)$.

\medskip
The space of all arrangements  of $k$ linear hyperplanes which belong to one fibre of $E$ is the total space of the pullback 
\[
\PP(E)\times_{B}\cdots\times_{B}\PP(E):=d^*(\PP(E)\times\dots\times \PP(E))=d^*(\PP(E)^{\times k})
\] 
of the product vector bundle $\PP(E)^{\times k}$ via the diagonal embedding $d\colon B\longrightarrow B^{\times k}$, $x\longmapsto(x,\dots, x)$.
In other words, there is a pullback diagram
\[
\xymatrix{d^*(\PP(E)^{\times k})\ar[r]^-{D} \ar[d] & \PP(E)^{\times k}\ar[d] \\
B\ar[r]^-{d}&B^{\times k}.
}
\]

\medskip
Let us denote by  $\Pi_i\colon \PP(E)^{\times k}\longrightarrow \PP(E)$, $(b,(L_1,\dots,L_k))\longmapsto (b,L_i)$, the projection on the $i$-th factor, and by $\Theta_i$ the composition $\Pi_i\circ D\colon d^*(\PP(E)^{\times k})\longrightarrow \PP(E)$, where $1\leq i\leq k$.

\medskip
Now, the space of all arrangements of $k$ oriented linear hyperplanes which belong to one fibre of $E$ is the total space of the pullback 
\[
S(E)\times_{B}\cdots\times_{B}S(E):=
d^*(S(E)\times\dots\times S(E))=
d^*(S(E)^{\times k}).
\] 
The quotient map $d^*(S(E)^{\times k})\longrightarrow d^*(\PP(E)^{\times k})$, induced by taking orbits of the natural fibrewise free action of $(\Z/2)^k$ on $d^*(S(E)^{\times k})$, is  a $2^k$-fold cover map with a  fibre $S(L_1)\times\dots\times S(L_k)$ at $(L_1,\dots , L_k)\in\PP(E_b)^{\times k}$ for some $b\in\B$.
Recall that each sphere $S(L_1),\dots,S(L_k)$ is just a $0$-dimensional sphere.

\medskip
Like in the classical case, the covering $d^*(S(E)^{\times k})\longrightarrow d^*(\PP(E)^{\times k})$ induces a $2^k$-dimensional real vector bundle $A_k(E)$ over $d^*(\PP(E)^{\times k})$ with fibre at $(L_1,\dots,L_k)\in \PP(E_b)^k$, for some $b\in B$, defined to be the vector space  $\map \big(\prod_{i=1}^kS(L_i),\R\big)$ of all real valued functions on $\prod_{i=1}^kS(L_i)$.
Each fibre is equipped with the natural $(\Z/2)^k$-action given by antipodal actions on the $0$-dimensional spheres.
There is an isomorphism of vector bundles
\[
A_k(E)\cong
\Theta _1^*\big(H(E)\oplus\underline{\R}\big)
\otimes\dots\otimes
\Theta_k^*\big(H(E)\oplus\underline{\R}\big),
\]
where $\underline{\R}$  denotes the trivial line bundle over $\PP(E)$, and $\Theta_i^*\big(H(E)\oplus\underline{\R}\big)$ is the pullback vector bundle.
In particular, the vector bundle $A_k(E)$ has a trivial line bundle determined by all constant maps $\prod_{i=1}^kS(L_i)\longrightarrow \R$.

\medskip
Let us now consider a continuous function $\varphi\colon S(E)\longrightarrow\R$.
Such a map induces a section $s_{\varphi}\colon d^*(\PP(E)^{\times k})\longrightarrow A_k(E)$ of the vector bundle $A_k(E)$ by
\[
\xymatrix{
(b,(L_1,\dots,L_k)) \ar@{|->}[r]  & \ \big( s_{\varphi} (b,(L_1,\dots,L_k))\colon \prod_{i=1}^kS(L_i) \longrightarrow \R  \big)
}
\]
for  $b\in B$ and $(L_1,\dots,L_k)\in \PP(E_b)^{\times k}$, where
\[
s_{\varphi} (b,(L_1,\dots,L_k))(v_1,\dots,v_k):=\int_{\mathcal{O}_{b,v_1,\dots,v_k}\cap S(E_b)} \varphi
\]
for $(v_1,\dots,v_k)\in \prod_{i=1}^kS(L_i)$. 
Here, $\mathcal{O}_{b,v_1,\dots,v_k}$ denotes the subset of $E_b$ defined by
\[
\mathcal{O}_{b,v_1,\cdots,v_k}:=\{u\in E_b\,:\, \langle u, v_1\rangle> 0\}\cap\dots\cap\{u\in E_b\,:\, \langle u, v_k\rangle> 0\}.
\]
Once again, the integration is assumed to be with respect to the measure of the sphere $S(E_b)$ induced by the metric on $E_b$.

\medskip
Now we can state the CS\,/\,TM scheme theorem for the GHR problem for mass assignments, which is analogous to Theorem \ref{th : CSTM for classical problem}.

\begin{theorem}
\label{th : CSTM for assignement problem}	
Let $E$ be a Euclidean vector bundle over a compact and connected ENR base space $B$, and let $k\geq 1$ and $j\geq 1$ be integers. 

\smallskip\noindent
If the Euler class of the vector bundle $\big( A_k(E)/\underline{\R}\big)^{\oplus j}$ does not vanish, then for every collection of $j$ continuous functions $\varphi_1,\dots,\varphi_j\colon S(E)\longrightarrow\R$ there exists a point $b\in B$ and an arrangement of $k$ linear hyperplanes $H_1, \dots, H_k$ in the fibre $E_b$ with the property that for every pair of connected components $(\mathcal{O}',\mathcal{O}'')$ of the arrangement complement $E_b-(H_1\cup\dots\cup H_k)$ the following statement holds
\[
\int_{\mathcal{O}'\cap S(E_b)}\varphi_1=\int_{\mathcal{O}''\cap S(E_b)}\varphi_1
\quad , \dots , \quad 
\int_{\mathcal{O}'\cap S(E_b)}\varphi_j=\int_{\mathcal{O}''\cap S(E_b)}\varphi_j.
\]
In other words,
\[
 \e \big(\big( A_k(E)/\underline{\R}\big)^{\oplus j} \big)\neq 0
\quad\Longrightarrow\quad (j,k)\in \Delta_S(E).
\]
\end{theorem}
\begin{proof}
Our follows in the footsteps of the proof of Theorem \ref{th : CSTM for classical problem}.
Assume that the Euler class of the vector bundle $\big( A_k(E)/\underline{\R}\big)^{\oplus j}$ does not vanish.
Consequently, every section of $\big( A_k(E)/\underline{\R}\big)^{\oplus j}$ has a zero.

\medskip
Consider a collection $\varphi_1,\dots,\varphi_j\colon S(E)\longrightarrow\R$ of continuous functions on the  sphere bundle $S(E)$, and the associated section $s=(s_{\varphi_1},\dots,s_{\varphi_j})$ of the vector bundle $\big( A_k(E)/\underline{\R}\big)^{\oplus j}$.

\medskip
Denote by $\Pi\colon A_k(E)^{\oplus j}\longrightarrow \big( A_k(E)/\underline{\R}\big)^{\oplus j}$ the canonical projection.
Then, from the assumption on the Euler class, the section $\Pi\circ s$ of the vectors bundle $\big( A_k(E)/\underline{\R}\big)^{\oplus j}$ has a zero.
In other words, there exists a point $(b,(L_1,\dots,L_k))\in d^*(\PP(E)^{\times k})$ with the property that $s(b,(L_1,\dots,L_k))$ is contained in the trivial vector subbundle $\underline{\R}^{\oplus j}$ of the vector bundle $A_k(E)^{\oplus j}$.
This means that for every pair of connected components $(\mathcal{O}',\mathcal{O}'')$ of the arrangement complement $E_b-(L_1^{\perp} \cup\dots\cup L_k^{\perp})$ the following equalities hold
\[
\int_{\mathcal{O}'\cap S(E_b)}\varphi_1=\int_{\mathcal{O}''\cap S(E_b)}\varphi_1
\quad , \dots , \quad 
\int_{\mathcal{O}'\cap S(E_b)}\varphi_j=\int_{\mathcal{O}''\cap S(E_b)}\varphi_j.
\]
Hence, we have proved the theorem.
\end{proof}

%--------------------------------------------------------------------------------------%
\subsection{The GHR problem for mass assignments plus constraints}
\label{subsec : CSTM for mass assignements + constraints}
%--------------------------------------------------------------------------------------%
In this section we extend the CS\,/\,TM schemes presented in  Section \ref{subsec : CSTM for mass assignements} to incorporate an additional constraint.
More precisely, we require for the normals of the hyperplanes to belong to the specific, not necessarily equal, vector subbundles.

\medskip
Fix an integer $k\geq 1$.
Let $E$ be an $n$-dimensional Euclidean vector bundle over a compact and connected ENR $B$, and let  $E(i)$ be a vector subbundles of $E$,
for $1\leq i\leq k$. 
Following the notation from  Section \ref{subsec : CSTM for mass assignements} we denote by $\PP(E(i))$ the projective bundle of $E(i)$, that is
\[
\PP(E(i))=\{(b;L) : b\in B,\, L\in \PP(E(i)_b)\}.
\]
In particular, $S(E(i))/(\Z/2)\cong P(E(i))$.
Furthermore, let $H(E(i))$ be the Hopf bundle associated to the vector bundle $E(i)$, or in other words
\[
H(E(i)):=\{(b,L,v) : b\in B,\, L\in \PP(E(i)_b),\,v\in L\}.
\]

\medskip
The space of all arrangements  of $k$ linear hyperplanes which belong to one fibre of $E$ and are determined by the collection of vector subbundles $E(1),\dots,E(k)$ can be seen as the total space of the pullback vector bundle
\[
\PP(E(1))\times_{B}\cdots\times_{B}\PP(E(k)):=d^*\big(\PP(E(1))\times\dots\times\PP(E(k))\big) 
\] 
via the diagonal embedding $d\colon B\longrightarrow B^k$.
We denote by
\[
D\colon d^*\big(\PP(E(1))\times\dots\times\PP(E(k))\longrightarrow \PP(E(1))\times\dots\times\PP(E(k))
\] 
the pullback map between the bundles.
Furthermore, let 
\[
\Pi_i\colon\PP(E(1))\times\dots\times\PP(E(k))\longrightarrow \PP(E(i))
\]
be the projection on the $i$-th factor $(b,(L_1,\dots,L_k))\longmapsto (b,L_i)$, and let $\Theta_i:=\Pi_i\circ D$. 
 
\medskip
The space of all arrangements of $k$ oriented linear hyperplanes which belong to one fibre of $E$ and are given by the collection of vector subbundles $E(1),\dots,E(k)$ is the total space of the pullback 
\[
S(E(1))\times_{B}\cdots\times_{B}S(E(k)):=d^*\big(S(E(1))\times\dots\times S(E(k))\big).
\] 
The quotient map 
\[
d^*\big(S(E(1))\times\dots\times S(E(k))\big) 
\longrightarrow
d^*\big(\PP(E(1))\times\dots\times \PP(E(k))\big),
\]
induced by taking orbits of the natural fibrewise free action of the group $(\Z/2)^k$, is a $2^k$-fold cover map with a typical fibre $S(L_1)\times\dots\times S(L_k)$ where $(L_1,\dots , L_k)\in\PP(E(1)_b)\times\dots\times\PP(E(k)_b)$ for some $b\in\B$.

\medskip
This covering induces a $2^k$-dimensional real vector bundle  $A_k(E(1),\dots, E(k))$ over $d^*\big(\PP(E(1))\times\dots\times \PP(E(k))\big)$  with fibre at $(L_1,\dots,L_k)\in \PP(E(1)_b)\times\dots\times\PP(E(k)_b)$, for some $b\in B$, defined to be the vector space $\map \big(\prod_{i=1}^kS(L_i),\R\big)$ of all real valued functions on $\prod_{i=1}^kS(L_i)$.
There is an isomorphism of vector bundles
\[
A_k(E(1),\dots, E(k))
\cong
\Theta _1^*\big(H(E(1))\oplus\underline{\R(1)}\big)
\otimes\dots\otimes
\Theta_k^*\big(H(E(k))\oplus\underline{\R(k)}\big),
\]
where $\underline{\R(i)}$  denotes the trivial line bundle over $\PP(E(i))$, and $\Theta_i^*\big(H(E(i))\oplus\underline{\R(i)}\big)$ is the pullback vector bundle.
In particular, the vector bundle $A_k(E(1),\dots, E(k))$ has a trivial line bundle determined by all constant maps $\prod_{i=1}^kS(L_i)\longrightarrow \R$, or in other words the vector subbundle $\underline{\R(1)}\otimes\dots\otimes\underline{\R(k)}$.
Clearly, $A_k(E)=A_k(\underbrace{ E,\dots, E}_{k\text{ times}})$.

\medskip
Now we consider a continuous function $\varphi\colon S(E)\longrightarrow\R$.
It induces a section $s_{\varphi}\colon d^*\big(\PP(E(1))\times\dots\times \PP(E(k))\big)\longrightarrow A_k(E(1),\dots, E(k))$ of the vector bundle $A_k(E(1),\dots, E(k))$ by
\[
\xymatrix{
(b,(L_1,\dots,L_k)) \ar@{|->}[r]  & \ \big( s_{\varphi} (b,(L_1,\dots,L_k))\colon \prod_{i=1}^kS(L_i) \longrightarrow \R  \big)
}
\]
for  $b\in B$ and $(L_1,\dots,L_k)\in \PP(E(1)_b)\times\dots\times \PP(E(k)_b)$, where
\[
s_{\varphi} (b,(L_1,\dots,L_k))(v_1,\dots,v_k):=\int_{\mathcal{O}_{b,v_1,\dots,v_k}\cap S(E_b)} \varphi
\]
for $(v_1,\dots,v_k)\in \prod_{i=1}^kS(L_i)$. 
Recall, $\mathcal{O}_{b,v_1,\dots,v_k}$ denotes the set:
\[
\mathcal{O}_{b,v_1,\cdots,v_k}:=\{u\in E_b\,:\, \langle u, v_1\rangle> 0\}\cap\dots\cap\{u\in E_b\,:\, \langle u, v_k\rangle> 0\}.
\]

\medskip
The CS\,/\,TM scheme theorem for the GHR problem for mass assignments with constraints is as follows.
\begin{theorem}
\label{th : CSTM for assignement problem with constraints}	
Let $E$ be a Euclidean vector bundle over a compact and connected ENR base space $B$, $k\geq 1$ and $j\geq 1$ integers, and let $E(1),\dots,E(k)$ be vector subbundles of $E$.

\smallskip\noindent
If the Euler class of the vector bundle $\big( A_k(E(1),\dots, E(k))/\underline{\R}\big)^{\oplus j}$ does not vanish, then for every collection of $j$ continuous functions $\varphi_1,\dots,\varphi_j\colon S(E)\longrightarrow\R$ there exists a point $b\in B$ and an arrangement of $k$ linear hyperplanes $H_1, \dots, H_k$ in the fibre $E_b$ determined by the collection of vector subbundles $E(1),\dots,E(k)$ with the property that for every pair of connected components $(\mathcal{O}',\mathcal{O}'')$ of the arrangement complement $E_b-(H_1\cup\dots\cup H_k)$ the following statement holds
\[
\int_{\mathcal{O}'\cap S(E_b)}\varphi_1=\int_{\mathcal{O}''\cap S(E_b)}\varphi_1
\quad , \dots , \quad 
\int_{\mathcal{O}'\cap S(E_b)}\varphi_j=\int_{\mathcal{O}''\cap S(E_b)}\varphi_j.
\]

\end{theorem}
\begin{proof}
A proof is a slight modification of the proof of Theorem \ref{th : CSTM for assignement problem}, so we do not repeat it.

\end{proof}

%--------------------------------------------------------------------------------------%
\subsection{The orthogonal GHR problem for mass assignments}
\label{subsec : CSTM for mass assignements orthogonal}
%--------------------------------------------------------------------------------------%

The scheme for the partitions with orthogonal arrangements is just a ``restriction'' of the scheme presented in Section \ref{subsec : CSTM for mass assignements}.

\medskip
For a Euclidean vector bundle over a compact and connected ENR base space $B$, and integers $k\geq 1$ and $j\geq 1$, we proved the following: 

\smallskip\noindent
If the Euler class of the vector bundle $\big( A_k(E)/\underline{\R}\big)^{\oplus j}$ over $d^*(\PP(E)^{\times k})$ does not vanish, then for every collection of $j$ continuous functions $\varphi_1,\dots,\varphi_j\colon S(E)\longrightarrow\R$ there exists a point $b\in B$ and an arrangement of $k$ linear hyperplanes $H_1, \dots, H_k$ in the fibre $E_b$ with the property that for every pair of connected components $(\mathcal{O}',\mathcal{O}'')$ of the arrangement complement $E_b-(H_1\cup\dots\cup H_k)$ holds
\[
\int_{\mathcal{O}'\cap S(E_b)}\varphi_1=\int_{\mathcal{O}''\cap S(E_b)}\varphi_1
\quad , \dots , \quad 
\int_{\mathcal{O}'\cap S(E_b)}\varphi_j=\int_{\mathcal{O}''\cap S(E_b)}\varphi_j.
\]

\medskip
Since we are interested in partitions by specifically orthogonal arrangements the space of all possible solutions becomes the following subspace of $X_k(E):=d^*(\PP(E)^{\times k})$:
\[
Y_k(E):=\{ (b,(L_1,\dots,L_k))\in X_k(E) : L_r \perp L_s \text{ for all }1\leq r< s\leq k\}.
\]
In addition, let us denote by $q_k$ the inclusion  $Y_k(E)\lhook\joinrel\longrightarrow X_k(E)$.
Thus, the vector bundle we are interested in is the restriction bundle  $B_k(E):=A_k(E)|_{Y_k(E)}$.
In particular, there is an isomorphism of vector bundles
\[
B_k(E)\cong
\Psi _1^*\big(H(E)\oplus\underline{\R}\big)
\otimes\dots\otimes
\Psi_k^*\big(H(E)\oplus\underline{\R}\big),
\]
where $\Psi_i=\Theta_i\circ q_k$ for $1\leq i\leq k$.
Recall $H(E)$ and $\underline{\R}$ are here the Hopf line and trivial line bundle over $\PP(E)$, respectively.

\medskip
Now, we get the CS\,/\,TM scheme theorem for the GHR problem for mass assignments by orthogonal arrangements directly from the proof of Theorem \ref{th : CSTM for assignement problem}.

\begin{theorem}
\label{th : CSTM for orthogonal assignement problem}	
Let $E$ be a Euclidean vector bundle over a compact and connected ENR base space $B$, and let $k\geq 1$ and $j\geq 1$ be integers. 

\smallskip\noindent
If the Euler class of the vector bundle $\big( B_k(E)/\underline{\R}\big)^{\oplus j}$ does not vanish, then for every collection of $j$ continuous functions $\varphi_1,\dots,\varphi_j\colon S(E)\longrightarrow\R$ there exists a point $b\in B$ and an orthogonal arrangement of $k$ linear hyperplanes $H_1, \dots, H_k$ in the fibre $E_b$ with the property that for every pair of connected components $(\mathcal{O}',\mathcal{O}'')$ of the arrangement complement $E_b-(H_1\cup\dots\cup H_k)$ the following statement holds
\[
\int_{\mathcal{O}'\cap S(E_b)}\varphi_1=\int_{\mathcal{O}''\cap S(E_b)}\varphi_1
\quad , \dots , \quad 
\int_{\mathcal{O}'\cap S(E_b)}\varphi_j=\int_{\mathcal{O}''\cap S(E_b)}\varphi_j.
\]
\end{theorem}

\medskip
The proof of this result is a copy of the proof of Theorem \ref{th : CSTM for assignement problem} with $Y_k(E)$ in place of $X_k(E)$ and the vector bundle $B_k(E)$ in place of the vector bundle $A_k(E)$.

%--------------------------------------------------------------------------------------%
\subsection{The Fairy Bread Sandwich theorem}
\label{subsec : CSTM for Fairy Bread Sandwich theorem}
%--------------------------------------------------------------------------------------%
Fix integers $d\geq 1$ and $k\geq 1$ with $d\geq k$, and  let $V=\R^{d+1}$.
Let $(j_k,\dots,j_d)$ be a permutation of the set  $\{k,\dots, d\}$, and let $\varphi_{a,b}\colon S(E_{a+1}^{d+1})\longrightarrow\R$, $k\leq a\leq d$, $1\leq b\leq j_a$, be a collection of functions from the sphere bundle of the tautological vector bundle  $E_{a+1}^{d+1}$ over the Grassmann manifold $G_{a+1}(V)$ to the real numbers.
 
\medskip
The space of all potential solutions of the partition problem considered in Theorem \ref{cor : F} is the following flag manifold
\begin{align*}
	\flag_{k,\dots,d} (V) &=\big\{(V_{k},\dots,V_{d})\in \prod_{i=k}^{d} G_{i}(V) :  0\subseteq V_{k}\subseteq \dots\subseteq V_{d}\subseteq V\big\}\\	
	 &\cong\big\{(W_{k},\dots,W_{d+1})\in G_{k}(V)\times  G_{1}(V)^{d-k+1} : \\
	 &\hspace{3.5cm}  W_{i'}\perp W_{i''}   \text{ for all } k\leq i'<i''\leq d+1\big\}.
\end{align*}
We used the homeomorphism between these two presentations
\begin{equation}
\label{eq : identification}	
(W_{k},\dots,W_{d+1}) \longmapsto \big(W_{k}, (W_{k}\oplus W_{k+1}),\dots,( W_{k}\oplus W_{k+1}\oplus\dots\oplus W_{d-1})\big)
\end{equation}
to identify the corresponding elements.
More detail on flag manifolds can be found in Section \ref{sec : proof flag}.

\medskip
For every $k+1\leq i\leq d+1$ we define a $2$-dimensional real vector bundle $K_i$ over $\flag_{k,\dots,d} (V)$ whose fiber over the point $(W_{k},\dots,W_{d+1})\overset{\eqref{eq : identification}}{=}(V_k,V_{k+1},\dots,V_d)\in \flag_{k,\dots,d} (V)$ is the real vector space $\map(S(W_{i}),\R)$.
The vector bundle $K_i$ decomposes into the direct sum 
\[
K_i\cong E_{i}\oplus\underline{\R},
\]
where $E_{i}$, as in Section \ref{sec : proof flag}, denotes the  canonical line bundle associated to the flag manifold  $\flag_{k,\dots,d} (V)$ and $\underline{\R}$ is the trivial line bundle which corresponds to constant functions.

\medskip
Take an integer $k+1\leq i\leq d+1$, and let $\varphi\colon S(E^{d+1}_i)\longrightarrow\R$ be a continuous real valued function. 
It induces a section $s_{i,\varphi}$ of $K_i$ defined by
\[
(W_{k},\dots,W_{d})\overset{\eqref{eq : identification}}{=}(V_k,V_{k+1},\dots,V_d)\longmapsto \big( s_{\varphi}(W_{k},\dots,W_{d})\colon S(W_{i})\longrightarrow\R \big),
\]
where 
\[
s_{i,\varphi}(W_{k},\dots,W_{d})(u):=\int_{\{ v\in V_{i} :\langle v,u\rangle\geq 0 \}\cap S(V_i)} \varphi.
\]
The section $s_{i,\varphi}$ induces additionally a section  $s_{i,\varphi}'$ of the vector bundle $E_i$ by 
\[
(W_{k},\dots,W_{d})\overset{\eqref{eq : identification}}{=}(V_k,V_{k+1},\dots,V_d)\longmapsto \big( s_{i,\varphi}'(W_{k},\dots,W_{d})\colon S(W_{i})\longrightarrow\R \big),
\]
where 
\[
s_{i,\varphi}'(W_{k},\dots,W_{d})(u):= s_{i,\varphi}(W_{k},\dots,W_{d})(u)-s_{i,\varphi}(W_{k},\dots,W_{d})(-u).
\]

\medskip
Now, the CS\,/\,TM scheme theorem for the Fairy Bread Sandwich theorem can be stated as follows.

\begin{theorem}
\label{th : CSTM for Fairy Bread SandwichFairy Bread Sandwich}	
Let $d\geq 1$ and $k\geq 1$ be integers with $d\geq k$, and let $V=\R^{d+1}$ be a real vector space.
Fix a permutation $(j_k,\dots,j_d)$ of the set  $\{k,\dots, d\}$, and take an arbitrary collections of functions $\varphi_{a,b}\colon S(E_{a+1}^{d+1})\longrightarrow\R$, $k\leq a\leq d$, $1\leq b\leq j_a$, from the sphere bundle of the tautological vector bundle  $E_{a+1}^{d+1}$ over the Grassmann manifold $G_{a+1}(V)$ to the real numbers.

\smallskip\noindent
If the Euler class of the vector bundle $E:=E_{k+1}^{\oplus j_k}\oplus E_{k+2}^{\oplus j_{k+1}}\oplus\cdots\oplus E_{d+1}^{\oplus j_d}$ does not vanish, then there exists a flag $(W_{k},\dots,W_{d})=(V_k,\dots,V_d)\in \flag_{k,\dots,d}(V)$ such that for every $k\leq a \leq d$ and every $1\leq b\leq j_a$ the following equality holds
\[
\int_{\{v\in V_{a+1} : \langle v,u_a\rangle\geq 0\}\cap S(V_{a+1})}\varphi_{a,b}
=
\int_{\{v\in V_{a+1} : \langle v,u_a\rangle\leq 0\}\cap S(V_{a+1})}\varphi_{a,b}.
\]	
Here the unit vectors $u_{k}, \dots, u_{d}$ are determined, up to a sign, by the equality $V_{r}=\{v\in V_{r+1} : \langle v,u_{r}\rangle=0\}$, $k\leq r\leq d$, and with $V_{d+1}=V$.
This means that $u_{r}$ is a unit normal vector to $V_{r}$, considered as a hyperplane inside $V_{r+1}$.
In other words, $u_{r}\in S(W_{r+1})$ for all $k\leq r\leq d$.
\end{theorem}
\begin{proof}
Assume that 	the Euler class of the vector bundle $E=E_{k+1}^{\oplus j_k}\oplus\cdots\oplus E_{d+1}^{\oplus j_d}$ does not vanish.
Hence, every section of $E$ has a zero.

\medskip
The collections of functions $\varphi_{a,b}\colon S(E_{a+1}^{d+1})\longrightarrow\R$, $k\leq a\leq d$, $1\leq b\leq j_a$ induces a section of the vector bundle $E$ in the following way
\[
(W_{k},\dots,W_{d})
\longmapsto
\bigoplus_{1\leq b \leq j_k}s_{k+1,\varphi_{k,b}}'(W_{k},\dots,W_{d})
\oplus\cdots\oplus
\bigoplus_{1\leq b\leq j_d}s_{d+1,\varphi_{d,b}}'(W_{k},\dots,W_{d}).
\]
Thus, there exists   a flag $(W_{k},\dots,W_{d})=(V_k,\dots,V_d)\in \flag_{k,\dots,d}(V)$ such that for every $k\leq a \leq d$ and every $1\leq b\leq j_a$ the following statement holds
\begin{multline*}
s_{a+1,\varphi_{a,b}}'(W_{k},\dots,W_{d})(u)=\\ \int_{\{v\in V_{a+1} : \langle v,u\rangle\geq 0\}\cap S(V_{a+1})}\varphi_{a,b}
-
\int_{\{v\in V_{a+1} : \langle v,u\rangle\leq 0\}\cap S(V_{a+1})}\varphi_{a,b}
=0.	
\end{multline*}
for $u\in S(W_{a+1})$.
This concludes the proof. 
\end{proof}

%--------------------------------------------------------------------------------------%
%--------------------------------------------------------------------------------------%
%--------------------------------------------------------------------------------------%
\section{Proofs of Theorems \ref{th : main result 01} and \ref{th : main result 02} } 
\label{sec : Proofs}
%--------------------------------------------------------------------------------------%
%--------------------------------------------------------------------------------------%
%--------------------------------------------------------------------------------------%

For the proofs of the theorems we recall and show the following classical fact, see for example \cite[Satz und Def.\,VI.6.4]{BrockerTomDieck1970}, \cite[Thm.\,17.2.5 and Def.\,17.2.6]{Husemoller1994} and \cite[(1.13)]{Dold-1988}.

\begin{lemma}
\label{claim : 00}	
Let $E$ be a Euclidean vector bundle of dimension $n$ over a compact and connected ENR $B$, and let $\PP(E)$ denote the associated projective bundle of $E$. 
Then there is an isomorphism of $H^*(B;\F_2)$-algebras
\[
H^*(B;\F_2)[x]/\big( \sum_{s = 0}^{n}   w_{n-s}(E)\,x^{s} \big) \longrightarrow H^*(\PP(E);\F_2)
\]
which maps $x$ to the mod $2$ Euler class of the Hopf line bundle $H(E)$.
\end{lemma}

\medskip
For the proof first recall that for $m\geq 2$
\[
H^*(\PP(\R^m);\F_2)=H^*(\RP{m-1};\F_2)\cong\F_2[x]/(x^m),
\]
where $x=\e(H(\R^m))$ is the mod $2$ Euler class of the Hopf line bundle $H(\R^m)$.
In the case when $m=\infty$ we have 
\[
H^*(\PP(\R^{\infty});\F_2)=H^*(\RP{\infty};\F_2)\cong\F_2[x],
\]
where $x=\e(H)$ is the mod $2$ Euler class of the Hopf line bundle $H:=H(\R^{\infty})$.

\medskip
Second, we point out that for an $n$-dimensional vector bundle $E$ over a compact and connected ENR $B$ we can define its Stiefel--Whitney classes in the following way.
Consider, the projections 
\[
p_1\colon B \times \PP(\R^{\infty})\longrightarrow  B
\qquad\text{and}\qquad
p_2\colon B\times \PP(\R^{\infty})\longrightarrow  \PP(\R^{\infty}),
\]
and the mod $2$ Euler class of the vector bundle $p_1^*E\otimes p_2^*H$ which lives in the cohomology
\[
H^*(B \times \PP(\R^{\infty});\F_2)\cong H^*(B;\F_2)\otimes H^*(\PP(\R^{\infty});\F_2)  \cong H^*(B;\F_2)\otimes \F_2[x].
\]
Hence, there exist classes $w_i\in H^i(B;\F_2)$, $0\leq i \leq n$, such that
\begin{equation}
	\label{eq : def SW class 01}
	\e (p_1^*E\otimes p_2^*H)=\sum_{i=0}^n w_i \times x^{n-i}.
\end{equation}
Here ``$\times$'' denotes the cohomology cross product; see for example \cite[Thm.\,VI.3.2]{Bredon-1972}.

\medskip
Then we define the $i$-th Stiefel--Whitney class of $E$ to be $w_i$ for $0\leq i\leq n$ and $0$ otherwise, that is $w_i(E)=w_i$ for $0\leq i\leq n$ and $w_i(E)=0$ for $i\geq n+1$; consult for example \cite[Thm.\,17.2.5 and Def.\,17.2.6]{Husemoller1994}. 
Thus, the relation \eqref{eq : def SW class 01} becomes 
\begin{equation}
	\label{eq : def SW class 02}
	\e (p_1^*E\otimes p_2^*H)=\sum_{i=0}^n w_i(E)\times x^{n-i}.
\end{equation}

\medskip
Let us now consider a real line bundle $L$ over a compact ENR $B'$, and let $p_1'\colon B \times B'\longrightarrow  B$ and $p_2'\colon B \times B'\longrightarrow  B'$ be the projections.
The line bundle $L$ is isomorphic to a pull-back bundle $f^*H$ of the Hopf line bundle $H$ for some continuous map $f\colon B'\longrightarrow \PP(\R^{\infty})$.
In particular, the mod $2$ Euler class of $L$ is $\e (L)=f^*(t)$.
Consequently, first
\begin{equation}
	\label{eq : def SW class 02.5}
p_1'^*E\otimes p_2'^*L \cong (\id \times f)^*\big(p_1^*E\otimes p_2^*H\big).
\end{equation}
Second, the naturality  of the Euler class and the description of the map  $\id \times f$ on the level of cohomology imply that
\begin{align} \label{eq : def SW class 03}
\e (p_1'^*E\otimes p_2'^*L)& \overset{\eqref{eq : def SW class 02.5}}{=} (\id \times f)^*\big(\e (p_1^*E\otimes p_2^*H)\big)\overset{\eqref{eq : def SW class 02}}{=}  (\id \times f)^*  \big(\sum_{i=0}^n w_i(E)\times x^{n-i}	\big)\nonumber \\
&= \sum_{i=0}^n w_i(E)\times \e(L)^{n-i}.
\end{align}

\medskip
Now, if $B'=B$, and $d\colon B\longrightarrow  B\times B$ denotes the diagonal embedding, we have that $E\otimes L\cong d^*(p_1'^*E\otimes p_2'^*L)$.
Consequently, from \eqref{eq : def SW class 03} and the definition of the cup product \cite[Def.\,VI.4.1]{Bredon-1972}, we get 
\begin{align} 
	\label{eq : def SW class 04}
\e(E\otimes L) &=\e (d^*(p_1'^*E\otimes p_2'^*L))=d^*\big(\sum_{i=0}^n w_i(E)\times \e(L)^{n-i}\big)\nonumber \\
&=\sum_{i=0}^n w_i(E) \e(L)^{n-i}.
\end{align}

\begin{proof}[Proof of Lemma \ref{claim : 00}]
The powers of the mod $2$ Euler class of the Hopf line bundle $\e(H(E))^i\in H^i(\PP(E);\F_2)$ for $0\leq i\leq n-1$ when restricted to each fibre $\PP(E_b)$, $b\in B$, of the bundle $\PP(E)$ form a basis on  $H^*(\PP(E_b);\F_2)$.
By the Leray--Hirsch theorem $H^*(\PP(E);\F_2)$ is a free $H^*(B;\F_2)$-module with a basis $1, \e(H(E)),\dots, \e(H(E))^{n-1}$; see \cite[Thm.\,17.1.1]{Husemoller1994}.

\medskip
We recall that  the Hopf line bundle $H(E)$ is defined as
\[
H(E)=\{(b,L,v) : b\in B,\, L\in \PP(E_b),\,v\in L\}.
\]
In particular, it is a subbundle of the pull-back of $E$ along the projection map 
\[
g\colon \PP(E)=\{(b,L) : b\in B,\, L\in \PP(E_b)\}\longrightarrow B,
\qquad
(b,L)\longmapsto b.
\]
Now, note that the line bundle $H(E)$ can be identified with its dual line bundle $H(E)^*$ by the inner product. 
So 
\[
g^*E\otimes H(E)\cong g^*E \otimes H(E)^* \cong \Hom (H(E),g^*E)
\] 
is the vector bundle whose sections are linear maps $H(E)\longrightarrow g^*E$. 
The inclusion of $H(E)$ into $g^*E$ gives a nowhere-zero cross-section of the vector bundle $g^*E\otimes H(E)$.
Consequently, we have that
\begin{align*}
0&=\e(g^*E\otimes H(E))\overset{\eqref{eq : def SW class 04}}{=}\sum_{i=0}^n w_i(g^*E) \e(H(E))^{n-i}=\sum_{i=0}^n g^*(w_i(E)) \e(H(E))^{n-i}	\\
&= \sum_{i=0}^n w_i(E)\cdot \e(H(E))^{n-i}.
\end{align*}
Here ``$\cdot$'' refers to the $H^*(B;\F_2)$-module structure.
Therefore,
\[
\e(H(E))^{n}= \sum_{i=1}^n w_i(E)\cdot \e(H(E))^{n-i},
\]
which completely determines the structure of $ H^*(P(E);\F_2)$ as an $H^*(B;\F_2)$-algebra.

\end{proof}

\medskip
Now we proceed with the proofs of  Theorem \ref{th : main result 01} and  Theorem \ref{th : main result 02}.

%--------------------------------------------------------------------------------------%
\subsection{Proof of Theorem \ref{th : main result 01} }
\label{subsec : proof of main result 01}
%--------------------------------------------------------------------------------------%

Let $E$ be a Euclidean vector bundle of dimension $n$ over a compact and connected ENR $B$, and let the integers $k\geq 1$ and $j\geq 1$ be fixed.
Assume that $e_{k}(B)^j$ does not belong to the ideal $\II_{k}(E)$.

\medskip
The proof of the theorem relies on the criterion from Theorem \ref{th : CSTM for assignement problem}, that is:
\[
 \e \big(\big( A_k(E)/\underline{\R}\big)^{\oplus j} \big)\neq 0
\quad\Longrightarrow\quad 
(j,k)\in \Delta_S(E).
\]
Observe that the mod $2$ Euler class of the vector bundle $\big( A_k(E)/\underline{\R}\big)^{\oplus j}$, or in other words the top Stiefel--Whitney class, lives in the cohomology of the pullback bundle $H^*(d^*(\PP(E)^{\times k});\F_2)$.
We will prove that
\begin{compactitem}[\quad --]
\item  $H^*(d^*(\PP(E)^{\times k});\F_2)\cong R_{k}(B)/\II_{k}(E)$, and  
\item $w_{(2^k-1)j} \big(\big( A_k(E)/\underline{\R}\big)^{\oplus j} \big)=e_{k}(B)^j+\II_{k}(E)\in  R_{k}(B)/\II_{k}(E)$.
\end{compactitem}
Assuming these two claims to be true, the criterion from Theorem \ref{th : CSTM for assignement problem} yields:
\[
e_{k}^j + \II_{k}(E)\neq \II_{k}(E)\text{ in }R_{k}(B)/ \II_{k}(E)\quad\Longrightarrow\quad (j,k)\in \Delta_S(E).
\]
Thus, the proof of Theorem \ref{th : main result 01} is finished, up to a proof of the two facts we listed.

\medskip
First, we compute the cohomology of the pullback bundle $d^*(\PP(E)^{\times k})$ because the Stiefel--Whitney class $w\big(\big(A_k(E)/\underline{\R}\big)^{\oplus j}\big)$ belongs to $H^*\big(d^*(\PP(E)^{\times k});\F_2\big)$.

\begin{claim}
\label{claim : 01}
There is an isomorphism of $H^*(B;\F_2)$-algebras
\begin{multline*}
R_{k}(B)/ \II_{k}(E)  = 
H^*(B;\F_2)[x_1,\dots,x_k]\,/\,\big( \sum_{s = 0}^{n}   w_{n-s}(E)\,x_{r}^{s}  	\: : \:	1 \leq r \leq k \big) \\	\longrightarrow
H^*(d^*(\PP(E)^{\times k});\F_2)
\end{multline*}	
mapping $x_r$ to the mod $2$ Euler class of the pullback vector bundle $\Theta _r^*\big(H(E))$ for all $1\leq r\leq k$.
\end{claim}

\begin{proof}
The proof proceeds by induction on $j$ where $1\leq j\leq k$. 
If $j=1$, then the statement reduces to Lemma \ref{claim : 00}.
Let $j\geq 2$, and assume that there is an isomorphism
\begin{multline}\label{IH1}
H^*(B;\F_2)[x_1,\dots,x_{j-1}]\,/\,\big( \sum_{s = 0}^{n}   w_{n-s}(E)\,x_{r}^{s}  	\: : \:	1 \leq r \leq j-1 \big)  \\
\longrightarrow
H^*(d^*(\PP(E)^{\times (j-1)});\F_2)
\end{multline}	
which maps each class $x_r$ to the mod $2$ Euler class of the pullback vector bundle $\Theta _r^*\big(H(E))$, where $1\leq r\leq j-1$.

\medskip
The pullback bundle $d^*(\PP(E)^{\times (j-1)})$ is a bundle over $B$ with the corresponding projection map 
$p \colon d^*(\PP(E)^{\times (j-1)})\longrightarrow B$. 
Then $d^*(\PP(E)^{\times j})$ is isomorphic to the pullback bundle $p^*(P(E))\cong P(p^*(E))$ over $d^*(\PP(E)^{\times (j-1)})$.
Recall that $\PP(E)$ is the projective bundle associated to $E$, and therefore a bundle over $B$.
Hence, there is a pullback digram
\[
\xymatrix{d^*(\PP(E)^{\times j})\cong p^*(\PP(E))\cong \PP(p^*(E))\ar[r]\ar[d] & \PP(E)\ar[d] \\
d^*(\PP(E)^{\times (j-1)})\ar[r]^-{p}&B.
}
\]
Consequently, by Lemma \ref{claim : 00}, we get an isomorphism of $H^*(d^*(\PP(E)^{\times (j-1)});\F_2)$-algebras
\begin{multline}\label{IH2}
H^*(d^*(\PP(E)^{\times (j-1)});\F_2)[x_j]/\big( \sum_{s = 0}^{n}   w_{n-s}(E)\,x^{s}_j \big) \\\longrightarrow H^*(\PP(p^*(E));\F_2)\cong H^*(d^*(\PP(E)^{\times j});\F_2)
\end{multline}	
which maps $x_j$ to the mod $2$ Euler class of the Hopf line bundle $H(p^*(E))$.

\medskip
Now, the induction hypothesis \eqref{IH1} in combination with the isomorphism \eqref{IH2} completes the proof of the claim.
\end{proof}

\medskip
Finally, we evaluate the Stiefel--Whitney class $w_{(2^k-1)j}\big(\big( A_k(E)/\underline{\R}\big)^{\oplus j}\big)$.

\begin{claim}
\label{claim : 02}
The mod $2$ Euler class of the vector bundle $\big( A_k(E)/\underline{\R}\big)^{\oplus j}$ is equal to:
\begin{multline*}
w_{(2^k-1)j}\big(\big( A_k(E)/\underline{\R}\big)^{\oplus j}\big)=e_{k}(B)^j+\II_{k}(E) = \\ 
\prod_{(\alpha_{1},\dots, \alpha_{k})\in\F_2^k-\{ 0\}} (\alpha_{1}x_1 + \cdots + \alpha_{k}x_k)^{j}+\II_{k}(E)
\in  R_{k}(B)/ \II_{k}(E).
\end{multline*}
\end{claim}
\begin{proof}
Recall the  isomorphism of vector bundles
\[
A_k(E)\cong
\Theta _1^*\big(H(E)\oplus\underline{\R}\big)
\otimes\dots\otimes
\Theta_k^*\big(H(E)\oplus\underline{\R}\big),
\]
where $\underline{\R}$  is the trivial line bundle over $\PP(E)$, and $\Theta _i^*\big(H(E)\oplus\underline{\R}\big)$ is a pullback vector bundle.
Now the claim follows from the distributivity of the tensor product over the direct sum, the fact that the pullback of trivial  bundle is again a trivial bundle, and the equality $w(\alpha\otimes\beta)=1+ (w_1(\alpha)+w_1(\beta))$ which holds (only) for line bundles $\alpha$ and $\beta$ (see \cite[Prob.\,7-A]{Milnor1974}).
Note that $\alpha\otimes\beta$ is also a line bundle.
\end{proof}

\medskip
With the  claims verified, the proof of Theorem \ref{th : main result 01} is now  complete. 

%--------------------------------------------------------------------------------------%
\subsection{Proof of Theorem \ref{th : main result 02}}
\label{subsec : proof of main result 02}
%--------------------------------------------------------------------------------------%

The proof we present is an extension of the proof of Theorem \ref{th : main result 01}, and therefore it is just outlined.
Let $k\geq 1$ and $j\geq 1$ be fixed integers.
We consider a Euclidean vector bundle $E$ of dimension $n$ over a compact and connected ENR $B$, and, in addition, $k$ vector subbundles $E(1),\dots,E(k)$ of $\eta$ of dimensions $n_1,\dots,n_k$, respectively.
Assume that $j\leq \iota_k (E(1),\dots,E(k))=\max\big\{j :  e_{k}(B)^j\notin \II_k(E(1),\dots,E(k))\big\}$.

\medskip
The proof of the theorem uses the criterion from Theorem \ref{th : CSTM for assignement problem with constraints}. That is, if the Euler class 
$
 \e \big(  \big(A_k(E(1),\dots, E(k))/\underline{\R}\big)^{\oplus j} \big)\neq 0
$
does not vanish, then for every collection of $j$ continuous functions $\varphi_1,\dots,\varphi_j\colon S(E)\longrightarrow\R$ there exists a point $b\in B$ and an arrangement of $k$ linear hyperplanes $H_1, \dots, H_k$ in the fibre $E_b$ determined by the collection of vector subbundles $E(1),\dots,E(k)$ with the property that for every pair of connected components $(\mathcal{O}',\mathcal{O}'')$ of the arrangement complement $E_b -(H_1\cup\dots\cup H_k)$ the following equalities hold
\[
\int_{\mathcal{O}'\cap S(E_b )}\varphi_1=\int_{\mathcal{O}''\cap S(E_b )}\varphi_1
\quad , \dots , \quad 
\int_{\mathcal{O}'\cap S(E_b )}\varphi_j=\int_{\mathcal{O}''\cap S(E_b )}\varphi_j.
\]

\medskip
The mod $2$ Euler class of $\big( A_k(E(1),\dots, E(k))/\underline{\R}\big)^{\oplus j} $, or in other words the top Stiefel--Whitney class, lives in $H^*\big(d^*\big(\PP(E(1))\times\dots\times\PP(E(k))\big);\F_2\big)$.
Therefore, we prove that
\begin{compactitem}[\quad --]
\item  $H^*\big(d^*\big(\PP(E(1))\times\dots\times\PP(E(k))\big);\F_2\big)\cong R_{k}(B)/\II_k(E(1),\dots,E(k))$, and that
\item $w_{(2^k-1)j}\big( \big( A_k(E(1),\dots, E(k))/\underline{\R}\big)^{\oplus j}\big)=e_{k}(B)^j+\II_k(E(1),\dots,E(k))$.
\end{compactitem}
If these two statements are assumed to be true, then the criterion from Theorem \ref{th : CSTM for assignement problem with constraints}, in combination with the theorem assumption $j\leq \iota_k (E;E(1),\dots,E(k))$, implies that
\begin{multline*}
w_{(2^k-1)j}\big( \big( A_k(E(1),\dots, E(k))/\underline{\R}\big)^{\oplus j}\big)=\\e_{k}^j+\II_k(E(1),\dots,E(k))\neq \II_k(E(1),\dots,E(k))	.
\end{multline*}
Hence, 
$
  \e \big(  \big(A_k(E(1),\dots, E(k))/\underline{\R}\big)^{\oplus j} \big)\neq 0
$,
~and the proof of Theorem \ref{th : main result 01} is complete.
Indeed, the remaining claims are verified in the same way as in the proofs of Claim \ref{claim : 01} and Claim \ref{claim : 02}.

%--------------------------------------------------------------------------------------%
%--------------------------------------------------------------------------------------%
%--------------------------------------------------------------------------------------%
%--------------------------------------------------------------------------------------%
\section{Proofs of Propositions \ref{prop : ham-sandwich} and  \ref{prop : for k >1}}
%--------------------------------------------------------------------------------------%
%--------------------------------------------------------------------------------------%
%--------------------------------------------------------------------------------------%
%--------------------------------------------------------------------------------------%

We prove the main facts about the integers $\iota_1(E)$ and $\iota_k(E(1),\dots,E(k))$ stated in Proposition \ref{prop : ham-sandwich} and Proposition \ref{prop : for k >1}, as well as two related consequences, Corollary   \ref{cor : A} and  Corollary \ref{cor : D}.
%--------------------------------------------------------------------------------------%
\subsection{Proof of Proposition \ref{prop : ham-sandwich}}
\label{subsec : prop : ham-sandwich}
%--------------------------------------------------------------------------------------%

Let $E$ be a Euclidean vector bundle of dimension $n$ over a compact and connected ENR $B$.

\medskip
Since $k=1$ we simplify notation by taking $x=x_1$.
Hence, $e_1(B)=x$ and $\II_1(E)=\big( x^n+w_1(E)x^{n-1}+\dots +w_n(E)\big)$.
Set 
 \[
 a:=\iota_1(E)=\max\big\{j : x^j\notin \II_1(E) \big\}
 \]
 and
 \[
b:=\max\big\{j : 0\neq w_{j-n+1}(-E)\in H^{j-n+1}(B;\F_2)\big\}.
\]
In particular, we have that $w_{b-n+1}(-E)\neq 0$ and that $w_{r}(-E)=0$ for all $r\geq b-n+2$.
Now, we prove that $a=b$.

\medskip
Using the Euclidean algorithm in the polynomial ring $R_1(B)=H^*(B;\F_2)[x]$ we have that
\[
x^b=(x^n+w_1(E)x^{n-1}+\dots +w_n(E))q+d_{n-1}x^{n-1}+\dots+d_1x+d_0
\]
where $q\in R_1(B)$, and for $0\leq i\leq n-1$ it holds that
\[
d_i=w_{b-i}(-E)+w_1(E)w_{b-i-1}(-E)+\dots+w_{n-i-1}(E)w_{b-n+1}(-E)\in H^*(B;\F_2),
\]
as demonstrated by Crabb and Jan Jaworowski \cite[Proof of Prop.\,4.1]{Crabb2013}.
Since $w_{r}(-E)\neq 0$ for $r\geq b-n+2$ it follows that $d_i=w_{n-i-1}(E)w_{b-n+1}(-E)$ for $0\leq i\leq n-1$, and so
\begin{multline*}
x^b=(x^n+w_1(E)x^{n-1}+\dots +w_n(E))q+\\w_{b-n+1}(-E)\big(x^{n-1}+w_1(E)x^{n-2}\dots+w_{n-1}(E)\big).	
\end{multline*}
Consequently, from $w_{b-n+1}(-E)\neq 0$ it follows that $x^b\notin\II_1(E)$ and accordingly $b\leq a$.

\medskip
Let us now assume that $b<a$, or in other words $b-n+2\leq a-n+1$.
Recall that $w_{r}(-E)\neq 0$ for $r\geq b-n+2$, and in particular for  $r\geq a-n+1$. 
Once again we have 
\[
x^a=(x^n+w_1(E)x^{n-1}+\dots +w_n(E))q'+d_{n-1}'x^{n-1}+\dots+d_1'x+d_0'
\]
where 
\[
d_i'=w_{a-i}(-E)+w_1(E)w_{a-i-1}(-E)+\dots+w_{n-i-1}(E)w_{a-n+1}(-E)=0,
\]
for all $0\leq i\leq n-1$.
Hence,  $x^a\in\II_1(E)$, which is a contradiction with the definition of the integer $a$.
Therefore, $b\geq a$.

\medskip
We proved that $a=b$, or in other words that
\[
\iota_1(E)=\max\big\{j : 0\neq w_{j-n+1}(-E)\in H^{j-n+1}(B;\F_2)\big\},
\]	
as claimed. \qed

%--------------------------------------------------------------------------------------%
\subsection{Proof of Proposition \ref{prop : for k >1}}
\label{subsec : prop : for k>1} 
%--------------------------------------------------------------------------------------%

Let $k\geq 1$ be an integer, and let  $E(1),\dots,E(k)$ be Euclidean vector bundles over a compact and connected ENR $B$.
Set $n_i$ to be the dimension of the vector bundle $E(i)$ for $1\leq i\leq k$.
Let us denote by 
\begin{align*}
	a_i:=&\iota_1(E(i))= \max\big\{j : 0\neq w_{j-n_i+1}(-E(i))\in H^{j-n_i+1}(B;\F_2)\big\},\\
	a:=&\iota_k(a_1+1,\dots,a_k+1)= \max\big\{ j : e_k(\pt)^j\notin (x_1^{a_1+1},\dots,x_k^{a_k+1})\big\},\\
	b :=&\iota_k(E(1),\dots,E(k))=\max\big\{j :  e_{k}(B)^j\notin \II_k(E(1),\dots,E(k))\big\},
\end{align*}
where $1\leq i\leq k$, and
\[
\II_k(E(1),\dots,E(k))=\Big( \sum_{s = 0}^{n_r}   w_{n_r-s}(E(r))\,x_{r}^{s}  	\: : \:	1 \leq r \leq k \Big) \ \subseteq \ R_{k}(B).
\]
In the definition of $a_i$ we used the characterization from Proposition \ref{prop : ham-sandwich}.
In particular, we have that $w_{r}(-E(i))=0$ for all $r\geq a_i-n_i+2$.

\medskip
With the notation we just introduced the  the assumption of the proposition reads
\[
 w_{a_1-n_1+1}(-E(1))\cdots  w_{a_k-n_k+1}(-E(k))\neq 0,
\]
while the claim of the proposition becomes $a=b$.

\medskip
The main ingredients of our proof of Proposition \ref{prop : for k >1} are contained in the next two claims, where the first claim is used for the proof of the second.

\begin{claim}
\label{claim A01}
$x_1^{a_1}\cdots x_k^{a_k}\notin \II_k(E(1),\dots,E(k))$.	
\end{claim}
\begin{proof}
For simplicity set $\II:=\II_k(E(1),\dots,E(k))$.
Once again we use \cite[Proof of Prop.\,4.1]{Crabb2013} and get that for all $1\leq i\leq k$:
\[
x_i^{a_i}=(x_i^{n_i}+w_1(E(i))x_i^{n_i-1}+\dots +w_{n_i}(E(i)))\cdot q_i+d_{{n_{i}-1},i}x^{n_i-1}_i+\dots +d_{0,i},
\] 
where for $0\leq s\leq n_i-1$:
\begin{multline*}
d_{s,i}=w_{a_i-s}(-E(i))+w_1(E(i))\,w_{a_i-s-1}(-E(i))+\dots+\\w_{n_i-s-1}(E(i))\,w_{a_i-n_i+1}(-E(i)).
\end{multline*}
More precisely, since $w_{r}(-E(i))=0$ for all $r\geq a_i-n_i+2$, we have that
\[
d_{s,i}=w_{n_i-s-1}(E(i))\,w_{a_i-n_i+1}(-E(i)).
\]
Consequently,
\[
x_i^{a_i}+\II=
w_{a_i-n_i+1}(-E(i))\big(x_i^{n_i-1}+\dots +w_{n_i-1}(E(i))\big)+\II ,
\]
and so
\[
x_1^{a_1}\cdots x_k^{a_k}	+\II =
\prod_{i=1}^k w_{a_i-n_i+1}(-E(i)) \cdot \prod_{i=1}^k\big(x_i^{n_i-1}+\dots +w_{n_i-1}(E(i))\big) +\II.
\]
From the assumption of the proposition $\prod_{i=1}^k w_{a_i-n_i+1}(-E(i))\neq 0$ we have that $x_1^{a_1}\cdots x_k^{a_k}	+\II\neq \II$, or in other words $x_1^{a_1}\cdots x_k^{a_k}\notin \II$, as claimed.  
\end{proof}

\begin{claim}
\label{claim A02}
\begin{multline*}
(x_1^{a_1+1},\dots,x_k^{a_k+1})=\\
\ker\Big(\F_2[x_1,\dots,x_k]\longrightarrow  H^*(B;\F_2)[x_1,\dots,x_k]/	\II_k(E(1),\dots,E(k)) \Big).
\end{multline*}
\end{claim} 

\begin{proof}
The ring homomorphism we consider
\[
h\colon \F_2[x_1,\dots,x_k]\longrightarrow  H^*(B;\F_2)[x_1,\dots,x_k]/	\II_k(E(1),\dots,E(k)) 
\]	
is induced by the coefficient inclusion $\F_2\lhook\joinrel\longrightarrow H^0(B;\F_2)\lhook\joinrel\longrightarrow  H^*(B;\F_2)$.
Furthermore, denote by $\JJ:=\ker(h)$.

\medskip
Like in the proof of the previous claim we use \cite[Proof of Prop.\,4.1]{Crabb2013} and for every $1\leq i\leq k$ get that
\begin{multline*}
x_i^{a_i+1}=(x_i^{n_i}+w_1(E(i))x_i^{n_i-1}+\dots +w_{n_i}(E(i)))\cdot q_i+\\d_{n_i-1,i}'x^{n_i-1}_i+\dots +d_{0,i}' \ \in \ H^*(B;\F_2)[x_1,\dots,x_k].	
\end{multline*}
Here for $0\leq s\leq n_i-1$:
\begin{multline*}
d_{s,i}'=w_{a_i+1-s}(-E(i))+w_1(E(i))\,w_{a_i-s}(-E(i))+\dots+\\w_{n_i-s-1}(E(i))\,w_{a_i-n_i+2}(-E(i)).
\end{multline*}
In this case the fact that $w_{r}(-E(i))=0$ for all $r\geq a_i-n_i+2$ implies that $d_{s,i}'=0$ for all $0\leq s\leq n-1$ and all $1\leq i\leq k$.
Consequently, $x_i^{a_i+1}\in \II_k(E(1),\dots,E(k))$, or in other words, $x_i^{a_i+1}\in \JJ$, for all $1\leq i\leq k$.
Hence, $(x_1^{a_1+1},\dots,x_k^{a_k+1})\subseteq\JJ$.

\medskip
Assume that
\[
0\neq p=\sum_{(c_1,\dots,c_k)\in C } \alpha_{c_1,\dots,c_k} x_1^{c_1}\cdots x_k^{c_k} \in \JJ - (x_1^{a_1+1},\dots,x_k^{a_k+1}),
\]
where $C\subseteq \Z_{\geq 0}^k$ is a finite set of multi-exponents of the polynomial $p$, and $\alpha_{c_1,\dots,c_k}  \in\F_2$ are the coefficients.
After a possible modification of $p$, by taking away monomials which already belong to the ideal $(x_1^{a_1+1},\dots,x_k^{a_k+1})$, we can assume that the set of exponents satisfies  
\[
\emptyset\neq C\subseteq [0,a_1]\times\cdots\times [0,a_k].
\]
That is, no monomial in the representation of $p$ belongs to $(x_1^{a_1+1},\dots,x_k^{a_k+1})$.

\medskip
Let $s_{k}:=\min\{ s\in\Z_{\geq 0} : \alpha_{c_1,\dots,c_{k-1},s}\neq 0\}$. 
Then 
\[
x_k^{a_k-s_k}p\in \JJ-(x_1^{a_1+1},\dots,x_k^{a_k+1})
\]
with all monomials having degree of $x_k$ at least $a_k$. 
Taking away all monomials in $x_k^{a_k-s_k}p$ which already belong to $(x_1^{a_1+1},\dots,x_k^{a_k+1})$ we get a polynomial which still belongs to $\JJ-(x_1^{a_1+1},\dots,x_k^{a_k+1})$.
Now, repeat the procedure iteratively with variables $x_{k-1},\dots x_1$, respectively.
At the end we get that  
\[
x_1^{a_1}\cdots x_k^{a_k}\in \JJ-(x_1^{a_1+1},\dots,x_k^{a_k+1}).
\]
We reached a direct contradiction with Claim \ref{claim A01}.
In particular, it says, that $h(x_1^{a_1}\cdots x_k^{a_k})\neq 0$, or equivalently $x_1^{a_1}\cdots x_k^{a_k}\notin \ker(h)=\JJ$.

\end{proof}

\medskip
Finally, we complete the proof of Proposition  \ref{prop : for k >1} as follows.
According to the definition of $a$ we have that 
\[
e_k(\pt)^a\notin (x_1^{a_1+1},\dots,x_k^{a_k+1})=\ker(h)
\]
and
\[
e_k(\pt)^{a+1}\in (x_1^{a_1+1},\dots,x_k^{a_k+1})=\ker(h).
\] 
Consequently, 
\begin{align*}
e_k(B)^a+ \II_k(E(1),\dots,E(k))&=h(e_k(\pt)^a) \neq 0,	\\
e_k(B)^{a+1}+ \II_k(E(1),\dots,E(k))&=h(e_k(\pt)^{a+1})=0.
\end{align*}
From the definition of $b$ we conclude that $a=b$, as claimed. 
This argument completes the proof of  Proposition  \ref{prop : for k >1}. \qed

\subsection{Proof of Corollary \ref{cor : A}}
\label{sub : proof of cor: A}

In order to prove the statement, according to Proposition \ref{prop : for k >1}, it is enough to check whether 
$
\big(w_{d-\ell}(-E_{\ell}^{d})\big)^{k}\neq 0
$,
because $\iota_1(E_{\ell}^{d})=d-1$, as demonstrated in Corollary \ref{cor : ham-sandwich}. 	
Since $k\leq\ell$ it suffices to prove that $\big(w_{d-\ell}(-E_{\ell}^{d})\big)^{\ell}\neq 0$.
Indeed, the Gambelli's formula \cite[p.\,523]{Hiller-1980} \cite[Prop.\,9.5.37]{Hausmann2014} implies the equality 
\[
\big(w_{d-\ell}(-E_{\ell}^{d})\big)^{\ell}=
\det\big( w_{d-\ell+i-j}(-E_{\ell}^{d})\big)_{1\leq i,j\leq \ell}= 
[d-\ell,d-\ell,\dots, d-\ell]\neq
0.
\]
Here $[d-\ell ,d-\ell ,\dots, d-\ell]$ denotes a  Schubert class.
Note that $w_{r}(-E_{\ell})=0$ for all $r>d-\ell$, and that we assume $w_{r}(-E_{\ell})=0$ for $r<0$. \qed

 %--------------------------------------------------------------------------------------%
\subsection{Proof of Corollary \ref{cor : D}}
  %--------------------------------------------------------------------------------------%
 
From Theorem \ref{th : main result 01}	we have that $(j,k)\in \Delta_S(E_{\ell}^{d})$	if
$
e_{k}(B)^j\notin \II_{k}(E_{\ell}^{d})=\II_{k}(E_{\ell}^{d},\dots, E_{\ell}^{d})
$,
or in other words if
\[
	j\leq \iota_k(E_{\ell}^{d},\dots,E_{\ell}^{d})= \iota_k(d,\dots,d)= 
	\max\big\{ j' : e_k(\pt)^{j'}\notin (x_1^{d},\dots,x_k^{d}\big\}.
\]
Here the first equality comes from Corollary \ref{cor : A} while the second one is just the definition of $\iota_k(d,\dots,d)$.

\medskip
Since $j=2^t+r$ where $0\leq r\leq 2^t-1$ and $d\geq 2^{t+k-1}+r$, then according to \cite[Lem.\,4.2]{BlagojevicCallesCrabbDimitrijevic} we have that $e_k(\pt)^{j}\notin (x_1^{d},\dots,x_k^{d})$. 
Thus, indeed $j\leq  \iota_k(E_{\ell}^{d},\dots,E_{\ell}^{d})$ and the proof of the corollary is complete. \qed
 
%--------------------------------------------------------------------------------------%
%--------------------------------------------------------------------------------------%
%--------------------------------------------------------------------------------------%
%--------------------------------------------------------------------------------------%
\section{Proofs of Corollary \ref{cor : B}, \ref{cor : E} and Theorem \ref{cor : F}}
\label{sec : proof flag}
%--------------------------------------------------------------------------------------%
%--------------------------------------------------------------------------------------%
%--------------------------------------------------------------------------------------%
%--------------------------------------------------------------------------------------%

Before going into the proofs we recall the notion of a real flag manifold by introducing it in two equivalent ways. 
Furthermore we give description of the cohomology ring with coefficients in $\F_2$.

\medskip
Let $k\geq 1$ and $d\geq 2$ be integers.
Consider  a strictly increasing sequence of positive integers $(n_1,\dots,n_k)$ bounded by $d$, meaning $1\leq n_1<\dots<n_{k-1}<n_k\leq d-1$.
Set in addition $n_0=0$ and $n_{k+1}=d$.

\medskip
Let $V$ be a real vector space  of dimension $d$.
The real {\em flag manifold}, of type $(n_1,\dots,n_k)$, in $V$ is the space  $\flag_{n_1,\dots,n_k} (V)$ of all flags $0\subseteq V_1\subseteq \dots\subseteq V_k\subseteq V$ in $V$ with the property that $\dim(V_i)=n_i$ for every $1\leq i\leq k$.
Alternatively, we can say that $\flag_{n_1,\dots,n_k} (V)$ is a collection of all $(k+1)$-tuples of vector spaces $(W_1,\dots,W_{k+1})$ with the property that 
\begin{compactitem}[\quad --]
\item $\dim (W_i)=n_i-n_{i-1}$ for all $1\leq i\leq k+1$,   and
\item $W_{i'}\perp W_{i''}$ for all $1\leq i'<i''\leq k+1$.
\end{compactitem}
In other words
\begin{align*}
	\flag_{n_1,\dots,n_k} (V) &=\big\{(V_1,\dots,V_k)\in \prod_{i=1}^k G_{n_i}(V) :  0\subseteq V_1\subseteq \dots\subseteq V_k\subseteq V\big\}\\	
	 &\cong\big\{(W_1,\dots,W_{k+1})\in \prod_{i=1}^{k+1} G_{n_i-n_{i-1}}(V) : \\
	 &\hspace{3.5cm}  W_{i'}\perp W_{i''}   \text{ for all } 1\leq i'<i''\leq k+1\big\}\\
	 &\cong \frac{\OO(d)}{\OO(n_1-n_0)\times \OO(n_2-n_1)\times\dots \times\OO(n_{k+1}-n_k)}.
\end{align*}
The homeomorphism between these two presentations is given by
\[
(W_1,\dots,W_{k+1}) \longmapsto \big(W_1, (W_1\oplus W_2),\dots,( W_1\oplus W_2\oplus\dots\oplus W_{k-1})\big).
\]
The flag manifold $\flag_{n_1,\dots,n_k} (V)$ is indeed a compact $\delta$-dimensional manifold where $\delta:=\sum_{ 1\leq i'<i''\leq k+1}(n_{i'}-n_{i'-1}) (n_{i''}-n_{i''-1})$.
In the case when $k=d-1$, and consequently $n_i=i$ for all $1\leq i\leq k=d-1$, the flag manifold  $\flag_{1,2,\dots,d-1} (V)$ is called the {\em complete flag manifold}. 
Furthermore, the flag manifold $\flag_{n_1} (V)$ coincides with the Grassmann manifold $\Gr_{n_1}(V)\cong \Gr_{n_1}(\R^d)$.

\medskip
Over the flag manifold  $\flag_{n_1,\dots,n_k} (V)$ we have $k+1$ canonical vector bundles $E_1,\dots, E_{k+1}$ given by
\[
E_i:=\big\{ ((W_1,\dots,W_{k+1}),w)\in \flag_{n_1,\dots,n_k} (V)\times V : w\in W_i\big\},
\] 
where $1\leq i\leq k+1$.
In particular, $E_1\oplus\dots\oplus E_{k+1}$ is isomorphic to the trivial vector bundle $\flag_{n_1,\dots,n_k} (V)\times V$.
Now, the classical result of Armand Borel \cite[Thm.\,11.1]{Borel-1953} says that
\begin{multline*}
	H^*(\flag_{n_1,\dots,n_k} (V);\F_2)\cong  \\
	\F_2[w_1(E_1),\dots,w_{n_1-n_0}(E_1),\dots,w_1(E_{k+1}),\dots,w_{n_{k+1}-n_{k}}(E_{k+1})]\,/\,I_{n_1,\dots,n_k},
\end{multline*}
where the ideal $I_{n_1,\dots,n_k}$ is generated by the identity 
\begin{equation*}
	\big(1+w_1(E_1)+\dots +w_{n_1-n_0}(E_1)\big)\cdots\big(1+w_1(E_{k+1})+\dots+w_{n_{k+1}-n_{k}}(E_{k+1})\big)=1.
\end{equation*}
In particular, in the case of the complete flag manifold, equivalently when $k=d-1$, we have that
\begin{multline}
	\label{eq : cohomology of flag 01}
	H^*(\flag_{1,\dots,d-1} (V);\F_2)\cong   
	\F_2[w_1(E_1),w_1(E_2),\dots,w_{1}(E_d)]\,/\,I_{1,\dots,d-1}.
\end{multline}
In this case $E_1,\dots,E_d$ are all line bundles.
Here, the ideal $I_{1,\dots,d-1}$ is generated by the identity $\prod_{i=1}^d(1+w_1(E_i))=1$, which implies that a generating set for $I_{1,\dots,d-1}$ is the set of all elementary symmetric polynomials in $w_1(E_1),w_1(E_2),\dots,w_{1}(E_d)$ as variables.
Thus, 
\begin{equation}
\label{eq : cohomology of flag 02}
I_{1,\dots,d-1}=\Big(\sigma_{r}(w_1(E_1),w_1(E_2),\dots,w_{1}(E_d))	: 1\leq r\leq d\Big),
\end{equation}
where $\sigma_1,\dots,\sigma_d$ denote elementary symmetric polynomials.

\medskip
Flag manifolds of different types allow continuous maps between each other induced by a choice of a subflag.
In particular, for any type $(n_1,\dots,n_k)$ there is a continuous map
\[
\alpha_{n_1,\dots,n_k} \colon \flag_{1,\dots,d-1} (V) \longrightarrow \flag_{n_1,\dots,n_k} (V), 
\] 
given by the selection of a subflag
\[
0\subseteq V_1\subseteq V_2\subseteq\dots \subseteq V_{d-1}\subseteq V \ \longmapsto \ 0\subseteq V_{n_1}\subseteq V_{n_2} \subseteq\dots \subseteq  V_{n_k}\subseteq V .
\]   
An important feature of this map that the induced map in cohomology 
\[
\alpha_{n_1,\dots,n_k}^*\colon H^*(\flag_{n_1,\dots,n_k} (V);\F_2)\longmapsto H^*(\flag_{1,\dots,d-1} (V);\F_2)
\]
is injective; consult for example \cite[pp.\,523--524]{Hiller-1980}.

 %--------------------------------------------------------------------------------------%
\subsection{Proof of Corollary \ref{cor : B} }
\label{subsec : cor : flag}
%--------------------------------------------------------------------------------------%

We apply Proposition \ref{prop : for k >1}.
Thus, we need to compute first $\iota(E(i))$ for all $1\leq i\leq k$.
From Proposition \ref{prop : ham-sandwich} we have that
\[
 \iota_1(E(i))=\max\big\{j : 0\neq w_{j-\dim E(i)+1}(-E(i))\in H^{*}(B;\F_2)\big\},
\]
where in our situation $B:=\flag_{n_1,\dots,n_k} (V)$.

\medskip
Let $1\leq i\leq k$.
Consider the following commutative diagram of flag manifolds where all the maps are induced by a selection of the  corresponding subflags: 
\[
\xymatrix@1{
\flag_{n_1,\dots,n_k} (V)\ar[rr]^-{\beta_i} & &\flag_{n_i} (V)\cong \Gr_{n_i}(V) \\
 & \flag_{1,\dots,d-1} (V)\ar[ul]^{\alpha_{n_1,\dots,n_k}}.\ar[ur]_{\hspace{4mm}\alpha_{n_i}} &
}
\]
Since the induced maps in cohomology $\alpha_{n_1,\dots,n_k}^*$ and $\alpha_{n_i}^*$ are injective, it follows that the induced map $\beta_i^*$ is also injective.
Now, from the injectivity of $\beta_i^*$ and the fact $E(i)=\beta_i^*E_{n_i}^d$, in combination with Corollary \ref{cor : ham-sandwich} we have that 
$
\iota_1(E(i))= \iota_1(E_{n_i}^d)=d-1
$.
Here, like before, $E_{n_i}^d$ denotes the tautological bundle over the Grassmann manifold $\Gr_{n_i}(V)$.

\medskip
To conclude the proof of the corollary we verify the criterion from Proposition \ref{prop : for k >1}, that is, we prove that the following product does not vanish
\[
	u:=\prod_{i=1}^{k} w_{\iota_1(E(i))-n_i+1}(-E(i))=
	\prod_{i=1}^{k} w_{d-n_i}(-E(i))\in H^*(B;\F_2).
\]
 
 \medskip 
From the fact that $E(i)\oplus E_{i+1}\oplus\dots\oplus E_{k+1} = B\times V$ is a trivial vector bundle we get the following equality of total Stiefel--Whitney classes: 
\[
w(-E(i))=w(E_{i+1}\oplus\dots\oplus E_{k+1} )=w(E_{i+1})\cdots w( E_{k+1} ).
\]
Therefore,
\[
w_{d-n_i}(-E(i))=w_{d-n_i}(E_{i+1}\oplus\dots\oplus E_{k+1} )=w_{n_{i+1}-n_i}(E_{i+1})\cdots w_{n_{k+1}-n_k}( E_{k+1} ),
\]
because $\dim (E_{i+1}\oplus\dots\oplus E_{k+1} )=d-n_i$  and $\dim (E_r)=n_r-n_{r-1}$ for every $1\leq r\leq k+1$.
In particular, each Stiefel--Whitney class $w_{n_{r}-n_{r-1}}(E_{r})$ is the mod $2$ Euler class $\e(E_{r})$ of the vector bundle $E_{r}$. 
We calculate as follows:
\begin{align*}
	u &=\prod_{i=1}^{k} w_{d-n_i}(-E(i))  =\prod_{i=1}^{k} \prod_{r=i+1}^{k+1} w_{n_{i+1}-n_i}(E_{i+1})\cdots w_{n_{k+1}-n_k}( E_{k+1} )\\
	&= w_{n_{2}-n_1}(E_{2}) \cdot w_{n_{3}-n_2}(E_{3})^2 \cdots w_{n_{k+1}-n_k}( E_{k+1} )^k.
\end{align*}
Thus, it remains to show  that the class 
\[
w_{n_{2}-n_1}(E_{2}) \cdot w_{n_{3}-n_2}(E_{3})^2 \cdots w_{n_{k+1}-n_k}( E_{k+1} )^k 
\]
does not vanish in $H^*(B;\F_2)$.

\medskip
For that we apply the  homomorphism $\alpha_{n_1,\dots,n_k}^*$ to the class $u$ and land in the cohomology of the complete flag manifold $H^*(\flag_{1,\dots,d-1} (V);\F_2)$, that is
\begin{align*}
	\alpha_{n_1,\dots,n_k}^*(u) &= \alpha_{n_1,\dots,n_k}(w_{n_{2}-n_1}(E_{2}) \cdot w_{n_{3}-n_2}(E_{3})^2 \cdots w_{n_{k+1}-n_k}( E_{k+1} )^k)\\
	&= w_1(E_{n_1+1})\cdots w_1(E_{n_2}) \\
	&\hspace{5mm}w_1(E_{n_2+1})^2\cdots w_1(E_{n_3})^2\\
	&\hspace{23mm}\dots\\
	&\hspace{5mm}w_1(E_{n_k+1})^k\cdots w_1(E_{n_{k+1}})^k.
\end{align*}
The vector bundles on the farthest right hand side of the last equality are canonical line bundles over the complete flag manifold.
Here we used the isomorphisms
\[
\alpha_{n_1,\dots,n_k}^*E_2\cong  E_{n_1+1}\oplus\dots\oplus E_{n_2}, \ \dots \ ,
\alpha_{n_1,\dots,n_k}^*E_{k+1}\cong E_{n_k+1}\oplus\dots\oplus E_{n_{k+1}}.
\]

\medskip
Now, we observe that the monomial in the cohomology of the complete flag manifold
\[
w_1(E_{n_1+1})\cdots w_1(E_{n_2})\, w_1(E_{n_2+1})^2\cdots w_1(E_{n_3})^2\cdots w_1(E_{n_k+1})^k\cdots w_1(E_{n_{k+1}})^k
\] 
divides the monomial
\[
w_1(E_1)^0w_1(E_2)^1w_1(E_3)^2\cdots w_1(E_d)^{d-1}.
\]
Thus, in order to prove that $\alpha_{n_1,\dots,n_k}^*(u)\neq 0$ and consequently conclude $u\neq 0$ it suffices to show that
\begin{multline*}
	0\neq w_1(E_1)^0w_1(E_2)^1w_1(E_3)^2\cdots w_1(E_d)^{d-1}  \\
	\in \ H^*(\flag_{1,\dots,d-1} (V);\F_2)\cong 
	\F_2[w_1(E_1),w_1(E_2),\dots,w_{1}(E_d)]\,/\, I_{1,\dots,d-1}.
\end{multline*}
Recall that the ideal $I_{1,\dots,d-1}=\Big(\sigma_{r}(w_1(E_1),\dots,w_{1}(E_d))	: 1\leq r\leq d\Big)$ is generated by elementary symmetric polynomials.
Hence  
\[
w_1(E_1)^0w_1(E_2)^1 \cdots w_1(E_d)^{d-1}\neq 0   \Longleftrightarrow    w_1(E_{\pi(1)})^0w_1(E_{\pi(2)})^1 \cdots w_1(E_{\pi(d)})^{d-1}\neq 0
\]
for every permutation $\pi\in\Sym_d$.
For the sake of brevity we prove that 
\begin{equation}
	\label{eq : not zero}
	w_1(E_d)^0w_1(E_{d-1})^1 \cdots w_1(E_1)^{d-1}\neq 0
\end{equation}
in $H^*(\flag_{1,\dots,d-1}(V);\F_2)$.

\medskip
The proof of \eqref{eq : not zero} proceeds by induction as follows.
First, obviously $w_1(E_1)^{d-1}\neq 0$ in 
\[
H^*(\flag_1 (V);\F_2)\cong H^*(\PP(V);\F_2)\cong \F_2[x]/(x^d),
\]
where $x$ corresponds to $w_1(E_1)^{d-1}$.
Next, let $1\leq k\leq d-2$ and let 
\begin{equation}
	\label{eq : IHSW}
w_1(E_k)^{d-k}w_1(E_{k-1})^{d-k+1} \cdots w_1(E_1)^{d-1}\neq 0
\end{equation}
in $H^*(\flag_{1,\dots,k}\flag_{1,\dots,k} (V);\F_2)$.
Finally, the map
\[
\flag_{1,\dots,k+1} (V)\longrightarrow  \flag_{1,\dots,k} (V) ,
\]
given by
\[
0\subseteq V_1\subseteq \dots \subseteq V_k\subseteq V_{k+1} \  \longmapsto \ 0\subseteq V_1\subseteq \dots \subseteq V_k,
\]
is the projective bundle of the vector bundle $(E_1\oplus\cdots\oplus E_k)^{\perp}$ over the flag manifold $\flag_{1,\dots,k} (V)$, that is $\PP\big((E_1\oplus\cdots\oplus E_k)^{\perp}\big)$. 
From Lemma \ref{claim : 00} we have that
\begin{align*}
H^*(\flag_{1,\dots,k+1} (V);\F_2) &\cong H^*\big(\PP\big((E_1\oplus\cdots\oplus E_k)^{\perp}\big);\F_2\big) \\
&\cong H^*(\flag_{1,\dots,k} (V);\F_2)[x]/\Big(\sum_{s=0}^{d-k}w_{d-k-s} \ x^s\Big),
\end{align*}
where $w_{d-k-s}=w_{d-k-s}(-(E_1\oplus\cdots\oplus E_k))$ and $x=w_1(E_{k+1})$.
Thus, from assumption \eqref{eq : IHSW}, that is $w_1(E_d)^0w_1(E_{d-1})^1 \cdots w_1(E_1)^{d-1}\neq 0$ in $H^*(\flag_{1,\dots,k} (V);\F_2)$ we obtain
\[
w_1(E_{k+1})^{d-k-1}w_1(E_k)^{d-k}w_1(E_{k-1})^{d-k+1} \cdots w_1(E_1)^{d-1}\neq 0
\]
in $H^*(\flag_{1,\dots,k+1} (V);\F_2)$.
Consequently \eqref{eq : not zero} holds. 
This concludes the  argument and completes the proof of the corollary. \qed

\medskip
Let us also point out that the non-vanishing of the class $u$ can also be deduced using \cite[Rem.\,2.8]{Crabb2022}. 

%--------------------------------------------------------------------------------------%
\subsection{Proof of Corollary \ref{cor : E} }
\label{subsec : cor : flag 02}
%--------------------------------------------------------------------------------------%
 
 Like in the previous section we assume that $k\geq 1$ and $d\geq 2$ are integers, and that $0=n_0< n_1<\dots<n_k<n_{k+1}= d$ is a strictly increasing sequence of integers.
 We take $V=\R^d$ and denote by $E_1,\dots, E_{k+1}$ the canonical vector bundles over the flag manifold $\flag_{n_1,\dots,n_k} (V)$.
 Furthermore, $E(i):=\bigoplus_{1\leq r\leq i}E_r$ for all $1\leq i\leq k$, and  $E:=E(k)$.
 In addition, we assume  that $j=2^t+r$ is an integer, with $0\leq r\leq 2^t-1$, and $d=\dim(V)\geq 2^{t+k-1}+r+1$.

\medskip
In order to prove the existence of the desired partition we use Theorem \ref{th : main result 02}.
More precisely, if $j\leq \iota_k(E(1),\dots,E(k))$, then the theorem guarantees the existence of a point $b:=(W_1,\dots,W_{k+1})$ in the base space $\flag_{n_1,\dots,n_k} (V)$ of the vector bundle $E$ and an arrangement  $\HH^b=(H_1^b,\dots,H_k^b)$ of $k$ linear hyperplanes in the fiber $E_b$ such that for every pair of connected components $(\mathcal{O}',\mathcal{O}'')$ of the arrangement complement $E_b-(H_1^b\cup\dots\cup H_k^b)$ holds
\[
\int_{\mathcal{O}'\cap S(E_b)}\varphi_1=\int_{\mathcal{O}''\cap S(E_b)}\varphi_1
\quad , \dots , \quad 
\int_{\mathcal{O}'\cap S(E_b)}\varphi_j=\int_{\mathcal{O}''\cap S(E_b)}\varphi_j,
\]	
and in addition
\[
(H_1^b)^{\perp}\subseteq E(1)_b ,\ 
(H_2^b)^{\perp}\subseteq E(2)_b , \ \dots \ ,
(H_k^b)^{\perp}\subseteq E(k)_b .
\]
Since $E(i)_b=\bigoplus_{1\leq r\leq i}(E_r)_b=\bigoplus_{1\leq r\leq i}W_r$ for every $1\leq i\leq k$, we have that
\[
(H_i^b)^{\perp}\subseteq E(i)_b \ \Longrightarrow \ H_i^b\supseteq (E(i)_b)^{\perp}=\big(\bigoplus_{1\leq r\leq i}W_r\big)^{\perp}=\bigoplus_{i+1\leq r\leq k+1}W_r.
\]
Hence, for the the proof of Corollary \ref{cor : E} it suffices to verify that  $j=2^t+r\leq \iota_k(E(1),\dots,E(k))$ when $d=\dim(V)\geq 2^{t+k-1}+r$.

\medskip
We have from Corollary \ref{cor : B}  that $\iota_k(E(1),\dots,E(k))=\iota_k(d,\dots,d)$, so we need to show that 
\[
j=2^t+r\leq \iota_k(d,\dots,d)=\max\big\{ j' : e_k(\pt)^{j'}\notin (x_1^{d},\dots,x_k^{d})\big\}.
\]
Since  $d\geq 2^{t+k-1}+r$, using \cite[Lem.\,4.2]{BlagojevicCallesCrabbDimitrijevic}, we get that  $e_k(\pt)^{j}\notin (x_1^{d},\dots,x_k^{d})$, and consequently $j\leq \iota_k(d,\dots,d)$.
This completes the proof of the corollary. \qed

%--------------------------------------------------------------------------------------%
\subsection{Proof of Theorem \ref{cor : F} }
\label{subsec : cor : flag 03}
%--------------------------------------------------------------------------------------%

Fix integers $d\geq 1$ and $k\geq 1$ with $d\geq k$, and let $V=\R^{d+1}$.
Let $(j_k,\dots,j_d)$ be a permutation  of the set  $\{k,\dots, d\}$, and take an arbitrary collections of functions $\varphi_{a,b}\colon S(E_{a+1}^{d+1})\longrightarrow\R$, $k\leq a\leq d$, $1\leq b\leq j_a$, from the sphere bundle of the tautological vector bundle  $E_{a+1}^{d+1}$ over the Grassmann manifold $G_{a+1}(V)$ to the real numbers.

\medskip
According to Theorem \ref{th : CSTM for Fairy Bread SandwichFairy Bread Sandwich}, for the existence of the desired partition it suffices to prove the non-vanishing of the Euler class of the vector bundle 
\[
E=E_{k+1}^{\oplus j_k}\oplus E_{k+2}^{\oplus j_{k+1}}\oplus\cdots\oplus E_{d+1}^{\oplus j_d}.
\]
For this we show that the related mod $2$ Euler class which lives in the cohomology ring $H^*(\flag_{k,\dots,d}(V);\F_2)$ is not zero.
As already discussed at the beginning of Section \ref{sec : proof flag} we have that
\[
w(E)=(1+w_1(E_{k+1}))^{j_k}\cdots (1+w_1(E_{d+1}))^{j_d}
\] 
implying that the mod $2$ Euler class of $E$ is $\e(E)=w_1(E_{k+1})^{j_k}\cdots w_1(E_{d+1})^{j_d}$. 
Applying the map $\alpha_{k,\dots,d}^*$, with the usual abuse of notation we have that
\[
\alpha_{k,\dots,d}^*(\e(E))=w_1(E_{k+1})^{j_k}\cdots w_1(E_{d+1})^{j_d}\neq 0
\] 
in  $H^*(\flag_{1,\dots,d}(V);\F_2)$, according to \eqref{eq : not zero}.
Consequently, $\e(E)\neq 0$ and the proof of the theorem is complete.

%--------------------------------------------------------------------------------------%
%--------------------------------------------------------------------------------------%
%--------------------------------------------------------------------------------------%
%--------------------------------------------------------------------------------------%
\section{Proofs of Proposition \ref{prop : numbers} and \ref{prop : numbers2}}
%--------------------------------------------------------------------------------------%
%--------------------------------------------------------------------------------------%
%--------------------------------------------------------------------------------------%
%--------------------------------------------------------------------------------------%

In this section we verify properties of integers $\iota_k(m_1,\dots,m_k)$ stated in Proposition \ref{prop : numbers} and Proposition \ref{prop : numbers2}.

%--------------------------------------------------------------------------------------%
\subsection{Proof of Proposition \ref{prop : numbers}}
\label{subsec : prop : numbers} 
%--------------------------------------------------------------------------------------%

Let $k\geq 1$ be an integer and let $m_1,\dots,m_k$ be positive integers.
Recall that 
\[
\iota_k(m_1,\dots,m_k)=\max\big\{ j : e_k(\pt)^j\notin (x_1^{m_1},\dots,x_k^{m_k})\big\},
\]
where
\[
e_k(\pt)=\prod_{(\alpha_{1},\dots, \alpha_{k})\in\F_2^k-\{ 0\}} (\alpha_{1}x_1 + \cdots + \alpha_{k}x_k) \ \in \ R_{k}(\pt)\cong \F_2[x_1,\dots,x_k].
\]
We prove the claims in the order they are listed.

\medskip
\noindent
{\bf (1)}
Assume that $m_k\geq 2^{k-1}m+1$ and in addition that $\iota_{k-1}(	m_1,\dots,m_{k-1})\geq m$. 
Then $e_k(\pt)^m\notin (x_1^{m_1},\dots,x_{k-1}^{m_{k-1}})$.
We transform as follows
\begin{align*}
	e_k(\pt)^m & =e_{k-1}(\pt)^m \prod_{(\alpha_{1},\dots, \alpha_{k-1})\in\F_2^{k-1}} (\alpha_{1}x_1 + \cdots + \alpha_{k-1}x_{k-1}+x_k)^m\\
	&=e_{k-1}(\pt)^m\cdot x_k^{2^{k-1}m}+p_{2^{k-1}m-1}\cdot x_k^{2^{k-1}m-1}+\cdots+p_{1}\cdot x_k+p_0,
\end{align*}
where $p_{2^{k-1}m-1},\dots,p_1,p_0\in \F_2[x_1,\dots,x_{k-1}]$.
Consequently,
\[
e_k(\pt)^m\notin  (x_1^{m_1},\dots,x_{k-1}^{m_{k-1}}, x_k^{2^{k-1}m+1}).
\]
Since,  $m_k\geq 2^{k-1}m+1$ we have that $(x_1^{m_1},\dots,  x_k^{m_k})\subseteq (x_1^{m_1},\dots,  x_k^{2^{k-1}m+1})$ and thus $e_k(\pt)^m\notin  (x_1^{m_1},\dots,x_{k-1}^{m_{k-1}}, x_k^{m_k})$.
Therefore, $\iota_k(m_1,\dots,m_k)\geq m$, as claimed.

\medskip
\noindent
{\bf (2)} 
We prove the claim by induction on $k$. 
For $k=1$ we assume that $m_1\geq m+1$. 
Then 
\[
\iota_1(m_1)=\max\big\{ j : e_1(\pt)^j=x_1^j\notin (x_1^{m_1})\big\}=m_1-1\geq m.
\]
Now, assume that claim holds for $k-1\geq 1$, and assume in addition that $m_1\geq 2^{i-1}m+1$ for all $1\leq i\leq k$.
Then from the assumption $\iota_{k-1}(	m_1,\dots,m_{k-1})\geq m$, and consequently by part (1) of this claim it  follows that $\iota_k(m_1,\dots,m_k)\geq m$.
 
\medskip
\noindent
{\bf (3)}
In this case we have that $m_1=m+1,m_2=2m+1,\dots, m_k=2^{k-1}+1$. 
According to the part (2) of this claim, it follows that 
\[
\iota_k(m+1,2m+1,2^2m+1\dots,2^{k-1}m+1)\geq m.
\]

\medskip
Now, assume that  $\iota_k(m_1,\dots,m_k)\geq r\geq 1$ for some sequence of positive integers $m_1,\dots,m_k$.
Hence, $e_k(\pt)^r \notin  (x_1^{m_1},\dots,  x_k^{m_k})$.
We expand the transformation from the proof of part (1) of this claim as follows:
\begin{align*}
	e_k(\pt)^r & =e_{k-1}(\pt)^r \prod_{(\alpha_{1},\dots, \alpha_{k-1})\in\F_2^{k-1}} (\alpha_{1}x_1 + \cdots + \alpha_{k-1}x_{k-1}+x_k)^r\\
	&= e_{k-2}(\pt)^r \cdot \\
	&\hspace{18pt} \prod_{(\alpha_{1},\dots, \alpha_{k-2})\in\F_2^{k-2}} (\alpha_{1}x_1 + \cdots + x_{k-1})^r\prod_{(\alpha_{1},\dots, \alpha_{k-1})\in\F_2^{k-1}} (\alpha_{1}x_1 + \cdots + x_k)^r\\
	&\hspace{153pt} \dots \\
	&=  x_k^{2^{k-1}r}x_{k-1}^{2^{k-2}r}\cdots x_2^{2r}x_1^r+ q.
\end{align*}
Here  $q$ is a polynomial whose additive representation in the monomial basis does not contain the monomial $x_k^{2^{k-1}r}x_{k-1}^{2^{k-2}r}\cdots x_2^{2r}x_1^r$. 
Since, $e_k(\pt)^r \notin  (x_1^{m_1},\dots,  x_k^{m_k})$ we conclude that
\[
m_k\geq 2^{k-1}r+1, m_{k-1}\geq 2^{k-2}r+1, \dots , m_1\geq r+1,
\]
implying that
\[
m_k+m_{k-1}+\cdots+m_2+m_1\geq (2^{k-1}+2^{k-2}+\cdots+2+1)r+k.
\]
In particular, 
\[
m_k+m_{k-1}+\cdots+m_2+m_1\geq (2^{k}-1)\iota_k(m_1,\dots,m_k)+k.
\]
Thus, in the case when $m_1=m+1,m_2=2m+1,\dots, m_k=2^{k-1}+1$, we have that 
\[
(2^{k}-1)m+k\geq (2^{k}-1)\iota_k(m+1,2m+1,2^2m+1\dots,2^{k-1}m+1)+k,
\]
or in other words $m\geq \iota_k(m+1,2m+1,2^2m+1\dots,2^{k-1}m+1)$.

\medskip
Hence, we showed that $\iota_k(m+1,2m+1,2^2m+1\dots,2^{k-1}m+1)=m$, as claimed.

\medskip
\noindent
{\bf (4)}
We start with the following transformation
\begin{align*}
	e_k(\pt)^m & =e_{k-r}(\pt)^m \prod_{(\alpha_{k-r+1},\dots, \alpha_{k})\in\F_2^{r}-\{0\}} \prod_{(\alpha_{1},\dots, \alpha_{k-r})\in\F_2^{k-r}}(\alpha_{1}x_1 + \cdots + \alpha_{k}x_k)^m\\
	&= e_{k-r}(\pt)^m \prod_{(\alpha_{k-r+1},\dots, \alpha_{k})\in\F_2^{r}-\{0\}}\prod_{(\alpha_{1},\dots, \alpha_{k-r})\in\F_2^{k-r}}  \\
	&\hspace{58pt}\big((\alpha_{1}x_1 + \cdots +\alpha_{k-r}x_{k-r}) +(\alpha_{k-r+1}x_{k-r+1} +\cdots +\alpha_{k}x_k)\big)^m.\\
\end{align*}
Hence,
\[
	e_k(\pt)^m =\underbrace{e_{k-r}(\pt)^m \prod_{(\alpha_{k-r+1},\dots, \alpha_{k})\in\F_2^{r}-\{0\}}\big(\alpha_{k-r+1}x_{k-r+1} +\cdots +\alpha_{k}x_k\big)^{2^{k-r}m}}_{=p}+ q,
\]
where the sets of (non-zero) monomials in the additive presentations of the polynomials $p$ and $q$ are disjoint.

\medskip
The assumptions $\iota_{k-r}(m_1,\dots,m_{k-r})\geq m$ and $\iota_{r}(m_{k-r+1},\dots,m_{k})\geq 2^{k-r}m$ imply that 
\[
e_{k-r}(\pt)^m \notin \big(x_1^{m_1},\dots, x_{k-r}^{m_{k-r}}\big)
\]
and
\[
\prod_{(\alpha_{k-r+1},\dots, \alpha_{k})\in\F_2^{r}-\{0\}}\big(\alpha_{k-r+1}x_{k-r+1} +\cdots +\alpha_{k}x_k\big)^{2^{k-r}m}\notin \big(x_{k-r+1}^{m_{k-r+1}},\dots,x_k^{m_k}\big).
\]
Therefore, the polynomial $p$ is the witness that 
\[
e_k(\pt)^m \notin \big(x_1^{m_1},\dots, x_{k-r}^{m_{k-r}},x_{k-r+1}^{m_{k-r+1}},\dots,x_k^{m_k}\big),
\]
and consequently $\iota_{k}(	m_1,\dots,m_{k})\geq m$, as claimed

\medskip
\noindent
{\bf (5)}
The polynomial $e_k(\pt)^m$ can be presented as follows: 
\[
	e_k(\pt)^m =e_{k-1}(\pt)^m x_k^m \prod_{(\alpha_{1},\dots, \alpha_{k-1})\in\F_2^{r}-\{0\}} (\alpha_{1}x_1 + \cdots + \alpha_{k-1}x_{k-1}+ x_k)^m.
\]
Hence the lowest power of $x_k$ in $e_k(\pt)^m$ is $x_k^m$ with coefficient $e_{k-1}(\pt)^{2m}$.

\medskip
The assumption  $\iota_{k-1}(m_1,\dots,m_{k-1})\geq 2m$ implies that 
\[
e_{k-1}(\pt)^{2m}\notin (x_1^{m_1},\dots, x_{k-1}^{m_{k-1}}),
\]
and since $m_k\geq m+1$ it follows that $e_k(\pt)^m \notin (x_1^{m_1},\dots, x_{k-1}^{m_{k-1}},x_k^{m_k})$.

\medskip
\noindent
{\bf (6)}
In the case when $k=2$ we have that
\begin{equation}
	\label{eq : k=2}
	e_2(\pt)^m=\big( x_1x_2(x_1+x_2)\big)^m=\sum_{i=0}^m {m\choose i}x_1^{m+i}x_2^{2m-i}.
\end{equation}

\medskip
If  $m\leq \iota(m_1,m_2)$ then $e_2(\pt)^m\notin (x_1^{m_1},x_2^{m_2})$.
Hence, there exists a non-zero monomial ${m\choose i}x_1^{m+i}x_2^{2m-i}$ in presentation \eqref{eq : k=2} of $e_2(\pt)^m$ which does not belong to the ideal $(x_1^{m_1},x_2^{m_2})$.
This means, ${m\choose i}=1\mod 2$, $m+i\leq m_1-1$ and $2m-i\leq m_2-1$ for some integer $0\leq i\leq m$.

\medskip
Assume the opposite, that there is an integer $0\leq i\leq m$ such that ${m\choose i}=1\mod 2$ and $2m-m_2+1\leq i\leq m_1-m-1$.
Then the polynomial $e_2(\pt)^m$ when presented in the monomial basis has non-zero monomial ${m\choose i}x_1^{m+i}x_2^{2m-i}$ which does not belong to the ideal $(x_1^{m_1},x_2^{m_2})$.
Consequently, $e_2(\pt)^m\notin (x_1^{m_1},x_2^{m_2})$.

\medskip
\noindent
{\bf (7)}
This is a direct consequence of the previous claim with $m=2^t+r-1$, $m_1=2^t+2t$, $m_2=2^{t+1}+r$ and $i=r-1$ because  ${2^t+r-1\choose r-1}=1\mod 2$, and
\[
2m-m_2+1= r-1\leq i=r-1\leq m_1-m-1=r.
\]

\medskip
We completed the proof of the proposition. \qed

%--------------------------------------------------------------------------------------%
\subsection{Proof of Proposition \ref{prop : numbers2}}
\label{subsec : prop : numbers2} 
%--------------------------------------------------------------------------------------%
As before, $k\geq 1$ is an integer and $m_1,\dots,m_k$ are positive integers.
In the proof we use the fact that the polynomial $e_k(\pt)$ is the top Dickson polynomial in variables $x_1,\dots,x_{k}$. 
For more details on Dickson polynomials see for example \cite{Wilkerson1983}.

\medskip
\noindent
{\bf (1)}
Let $D_{k-1},D_{k-2},\dots,D_1$ be Dickson polynomials in variables $x_1,\dots,x_{k-1}$ of degree $2^{k-1}-1,2^{k-1}-2,\dots, 2^{k-1}-2^{k-2}$, respectively.
In particular, $D_{k-1}=e_{k-1}(\pt)$.
From \cite[Prop.\,1.1]{Wilkerson1983} we have that 
\begin{align*}
	D(x_k) &:= \prod_{(\alpha_{1},\dots, \alpha_{k-1})\in\F_2^{k-1}-\{0\}} (\alpha_{1}x_1 + \cdots + \alpha_{k-1}x_{k-1}+x_k)\\
	&=x_k^{2^{k-1}-1}+D_1\,x_k^{2^{k-2}-1}+\dots+D_i\,x_k^{2^{k-1-i}-1}+\dots+D_{k-2}\,x_k+D_{k-1}.
\end{align*}
Here $D(x_k)$ is considered a polynomial in $\F_2[x_1,\dots,x_{k-1}][x_k]$, and furthermore $e_k(\pt)=e_{k-1}(\pt)x_kD(x_k)$.

\medskip
Let  $0\leq r\leq 2^t-1$.
We compute in  $\F_2[x_1,\dots,x_{k-1}][x_k]$ as follows:
\begin{align}\label{eq : prod01}
	D(x_k)^{2t+r} &= \Big( x_k^{2^{k-1}-1}+\dots+D_i\,x_k^{2^{k-1-i}-1}+\dots+D_{k-2}\,x_k+D_{k-1}\Big)^{2t+r}\nonumber \\ 
	&=  \Big( x_k^{2^t(2^{k-1}-1)}+\dots+D_i^{2^t}\,x_k^{2^t(2^{k-1-i}-1)}+\dots+D_{k-2}^{2^t}\,x_k^{2^t}+D_{k-1}^{2^t}\Big)\cdot \nonumber\\
	& \hspace{14pt} \Big( x_k^{2^{k-1}-1}+\dots+D_i\,x_k^{2^{k-1-i}-1}+\dots+D_{k-2}\,x_k+D_{k-1}\Big)^{r}.
\end{align}
Then, the coefficient of $x_k^{2^t(2^{k-1}-1)}$ in $D(x_k)^{2t+r}$ is $D_{k-1}^r=e_{k-1}(\pt)^r$, obtained as the product of $x_k^{2^t(2^{k-1}-1)}$ from the first factor with $D_{k-1}^r$ from the second factor in \eqref{eq : prod01}.
Indeed, the only other candidate which might additionally contribute to the coefficient of  $x_k^{2^t(2^{k-1}-1)}$ is the product
\[
D_1^{2^t}x_k^{2^t(2^{k-2}-1)}\cdot x_k^{r(2^{k-1}-1)}=D_1^{2^t}x_k^{2^t(2^{k-2}-1)+r(2^{k-1}-1)}
\]
when
\[
2^t(2^{k-1}-1)=2^t(2^{k-2}-1)+r(2^{k-1}-1) \ \Longleftrightarrow \ 2^{t+k-2}=r(2^{k-1}-1).
\]
This cannot be because $0\leq r\leq 2^t-1$.
Consequently, the coefficient of $x_k^{2^{k-1+t}+r}$ in 
\[
e_k(\pt)^{2t+r}=e_{k-1}(\pt)^{2t+r}x_k^{2t+r}D(x_k)^{2t+r}
\]
is equal to $e_{k-1}(\pt)^{2t+2r}$.

\medskip
From the assumption $\iota_{k-1}(m_1,\dots,m_{k-1})\geq 2^t+2r$ we have that
\[
e_{k-1}(\pt)^{2t+2r}\notin (x_1^{m_1},\dots,x_{k-1}^{m_{k-1}}),
\]
and since $m_k\geq 2^{t+k-1}+r+1$ we conclude that
\[
e_k(\pt)^{2t+r} \notin (x_1^{m_1},\dots,x_{k-1}^{m_{k-1}},x_k^{m_k}).
\]
Thus, $\iota_{k}(m_1,\dots,m_{k})\geq 2^t+r$ as claimed.

\medskip
\noindent
{\bf (2)}
The claim follows from the previous instance of the proposition because
\[
2^{t+1}+r=2^{t}+r+2^t >2^{t}+r+r= 2^{t}+2r.
\]

\medskip
\noindent
{\bf (3)}
The proof is by induction on $k$ for every pair of integers $(2^t,r)$ with $1\leq r\leq 2^t-1$.
In the case $k=1$,  the assumption $m_1 \geq 2^{t}+r+1$ implies that
\[
\iota_1(m_1)=m_1-1\geq 2^{t}+r+1-1= 2^{t}+r.
\]
Let us assume that the claim holds for $k-1\geq 1$ and every every pair of integers $(2^t,r)$ with $1\leq r\leq 2^t-1$ (the induction hypothesis).
Take $m_i\geq 2^{t+k-1}+r+1=2^{(t+1)+(k-1)-1}+r+1$ for all $1\leq i\leq k$.
Applying the  induction hypothesis to the first $k-1$ inequalities and the pair $(2^{t+1},r)$ we get that
\[
\iota_{k-1}(m_1,\dots,m_k)\geq 2^{t+1}+r.
\]
Now, the inequalities $\iota_{k-1}(m_1,\dots,m_k)\geq 2^{t+1}+r+1$ and $m_k\geq 2^{t+k-1}+r+1$, and the previous claim of this proposition imply that $\iota_{k}(m_1,\dots,m_{k})\geq 2^t+r$.
This completes the proof.

\medskip
\noindent
{\bf (4)}
Since $e_k(\pt)^2=\prod_{(\alpha_{1},\dots, \alpha_{k})\in\F_2^k-\{ 0\}} (\alpha_{1}x_1^2 + \cdots + \alpha_{k}x_k^2)$, the following equivalence holds
\[
e_k(\pt)^{2m}\in  (x_1^{2m_1},\dots, x_k^{2m_k})
\ \Longleftrightarrow \ 
e_k(\pt)^{m}\in  (x_1^{m_1},\dots, x_k^{m_k}).
\]
This equivalence implies the claim.  \qed

%--------------------------------------------------------------------------------------%
%--------------------------------------------------------------------------------------%
%--------------------------------------------------------------------------------------%
%--------------------------------------------------------------------------------------%
\section{Proof of Theorem \ref{th : main result 03}}
\label{subsec : proof of main result 03}
%--------------------------------------------------------------------------------------%
%--------------------------------------------------------------------------------------%
%--------------------------------------------------------------------------------------%
%--------------------------------------------------------------------------------------%

Let $E$ be a Euclidean vector bundle of dimension $n$ over a compact and connected ENR $B$, and let the integers $1\leq k\leq n$ and $j\geq 1$ be fixed.
We first prove the equality of the ideals and then a criterion for the existence of orthogonal partitions.

\subsection{Proof of Part\,(1)}
We prove the equality of the ideals
 \begin{equation}
 	\label{eq - ideals}
 	 \mathcal{J}_k(E):=(f_1,\dots,f_k)=(\overline{f}_1,\dots,\overline{f}_k)=: \mathcal{J}_k'(E)
 \end{equation}
where 
\[
f_i:=\sum_{0\leq r_1+\dots+r_i\leq n-i+1}w_{n-i+1-(r_1+\dots+r_i)}(E)\, x_1^{r_1}\cdots x_i^{r_i},
\]
and
\[
\overline{f}_i:=\sum_{0\leq r_1+\dots+r_k\leq n-i+1}w_{n-i+1-(r_1+\dots+r_k)}(E)\, x_1^{r_1}\cdots x_k^{r_k},
\]
for $1\leq i\leq k$.

\medskip
To prove the equality of the ideal we first consider the polynomials 
\[
X_a[b]:=\sum_{r_1+\dots+r_b=n-a+1}x_1^{r_1}\cdots x_b^{r_b}
\]
for $1\leq a\leq n+1$ and $1\leq b\leq k$.
It is straightforward to see that the following equality holds
\begin{equation}
	\label{eq - induction -1}
X_a[b+1]=X_a[b]+x_{b+1}\cdot X_{a+1}[b+1].
\end{equation}
Indeed, we have that
\begin{multline*}
\hspace{-8pt}X_a[b+1]:=\sum_{r_1+\dots+r_b+r_{b+1}=n-a+1}x_1^{r_1}\cdots x_b^{r_b}x_{b+1}^{r_{b+1}} 	=   \\ 
\hspace{62pt} \sum_{r_1+\dots+r_b+0=n-a+1}x_1^{r_1}\cdots x_b^{r_b}x_{b+1}^{0} +x_{b+1}\sum_{r_1+\dots+r_{b+1}=n-a}x_1^{r_1}\cdots x_{b+1}^{r_{b+1}}= \\ \hspace{210pt} X_a[b+1]=X_a[b]+x_{b+1}\cdot X_{a+1}[b+1].
\end{multline*}

Next, using induction on $\ell\geq 0$, we prove the following identity:
\begin{equation}
	\label{eq - induction}
	X_{c+s}[c+\ell]=\sum_{c\leq b\leq c+\ell}\Big(\sum_{s_b+\cdots+s_{c+\ell}=b-c} x_{b}^{s_b}\cdots x_{c+\ell}^{s_{c+\ell}}\Big)X_{b+s}[b].
\end{equation}
In case when $\ell=0$ the equality \eqref{eq - induction} becomes the identity $X_{c+s}[c]=X_{c+s}[c]$, an so the induction basis is verified.
Now, we assume that the equality \eqref{eq - induction} holds for the given fixed integer $\ell\geq 1$.
For the induction step we compute and use induction hypothesis as follows:
\begin{multline*}
X_{c+s}[c+\ell+1] \overset{\eqref{eq - induction -1}}{=} X_{c+s}[c+\ell]+x_{c+\ell+1}\cdot X_{c+s+1}[c+\ell+1] \overset{\eqref{eq - induction}}{=}  \\
 \sum_{c\leq b\leq c+\ell}\Big(\sum_{s_b+\cdots+s_{c+\ell}=b-c} x_{b}^{s_b}\cdots x_{c+\ell}^{s_{c+\ell}}\Big)X_{b+s}[b]+x_{c+\ell+1}\cdot X_{c+s+1}[c+\ell+1]  =\\
  \sum_{c\leq b\leq c+\ell}\Big(\sum_{s_b+\cdots+s_{c+\ell}=b-c} x_{b}^{s_b}\cdots x_{c+\ell}^{s_{c+\ell}}\Big)X_{b+s}[b]+\\
  x_{c+\ell+1}\cdot \sum_{s_1+\dots+s_{c+\ell+1}=n-c-s}x_1^{s_1}\cdots x_{c+\ell+1}^{s_{c+\ell+1}}.
 \end{multline*}
Gathering two terms on the right hand since of the previous equality under one sum we get that
\[
X_{c+s}[c+\ell+1] =
\sum_{c\leq b\leq c+\ell+1}\Big(\sum_{s_b+\cdots+s_{c+\ell}=b-c} x_{b}^{s_b}\cdots x_{c+\ell+1}^{s_{c+\ell+1}}\Big)X_{b+s}[b].
\] 
This completes the induction and the proof of the relation \eqref{eq - induction}.

\medskip
We proceed with a proof of the equality \eqref{eq - ideals}.
Observe that for $1\leq i\leq k$:
\[
f_i=\sum_{0\leq s\leq n-i+1}w_s(E)X_{s+i}[i]
\qquad\text{and}\qquad
\overline{f}_i=\sum_{0\leq s\leq n-i+1}w_s(E)X_{s+i}[k],
\]
and in particular that $f_k=\overline{f}_k$.

\medskip
Now, using the relation \eqref{eq - induction} we have that
\begin{multline*}
	\overline{f}_i=\sum_{0\leq s\leq n-i+1}w_s(E)X_{s+i}[k] \overset{\eqref{eq - induction}}{=}\\
	\sum_{0\leq s\leq n-i+1}w_s(E)\Big(\sum_{i\leq b\leq k}\Big(\sum_{s_b+\dots+s_k=b-i}x_b^{s_b}\cdots x_{k}^{s_k}\Big)X_{s+b}[b]\Big)=\hspace{15pt}
	\\ \hspace{25pt}
	\sum_{i\leq b\leq k}\Big(\sum_{s_b+\dots+s_k=b-i}x_b^{s_b}\cdots x_{k}^{s_k}\Big)\Big(\sum_{0\leq s\leq n-i+1}w_s(E)X_{s+b}[b]\Big)=\\
	\sum_{i\leq b\leq k}\Big(\sum_{s_b+\dots+s_k=b-i}x_b^{s_b}\cdots x_{k}^{s_k}\Big)f_b.
\end{multline*}
In summary,
\begin{equation}
	\label{eq : fi}
	\overline{f}_r=\sum_{i\leq b\leq k}\Big(\sum_{s_b+\dots+s_k=b-i}x_b^{s_b}\cdots x_{k}^{s_k}\Big)f_b.
\end{equation}
Hence, $(\overline{f}_1,\dots,\overline{f}_k)\subseteq (f_1,\dots,f_k)$. 

\medskip
On the other hand, since $f_k=\overline{f}_k$ we have that $f_k\in \mathcal{J}_k'(E)=(\overline{f}_1,\dots,\overline{f}_k)$.
Now, for $1\leq r\leq k-1$ assume that $f_{r+1},\dots,f_k\in \mathcal{J}_k'(E)$.
Then from the equality \eqref{eq : fi} it follows that
\begin{multline*}
\overline{f}_r=\sum_{r\leq b\leq k}\Big(\sum_{s_b+\dots+s_k=b-r}x_b^{s_b}\cdots x_{k}^{s_k}\Big)f_b=\\f_r+\sum_{r+1\leq b\leq k}\Big(\sum_{s_b+\dots+s_k=b-r}x_b^{s_b}\cdots x_{k}^{s_k}\Big)f_b,
\end{multline*}
and consequently, by assumption, we have 
\[
f_r=\overline{f}_r+\sum_{r+1\leq b\leq k}\Big(\sum_{s_b+\dots+s_k=b-r}x_b^{s_b}\cdots x_{k}^{s_k}\Big)f_b\ \in \ \mathcal{J}_k'(E).
\]
Thus, $(\overline{f}_1,\dots,\overline{f}_k)\supseteq (f_1,\dots,f_k)$.

\medskip
We have completed the proof of the equality \eqref{eq - ideals}.

\subsection{Proof of Part\,(2)}
For the second part of the theorem assume that the class $e_{k}(B)^j$ does not belong to the ideal $\mathcal{J}_{k}(E)$.
The proof relies on the criterion from Theorem \ref{th : CSTM for orthogonal assignement problem}.
In other words, it suffices to prove that 
\[
 \e \big(\big( B_k(E)/\underline{\R}\big)^{\oplus j} \big)\neq 0.
\]
The mod $2$ Euler class of the vector bundle $\big(B_k(E)/\underline{\R}\big)^{\oplus j}$, or in other words the top Stiefel--Whitney class, lives in the cohomology of $H^*(Y_k(E);\F_2)$.
We show that
\begin{compactitem}[\quad --]
\item  $H^*(Y_k(E);\F_2)\cong R_{k}(B)/\mathcal{J}_{k}(E)$, and that
\item $w_{(2^k-1)j} \big(\big( B_k(E)/\underline{\R}\big)^{\oplus j} \big)=e_{k}(B)^j+\mathcal{J}_{k}(E)\in  R_{k}(B)/\mathcal{J}_{k}(E)$.
\end{compactitem}
The second claim follows from the first claim, the fact that $B_k(E)$ is the restriction of $A_k(E)$, and the related computation of $w_{(2^k-1)j} \big(\big( A_k(E)/\underline{\R}\big)^{\oplus j} \big)$ in the proof of Theorem \ref{th : CSTM for assignement problem}.
Thus we need to prove only the first statement, that is to compute the cohomology ring  $H^*(Y_k(E);\F_2)$.

\medskip
First, we give a description of the space $Y_k(E)$ as a projective bundle at the end of the tower of projective bundles
\begin{equation}
\label{eq ; tower}
\xymatrix{
Y_k(E)=\PP(E_k)\ar[r]^-{p_k} & \PP(E_{k-1})\ar[r]^-{p_{k-1}} & \ \cdots \ \ar[r]^-{p_2} &\PP(E_1)\ar[r]^-{p_1} & B,
}	
\end{equation}
where $E_1:=E$ and $p_1$ is the projection.
The vector bundles $E_{2},\dots,E_{k}$ and the maps $p_{2},\dots,p_{k}$ are defined iteratively as follows.

\medskip
Let $H(E_1)$ be the Hopf line bundle over $\PP(E_1)$, and reacall that  $p_1\colon \PP(E_1)\longrightarrow B$ is the projection map.
Then $H(E_1)$ is a vector subbundle of the pull-back vector bundle $p_1^*E_1$, and we set
\[
E_2:= H(E_1)^{\perp}
\]
to be the orthogonal complement of $H(E_1)$ inside $p_1^*E_1$.
In particular, $E_2$ is a $(n-1)$-dimensional vector bundle over $\PP(E_1)$.
Set $p_2\colon \PP(E_2)\longrightarrow \PP(E_1)$ to be the projection map.

\medskip
Next, $H(E_2)\oplus p_1^*H(E_1)$ is a vector subbundle of the pull-back vector bundle $(p_2\circ p_1)^*E_1$, and so we define
\[
E_3:= \big( H(E_2)\oplus p_1^*H(E_1)\big)^{\perp},
\]
and $p_3$ to be the projection map  $\PP(E_3)\longrightarrow \PP(E_2)$.

\medskip
We continue in the same way.
Assume that for $1\leq i\leq k-1$, all the vector bundles $E_1,\dots,E_i$, of dimensions $n,n-1,\dots,n-i+1$, respectively,  and the projection maps $p_1,\dots,p_i$ are defined.
Notice that 
\[
H(E_i)\oplus p_i^*H(E_{i-1})\oplus (p_i\circ p_{i-1})^*H(E_{i-1})\oplus\cdots\oplus (p_i\circ\cdots\circ p_{1})^*H(E_{1})
\] 
is a vector subbundle of $(p_i\circ\cdots\circ p_{1})^*E_{1}$.
We define the vector bundle $E_{i+1}$ as the orthogonal complement
\begin{equation}
\label{eq : def of Es}	
E_{i+1}:=\Big(H(E_i)\oplus p_i^*H(E_{i-1})\oplus  \cdots\oplus (p_i\circ\cdots\circ p_{1})^*H(E_{1})\Big)^{\perp}.
\end{equation}
The map $p_{i+1}$ is defined to be the standard projection $\PP(E_{i+1})\longrightarrow\PP(E_i)$.
It is clear that $Y_k(E)=\PP(E_k)$.

\medskip
Now, we use the tower of projective bundles \eqref{eq ; tower}, Lemma \ref{claim : 00}, as well as the proof of Claim \ref{claim : 01}, to describe the cohomology ring  $H^*(Y_k(E);\F_2)=H^*(\PP(E_k);\F_2)$. 

\medskip
Since  $H^*(Y_k(E);\F_2)=H^*(\PP(E_k);\F_2)$ where $\PP(E_k)$ is the projective bundle of the $(n-k+1)$-dimensional vector bundle $E_k$ over $\PP(E_{k-1})$ from Lemma \ref{claim : 00} we have that
\[
H^*(Y_k(E);\F_2) 	\cong H^*(\PP(E_{k-1});\F_2)[x_k]/\Big(\sum_{s = 0}^{n-k+1}   w_{n-k+1-s}(E_{k})\,x_k^{s}\Big),
\]
where $x_k$ corresponds to mod $2$ Euler class of the Hopf line bundle $H(E_{k})$.
Continuing to apply  Lemma \ref{claim : 00} for projective bundles $\PP(E_{k-1}),\dots, \PP(E_{1})$ we get the following conclusion
\begin{multline}
\label{eq : cohomology of Yk}
	H^*(Y_k(E);\F_2) \cong \\ H^*(B;\F_2)[x_1,\dots,x_k]/\Big(\sum_{s = 0}^{n}   w_{n-s}(E_{1})\,x_1^{s},\dots,\sum_{s = 0}^{n-k+1}   w_{n-k+1-s}(E_{k})\,x_k^{s}\Big).
\end{multline}
Here $x_i$, for all $1\leq i\leq k$, with a bit of abuse of notation, corresponds to the  mod $2$ Euler class of the Hopf line bundle $H(E_i)$, or more precisely to the mod $2$ Euler class of the pull-back line bundle $(p_k\circ\cdots\circ p_{i+1})^*H(E_{i})$.
Set $f_i:=\sum_{s = 0}^{n-i+1}   w_{n-i+1-s}(E_{i})\,x_i^{s}$ for $1\leq i\leq k$.
Then  
\[
H^*(Y_k(E);\F_2)\cong H^*(B;\F_2)[x_1,\dots,x_k]/\big(f_1,\dots,f_k\big).
\]

\medskip
Now we focus on identification of Stiefel--Whitney classes of the vector bundles  $E_{1},\dots, E_k$ in terms of the Stiefel--Whitney classes $E$.
Note that $E_1=E$ by definition, and so $w(E_1)=w(E)$.
Next, from the definition \eqref{eq : def of Es} of the vector bundles $E_{i}$ for $2\leq i\leq k$, as an orthogonal complements, we get that
\begin{align*}
w(E_i)&=w\Big(-\big(H(E_{i-1})\oplus p_{i-1}^*H(E_{i-2})\oplus  \cdots\oplus (p_{i-1}\circ\cdots\circ p_{1})^*H(E_{1})\big)\Big)	\\
&=w\big(-H(E_{i-1})\big)\cdot w\big(-p_{i-1}^*H(E_{i-2}) \big)\cdots  w\big(-(p_{i-1}\circ\cdots\circ p_{1})^*H(E_{1})\big).
\end{align*}
From Lemma  \ref{claim : 00} we also know that 
\[
w(H(E_{i-1}))=1+x_{i-1},\ \dots \ ,  w((p_{i-1}\circ\cdots\circ p_{1})^*H(E_{1}))=1+x_1.
\]
Here we assume the expected identifications of classes $x_1,\cdots, x_{i-1}$ along the sequence of isomorphisms given in Lemma \ref{claim : 00}.
Combining last two observations we have that 
\[
w(E_i)=\frac{1}{1+x_{i-1}}\cdot \frac{1}{1+x_{i-2}}\cdots \frac{1}{1+x_1}=\sum_{r_{i-1}\geq 0}x_{i-1}^{r_{i-1}}\cdot \sum_{r_{i-2}\geq 0}x_{i-2}^{r_{i-2}} \ \cdots \ \sum_{r_{1}\geq 0}x_{1}^{r_{1}},
\]
for $2\leq i\leq k$.
Consequently, we have that 
\begin{multline*}
	f_i=\sum_{s = 0}^{n-i+1}   w_{n-i+1-s}(E_{i})\,x_i^{s}=\\ \sum_{0\leq r_1+\dots+r_i\leq n-i+1}w_{n-i+1-(r_1+\dots+r_i)}(E)\, x_1^{r_1}\cdots x_i^{r_i}
\end{multline*}
for every $1\leq i\leq k$.

\medskip
This finishes the proof of the second claim, and so the proof of Theorem \ref{th : CSTM for orthogonal assignement problem} is complete.

\subsection{Proof of Proposition \ref{prop : main result 03}}

Let $E$ be a Euclidean vector bundle of dimension $n$ over a compact and connected ENR $B$, and let $k\geq 1$ and $j\geq 1$ be integers.

\medskip
Consider the composition inclusion
\[
Y_k(E)\lhook\joinrel\longrightarrow X_k(E)\lhook\joinrel\longrightarrow X_k(E\oplus\underline{\R} ).
\]
The image, $Y_k(E)$, can be seen as the zero-set of the section $s$ of the vector bundle $A_k(E\oplus\underline{\R})/\underline{\R}$ which is defined as follows. 

\medskip
The fibre of $A_k(E\oplus\underline{\R})/\underline{\R}$ over the point $(b,(L_1,\ldots ,L_k))\in  X_k(E\oplus\underline{\R} )$ decomposes into the direct sum
\[
\big( \bigoplus_{1\leq i\leq k}L_i\big)\oplus \big(\bigoplus_{1\leq i<j\leq k} L_i\otimes L_j)\oplus\cdots
\]
For every $1\leq i\leq k$ denote by $a_i$ the dual of the (linear) projection map given by the composition
\[
L_i \lhook\joinrel\longrightarrow E_b\oplus\R \longrightarrow \R.
\]
Similarly, for $1\leq i<j\leq k$ we set $a_{i,j}'$ to be the dual of the (linear) map induced by the inner product
\[
L_i\otimes L_j \lhook\joinrel\longrightarrow (E_b\oplus\R )\otimes (E_b\oplus\R)\longrightarrow\R.
\] 
Now, define $s$ by
$
(b,(L_1,\ldots ,L_k)) \longmapsto ((a_i)_{1\leq i\leq k}, (a_{i,j}')_{1\leq i<j\leq k}, 0, \dots, 0)
$.
Hence, the zero-set of the section $s$ is indeed $Y_k(E)$.
Additionally, the vector bundle $A_k(E\oplus\underline{\R })/\underline{\R}$ over 
$X_k(E\oplus\underline{\R })$ restricts to the vector bundle $B_k(E)/\underline{\R}$ over $Y_k(E)$.

\medskip
Consequently, if the Euler class of $(A_k(E\oplus\underline{\R })/\underline{\R})^{j+1}$ is non-zero, then the Euler class of $(B_k(E)/\underline{\R})^j$ is non-zero.
Indeed, see for example \cite[Prop.\,2.7]{Crabb2013}, which says that if $x$ is any class in the cohomology of $X_k(E\oplus\underline{\R})$ that restricts to zero in the cohomology of the zero-set, in this case $Y_k(E)$, then the product of $x$ with the Euler class of $A_k(E\oplus\underline{\R })/\underline{\R}$ is zero.
This concludes the proof of the proposition.

%--------------------------------------------------------------------------------------%
%--------------------------------------------------------------------------------------%
%--------------------------------------------------------------------------------------%
%--------------------------------------------------------------------------------------%
\section{Even more main results}
\label{sec : Even more main results}
%--------------------------------------------------------------------------------------%
%--------------------------------------------------------------------------------------%
%--------------------------------------------------------------------------------------%
%--------------------------------------------------------------------------------------%

In this section, we use methods developed in previous sections to give new proofs and generalise results of Larry Guth \& Nets Hawk Katz \cite{GuthKatz2015}, Blagojevi\'c, Dimitrijevi\'c Blagojevi\'c \& G\"unter M. Ziegler \cite{Blagojevic2017-Polynomial}, Schnider \cite{Schnider2019}, and Sober\'on \& Yuki Takahashi \cite{SoberonTakahashi}.

\medskip
Throughout this section $B$ will be a compact, connected ENR, and $E$ will be a Euclidean real vector bundle of dimension $n$ over $B$.
For an integer $k\geq 1$, $E(1),\dots, E(k)$ will be finite-dimensional non-zero real vector bundles over $B$ with $\dim E(i)=n_i$.
As before, we write $S(E(i))$ for the sphere bundle of $E(i)$ with fibre at $b\in B$ the space of oriented $1$-dimensional subspacesof $E(i)_b$. 
Equivalently, $S(E(i))$ is the unit sphere bundle for a chosen Euclidean structure.
Also, we shall use $V$ for a Euclidean vector space $V$, and sometimes see it as a vector bundle over a point.

\medskip
Recall that $A_k(E(1),\ldots ,E(k))$ is the $2^k$-dimensional real vector bundle over $\PP (E(1))\times_B \cdots\times_B \PP (E(k))$ with fibre at $(L_1,\ldots ,L_k)$, where $L_i\in\PP (E(i)_b)$, $b\in B$,  the real vector space of all functions $S(L_1)\times\cdots\times S(L_k)\longrightarrow \R$.
As a space of real-valued functions, each fibre of $A_k(E(1),\ldots ,E(k))$ can be equipped with a partial order by setting
\[
f_1\leq f_2  \quad \Longleftrightarrow \quad (\forall x\in S(L_1)\times\cdots\times S(L_k)) \ f_1(x)\leq f_2 (x)
\]
for $f_1,f_2\in A_k(E(1),\ldots ,E(k))$.
Hence, every finite non-empty subset of functions $S$  has a least upper bound, which we shall denote by $\max (S)$.

%--------------------------------------------------------------------------------------%
\subsection{Partitioning by polynomials}
\label{subsec : Partitioning by polynomials}
%--------------------------------------------------------------------------------------%

Now we give an extension of the results \cite[Thm.\,4.1]{GuthKatz2015}, \cite[Thm.\,0.3]{Guth2015} and \cite[Thm.\,1.3]{Blagojevic2017-Polynomial} to the setting of mass assignments over an arbitrary real vector bunlde $E$.  
In the case of a vector bundle over a point we recover the original results.

For an integer $d\geq 0$, let $\mathcal{P}^d(E)$ denote the real vector bundle of dimension $\binom{n+d-1}{d}$ over $B$ with fibre at $b\in B$ the vector space of homogeneous polynomial functions $v \colon E_b\longrightarrow \R$ of degree $d$. 
It is the dual $(S^d E)^*$ of the vector bundle obtained from the $d$th symmetric power of $E$.
If $d=1$, we can identify $\mathcal{P}^1(E)=E^*$ with $E$ using the inner product.

\medskip
In the following the crucial property of polynomial functions that we shall need is that for a non-zero homogeneous polynomial function $v\in\mathcal{P}^d(V)$, the zero-set 
\[
Z(v)=\{ x\in S(V)\, |\,  v(x)=0\}
\]
is null with respect to the Lebesgue measure on the Riemannian manifold $S(V)$. 
It follows that, for any $\epsilon >0$, there is an open neighbourhood of $Z(v)$ in the sphere $S(V)$ with volume less than $\epsilon$, consult \cite{Whitney1937}.

\def\Aa{\mathcal{A}}

\medskip
Now we extend our discussion from Section \ref{subsec : CSTM for mass assignements + constraints}.
Assume that $E(i)\subseteq \mathcal{P}^{d(i)}(E)$ is a vector subbundle of the vector bundle of homogeneous polynomial functions of degree $d(i)\geq 1$.
For $b\in B$, $(L_1,\ldots ,L_k)\in \PP(E(1)_b)\times\cdots\times \PP (E(k)_b)$, and $(v_1,\dots ,v_k)\in S(L_1)\times\cdots\times S(L_k)$, let us define an analogue of an orthant by
\[
\Aa_{b; v_1,\ldots ,v_k} := \{ u\in S(E_b) \, |\, v_1(u)>0,\, \ldots ,\, v_k(u)>0\}.
\]
We note that any real continuous function on the sphere bundle $\varphi \colon S(E)\longrightarrow \R$ restricts to a function $\varphi_b \colon  S(E_b)\longrightarrow \R$ which can be integrated over the set $\Aa_{b; v_1,\ldots ,v_k}$.

\medskip
The first generalization of \cite[Thm.\,4.1]{GuthKatz2015}, and also at the same time extension of our Theorem \ref{th : main result 02}, can be stated as follows.

\begin{theorem}\label{th : gen 01}
Under the hypotheses in the text, for an integer $j\geq 1$, given continuous functions $\varphi_1,\dots, \varphi_j\colon S(E)\longrightarrow \R$ assume that the $\F_2$-cohomology Euler class
\[
e(A_k(E(1),\ldots ,E(k))/\underline{\R})^j\in H^{(2^k-1)j}(\PP (E(1))\times_B\cdots\times_B \PP (E(k));\F_2)
\]
of the vector bundle $\underline{\R}^j\otimes (A_k(E(1),\ldots ,E(k))/\underline{\R})\cong (A_k(E(1),\ldots ,E(k))/\underline{\R})^{\oplus j}$ is non-zero.

\smallskip\noindent
Then there exists a point $b\in B$ and lines $L_i\in\PP (E(i)_b)$, $1\leq i\leq k$, such that, for each $1\leq \ell \leq j$, the function
\[
S(L_1)\times\cdots\times S(L_k) \longrightarrow \R,\qquad
(v_1,\ldots ,v_k)\longmapsto 
\int_{\Aa_{b;\, v_1,\ldots ,v_k}}  (\varphi_{\ell})_b
\]
is constant.	
\end{theorem}
\begin{proof}
As in in the Section \ref{subsec : CSTM for mass assignements + constraints}, we define for any continuous function $\varphi \colon S(E)\longrightarrow \R$ a section $s_{\varphi}$ of the vector bundle $A_k(E(1),\ldots ,E_k)$ by
\[
s_\varphi (b,(L_1,\ldots ,L_k))(v_1,\ldots ,v_k):= \int_{\Aa_{b;\, v_1,\ldots ,v_k}}  \varphi_b .
\]
Continuity of $s_\varphi$ follows from the fact that zero sets of polynomial functions are sets of Lebesgue measure zero on the sphere $S(V)$. 
The proof then follows the pattern of arguments in the proof of Theorem \ref{th : CSTM for assignement problem with constraints}	\!\!.
\end{proof}

The result remains true if the functions $\varphi_{\ell}$ are only assumed to be integrable in an appropriate sense.
Form the locally trivial bundle $L_B^1(S(E);\R )\longrightarrow B$ with fibre at $b\in B$ the Banach space $L^1(S(E_b);\R )$ of all absolutely Lebesgue integrable functions $S(E_b)\longrightarrow \R$.
If $\varphi$ is a section of this Banach bundle, then we can integrate $\varphi_b\in L^1(S(E_b);\R )$ and the associated section $s_{\varphi}$ is continuous.

\medskip
Next, we extend our results to probability measures as follows. 
Let us write $M_+(S(E))\longrightarrow B$ for the locally trivial bundle with fibre at $b\in B$ the space $M_+(S(E_b))$ of all finite Borel measures on the sphere $S(E_b)$, see Section \ref{subsec : What is the GHR problem for mass assignments?}.
A continuous section $\mu$ of $M_+(S(E))$ will be called a {\it family of probability measures} on $S(E)$ if $\mu_b\in M_+(S(E_b))$ is a probability
measure for each $b\in B$. 
In this more general context the zero set  of a polynomial function can have positive measure.

\medskip
Now, for each $b\in B$ and every $(L_1,\ldots ,L_k)\in \PP (E(1)_b)\times\cdots\times \PP (E(k)_b)$, 
we have $2^k$ non-negative real numbers 
$
\mu_b( \Aa_{b; v_1,\ldots ,v_k})\in \R$,
$
(v_1,\ldots ,v_k)\in S(L_1)\times\cdots\times S(L_k),
$
-- the measures of generalised orthants -- with sum less than or equal to $1$ (the measure of a zero sets can be positive).

\medskip
The following proposition allow us to transfer our more general setup in the previously developed topological framework.

\begin{proposition}\label{a5}
Assume that for an integer $j\geq 1$ there exist families of probability measures $\mu_1,\ldots ,\mu_j$ on $S(E)$ with the property that, for each $b\in B$ and every $(L_1,\ldots ,L_k)\in \PP (E(1)_b)\times\cdots\times \PP (E(k)_b)$, there is $(v_1,\ldots ,v_k)\in S(L_1)\times\cdots\times S(L_k)$ and some $\ell$ such that  $(\mu_{\ell})_b(\Aa_{b; v_1,\ldots ,v_k})> 1/2^k$.	

\smallskip\noindent
Then the vector bundle $\underline{\R}^j\otimes \big(A_k(E(1),\ldots ,E(k))/\underline{\R}\big)\cong (A_k(E(1),\ldots ,E(k))/\underline{\R})^{\oplus j}$
has a nowhere zero section.
\end{proposition}
\begin{proof}
For a fixed integer $1\leq \ell\leq j$, consider the set of points 
\begin{multline*}
U_{\ell}:=\big\{ 	\,
x=(b; L_1,\ldots ,L_k)\in \PP (E(1))\times_B\cdots\times_B \PP (E(k)) : \\
(\exists (v_1,\ldots ,v_k)\in S(L_1)\times\cdots\times S(L_k)) \ \ (\mu_{\ell})_b(\Aa_{b; v_1,\ldots ,v_k})> 1/2^k \, \big\},
\end{multline*}
which is an open subspace of the base space $X:=\PP (E(1))\times_B\cdots\times_B \PP (E(k))$.
From the assumption it follows that $U_1,\dots, U_j$ forms an open cover of the base space $X$,

\medskip
Using the local triviality of the vector bundles, for every point $x\in X$ we can manufacture a (continuous) section $s_{\ell}^x$ of
$A_k(E(1),\ldots ,E(k))$ and an open neighborhood $U_{\ell}^x$ of $x$ such that for each
$
x'=(b'; L_1',\ldots ,L_k')\in 
X
$
the following holds
\begin{compactenum}[\rm \quad (i)]

\item $s^x_{\ell}(x')(v_1',\ldots ,v_k')\in [0,1]$, for all $(v_1',\ldots ,v_k')\in S(L_1')\times\cdots\times S(L_k')$;	

\item if $s^x_{\ell}(x')(v_1',\ldots ,v_k')=1$, then $(\mu_{\ell})_{b'}(\Aa_{b'; v_1',\ldots ,v_k'})> 1/2^k$;

\item if $x'\in U_{\ell}^x$, then there is some $(v_1',\ldots ,v_k')$ such that $s^x_{\ell}(x')(v_1',\ldots ,v_k')=1$.
\end{compactenum}
Since $X$ is compact, and $U_1, \dots, U_j$ forms an open cover of $X$ it can be refined to a compact cover $K_1,\ldots ,K_j$ of $X$ with the property that $K_{\ell}\subseteq U_{\ell}$ for $1\leq \ell \leq j$.

\medskip
Now, for each $\ell$, we can choose a finite subset $S_{\ell}\subseteq U_{\ell}$ such that $K_{\ell}\subseteq \bigcup_{x\in S_{\ell}} U_{\ell}^x$.
This allows as to define a continuous section $s_{\ell}$ of $A_k(E(1),\ldots ,E(k))$ as $s_{\ell}:=\max\,\{ s^x_{\ell} : x\in S_{\ell}\}$.
Here the maximum is taken with respect to the partial order on the space of real valued functions which was introduced at the beginning of this section. 
The properties (i), (ii) and (iii) ensure that at each point $x\in K_{\ell}$ at least one of the $2^k$ components of $s_{\ell}(x)$ is equal to $1$, but that not all are equal to $1$. 
Thus, the associated section $\bar s_{\ell}$ of $A_k(E(1),\ldots ,E(k))/\underline{\R}$  has no zeros in $K_{\ell}$.
The sum $(\bar s_1,\ldots ,\bar s_j)$ is a nowhere zero section of
$\underline{\R}^j\otimes (A_k(E(1),\ldots ,E(k))/\underline{\R})$.	
\end{proof}

\medskip
Now a generalization of Theorem \ref{th : gen 01} can be stated as follows.

\begin{theorem}\label{th : gen 02}
Under the hypotheses in the text, suppose that for an integer $j\geq 1$, $\mu_1,\ldots , \mu_j$ are families of probability measures on the sphere bundle $S(E)$.

\smallskip\noindent
If the $\F_2$-cohomology Euler class
\[
e(A_k(E(1),\ldots ,E(k))/\underline{\R})^j\in
H^{(2^k-1)j}(\PP (E(1))\times_B\cdots\times_B \PP (E(k);\,\F_2)
\]
of the vector bundle $\underline{\R}^j\otimes \big(A_k(E(1),\ldots ,E(k))/\underline{\R}\big)$ is non-zero,
then there exists a point $b\in B$ and lines $L_i\in\PP (E(i)_b)$, $1\leq i\leq k$, such that, 
for each $1\leq \ell\leq j$ and every $(v_1,\ldots  ,v_k)\in S(L_1)\times\cdots\times S(L_k)$,
\[
\mu_\ell(\Aa_{b; v_1,\ldots ,v_k})\leq \frac{1}{2^k}.
\]
\end{theorem}
\begin{proof}
Since the Euler class is non-zero, every section of the vector bundle has a zero. So the assertion follows from 
Proposition \ref{a5}.
\end{proof}

\medskip
As an application we give a spherical version of a generalization in \cite[Thm.\,1.3]{Blagojevic2017-Polynomial} of \cite[Thm.\,4.1]{GuthKatz2015}.

\begin{corollary}
Let $V$ be a real vector space of dimension $n$.
There is a constant $C_n$ with the property that for integers $d\geq 1$ and $j\geq 1$, and probability
measures $\mu_1,\ldots ,\mu_j$ on the sphere $S(V)$, there exists a non-zero homogeneous polynomial function $v$ of degree $d$
on $S(V)$ such that for each component $\mathcal{O} $ of the complement in $S(V)$ of the zero set of $v$
\[
\mu_1(\mathcal{O} )<C_n\cdot \frac{j}{d^{n-1}},\ldots , \mu_j(\mathcal{O} ) <C_n\cdot \frac{j}{d^{n-1}}.
\]
\end{corollary}
\begin{proof}
Consider probability measures  $\mu_1,\ldots ,\mu_j$ and  fix an integer $k\geq 1$.
We shall apply Theorem \ref{th : gen 02} with $B=\pt$ a point, $V=\R^n$ and $E(i)=V(i)\subseteq \mathcal{P}^{d(i)}(V)$ a vector subspace of dimension $n_i\leq{n+d(i)-1\choose d(i)}$.

\medskip
Let $r (i)\geq 1$ be the least positive integer such that $r(i)^{n-1}> 2^{i-1}j$ and set 
\[
d(i)=(n-1)r(i) 
\qquad\text{and}\qquad
n_i=r(i)^{n-1}.
\]
Take $V(i)$ to be the $n_i=r(i)^{n-1}$ dimensional space of polynomials with basis, in terms of the standard coordinate functions $\xi_i$, the monomials 
\[
(\xi_1^{s_1}\xi_2^{r(i)-s_1})
(\xi_2^{s_2}\xi_3^{r(i)-s_2})\cdots
(\xi_{n-1}^{s_{n-1}}\xi_n^{r(i)-s_{n-1}}),
\]
where $0\leq s_1,\ldots ,s_{n-1}<r(i)$.

\medskip
It follows from Proposition \ref{prop : numbers} that
$\iota_k(n_1,\ldots ,n_k)\geq j$.
By Theorem \ref{th : gen 02}, there exist homogeneous polynomials $v_1,\dots,v_k$ of degree $d(1),\dots, d(k)$, respectively, such that $\mu_\ell(\Aa_{\pt; v_1,\ldots ,v_k})\leq 1/2^k$ for all $1\leq \ell\leq j$.
The product $v_1\cdots v_k$ has degree $d_k=d(1)+\ldots +d(k)$ and each component of the complement of its zero-set is contained  in some $\Aa_{\pt; v_1,\ldots ,v_j}$.
Since
\[
2^{\frac{i-1}{n-1}}\cdot j^{\frac1{n-1}} < r(i)\leq 
2\cdot 2^{\frac{i-1}{n-1}}\cdot j^{\frac1{n-1}},
\]
it follows that 
\begin{multline*}
d_k=d(1)+\ldots +d(k)\leq (n-1)(r(1)+\dots+r(k))\\ \leq 2(n-1)\cdot j^{\frac1{n-1}} \cdot \sum_{i=1}^k 2^{\frac{i-1}{n-1}} 
=	2(n-1)\cdot j^{\frac1{n-1}}\cdot \frac{2^{\frac{k}{n-1}}-1}{2^{\frac1{n-1}}-1}.
\end{multline*}
So $d_k^{n-1}< C'_n 2^kj$, where $C'_n=\Big(\frac{2(n-1)}{2^{\frac1{n-1}}-1} \Big)^{n-1}$.

\medskip
Now $(d_k)$ is a strictly increasing sequence.
If $k$ is chosen so that $d_k\leq d <d_{k+1}$, then $1/2^{k+1}<C'_nj/d_{k+1}^{n-1}$, and so
\[
\frac{1}{2^k} < C_n\cdot \frac{j}{d_{k+1}^{n-1}}\leq
 C_n\cdot \frac{j}{d^{n-1}},
\]
where $C_n=2C'_n$.
We can multiply $v_1\cdots v_k$ by any non-zero polynomial
of degree $d-d_1\cdots d_k$ to produce the required
polynomial of degree $d$.
 \end{proof}

%--------------------------------------------------------------------------------------%
\subsection{Partitioning by affine functions}
\label{subsec : Partitioning by affine functions}
%--------------------------------------------------------------------------------------% 

In this section we give an extension of our results on the spherical GHR problem for the mass assignments to the broader class of partitions by caps which are not necessarily hemispheres.

\def\Cc{\mathcal{C}}
\def\Ww{\mathcal{W}}

\medskip
Let $V$ be an $n$-dimensional real vector space with $n \geq 2$.
Using the inner product we can identify the vector space $\R\oplus V$ with the $(n+1)$-dimensional vector space of affine functions $V\longrightarrow \R$ where the pair $(t,w)\in\R\oplus V$ determines the function
$u\longmapsto t+\langle u,w\rangle$.
A unit vector $v=(t,w)\in S(\R\oplus V)$ decomposes the sphere $S(V)$ as the union  $S(V)=C(v)\cup C(-v)$ of 
two {\it caps}:
\[
C(v)=\{ u\in S(V) : \langle u,w\rangle \geq -t\}
\qquad\text{and}\qquad
C(-v)= \{ u\in S(V) : \langle u,w\rangle \leq -t\}
\]
with intersection $\{ u\in S(V) : \langle u,w\rangle =-t\}$.
For an illustration see Figure \ref{fig:caps}.
If $t=0$, the caps are hemispheres.
If $t>\| w\|$, then $C(-v)=\emptyset$; if $t<-\| w\|$, then $C(v)=\emptyset$.
If $t=\| w\|$, then $C(-v)$ is the single point $-w/\| w\|$, and, if $t=-\| w\|$, $C(v)=\{ w/\| w\|\}$.
The intersection $C(v)\cap C(-v)$, if $|t| <\| w\|$ is a sphere of dimension $n-2$ (if $n>1$, which we now assume).

\begin{figure}
\includegraphics[scale=01.0]{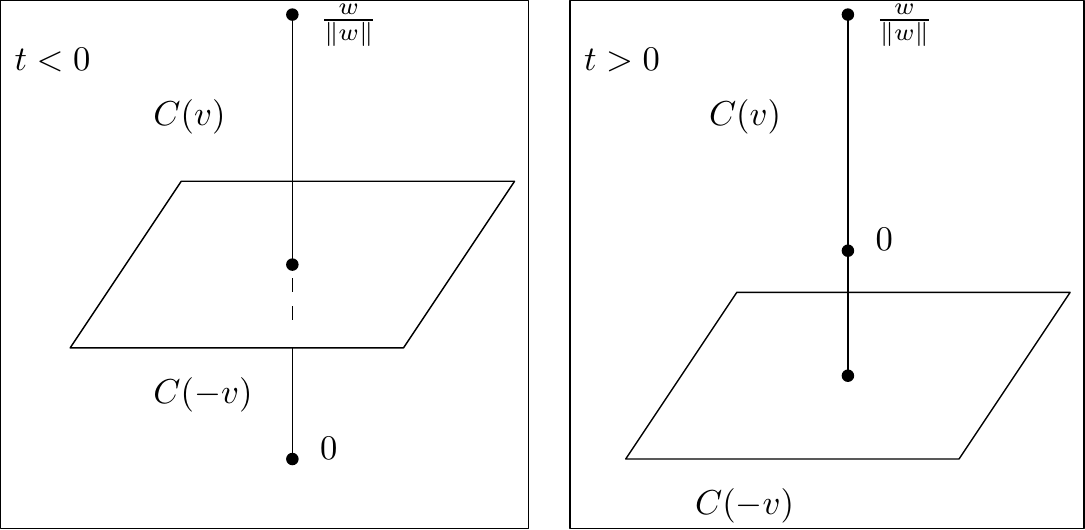}
\caption{\small The halfspaces defining the caps.}	
\label{fig:caps}
\end{figure}

\medskip
Now suppose that each vector bundle $E(i)$ is a subbundle of $\underline{\R}\oplus E$, regarding a vector $v\in E(i)_b\subseteq \R\oplus E_b$ in the fibre at $b\in B$ as an affine linear function $E_b\longrightarrow\R$. 
For a point $b\in B$, a collection of lines 
$(L_1,\ldots ,L_k)\in \PP (E(1)_b)\times\cdots\times \PP (E(k)_b)$, and a collection of vectors $(v_1,\ldots ,v_k)\in S(L_1)\times\cdots \times S(L_k)$,
we write another analogue of an orthant by
\[
\Cc_{b; v_1,\ldots ,v_k} :=
\{ u\in S(E_b) : v_1(u)>0,\, \ldots ,\, v_k(u)>0\}.
\]
The corresponding equipartition theorem is proved in the usual way
by constructing a section of the vector bundle
$\underline{\R}^j\otimes \big( A_k(E(1),\ldots ,E(k))/\underline{\R}\big)$.
\begin{theorem}
Under the hypotheses in the text, suppose that
for an integer $j\geq 1$, $\varphi_1,\dots,\varphi_j\colon S(E)\longrightarrow\R$ are continuous functions.

\smallskip\noindent
If the $\F_2$-cohomology Euler class
\[
e(A_k(E(1),\ldots ,E(k))/\underline{\R})^j\in
H^{(2^k-1)j}(\PP (E(1))\times_B\cdots\times_B \PP (E(k);\F_2)
\]
of the vector bundle $\underline{\R}^j\otimes\big( A_k(E(1),\ldots ,E(k))/\underline{\R}\big)$ is non-zero, then there exists a point $b\in B$ and lines $L_i\in\PP (E(i)_b)$, $1\leq i\leq k$ such that, 
for each $1\leq\ell\leq j$, the function
\[
S(L_1)\times\cdots\times S(L_k) \longrightarrow \R,\quad
(v_1,\ldots ,v_k)\longmapsto 
\int_{\Cc_{b;\, v_1,\ldots ,v_k}} \, (\varphi_\ell)_b
\]
is constant.
\end{theorem}

%--------------------------------------------------------------------------------------%
\subsection{Partitioning by spherical weges}
\label{subsec : Partitioning by spherical wedges}
%--------------------------------------------------------------------------------------% 

Next we describe an extension of the results of Schnider \cite{Schnider2019} and Sober\'on \& Takahashi \cite{SoberonTakahashi}.

\medskip
Let $V$ be a vector space of dimension $n\geq 3$, and let $U\subseteq V$ be a vector subspace of dimension $m\geq 2$. 
Then $V=U\oplus U^\perp$ is the direct sum of $U$ and its orthogonal complement $U^\perp$ and the unit sphere $S(V)=S(U\oplus U^\perp)$ is the join $S(U)*S(U^\perp)$.
To be precise, we also think of the join as the space
\[
S(V)=
\{ \cos (\theta )x+\sin (\theta )y \, |\, x\in S(U),\, y\in
S(U^\perp),\, 0\leq \theta\leq \pi /2\}.
\]
Like in the previous section, given $v=(t,w)\in S(\R\oplus U)$, we have the decomposition of the sphere $S(U)$ as $C(v)\cup C(-v)$, where
\[
C(v)=\{ u\in S(U) : \langle u,w\rangle \geq -t\}
\qquad\text{and}\qquad
C(v)=\{ u\in S(U) : \langle u,w\rangle \geq -t\}.
\]
This leads to a decomposition of the bigger sphere $S(V)$ as the union $W(v,U)\cup W(-v,U)$ of two {\it wedges} 
\begin{align*}
W(v,U)&=C(v)*S(U^\perp )\\
&=\{ \cos (\theta )u+\sin(\theta )y: u\in S(U),\, y\in S(U^\perp ), \, \langle u,w\rangle \geq -t,\, 0\leq\theta\leq\pi /2 \}	
\end{align*}
and
\begin{align*}
W(-v,U)&=C(-v)*S(U^\perp)\\
&=\{ \cos (\theta )u+\sin(\theta )y: u\in S(U),\, y\in S(U^\perp ), \, \langle u,w\rangle \leq -t,\, 0\leq\theta\leq\pi /2 \}.
\end{align*}
The intersection $W(v,U)\cap W(-v,U)$ is $S(U^\perp)$ if $|t|>\| w\|$, a disc of dimension $n-m$ if $|t|=\| w\|$, and a sphere of dimension $n-2$ if $|t| <\| w\|$.
(The subspace $\{ rx : r\geq 0,\, x\in W(v,U)\}$ of $V$ is an {\it $m$-cone} in the sense of \cite{Schnider2019}.)

\medskip
For example, take $U=\R^2$, $U^\perp =\R$, $V=\R^2\oplus\R$, so that $m=2$, $n=3$.
The wedges $W(v,U)$, where $v=(t,w)\in S(\R^2\oplus\R)$ with $|t|<\| w\|$, are the subsets
\[
\{ (\cos (\theta )\cos (\phi ),\, \cos (\theta )\sin (\phi),\,
\sin (\theta ))\in S(\R^2\oplus\R)  :  \alpha\leq\phi\leq\beta ,\, -\pi /2\leq 
\theta\leq\pi /2\}  
\]
where $0\leq \alpha <\beta <2\pi$.

\medskip
Now suppose that $F(i)\subseteq E$ is a vector subbundle of dimension $m_i\geq 2$, for every $1\leq i \leq k$, and that $E(i)$ is a subbundle of $\underline{\R}\oplus F(i)$ of dimension $n_i\leq m_i+1$.
For a point $b\in B$, lines $(L_1,\ldots ,L_k)\in \PP (E(1)_b)\times\cdots\times \PP (E(k)_b)$, and vectors $(v_1,\ldots ,v_k)\in S(L_1)\times\cdots\times S(L_k)$, we write
\[
\Ww_{b; v_1,\ldots ,v_k} :=\bigcap_{i=1}^k \big(S(E_b) -W(-v_i,F(i)_b)\big)
\]
as the intersection of the open subsets $S(E_b)-W(-v_i,F(i)_b)\subseteq W(v_i,E(i)_b)$.

\medskip
As in all the previous partition problems for mass assignments we derive the following result in almost identical manner.

\begin{theorem}\label{b3}
Under the hypotheses in the text, suppose that $j\geq 1$ is an integer, $\varphi_1,\dots , \varphi_j \colon 
S(E)\longrightarrow \R$ are continuous functions, and $j\leq \iota_k(E(1),\ldots ,E(k))$.

\smallskip\noindent
Then there exists a point $b\in B$ and lines $L_i\in\PP (E(i)_b)$, $1\leq i\leq k$, such that, for each $1\leq \ell \leq j$, the function
\[
S(L_1)\times\cdots\times S(L_k) \longrightarrow \R,\qquad
(v_1,\ldots ,v_k)\longmapsto 
\int_{\Ww_{b;\, v_1,\ldots ,v_k}} \, (\varphi_\ell)_b
\]
is constant.	
\end{theorem}

\medskip
In the special case of a vector bundle over a point we get the following corollary.

\begin{corollary}\label{b4}
Suppose that $j\geq 1$ is an integer, $\varphi_1, \ldots ,\phi_j : S(V)\longrightarrow \R$ are continuous functions and that $j\leq \iota_k (n+1,\ldots ,n+1)$.
Let $m_1,\ldots ,m_k$ be integers in the range $2\leq m_i \leq n$.

\smallskip\noindent

Then there exist vector subspaces $U_1,\ldots ,U_k\subseteq V$ with $\dim (U_i)=m_i$ and lines $L_i\in\PP (\R\oplus U_i)$, $1\leq i\leq k$, such that for each $1\leq \ell \leq j$, the function
\[
S(L_1)\times\cdots\times S(L_k) \longrightarrow \R,\quad
(v_1,\ldots ,v_k)\longmapsto 
\int_{W(v_1,U_1)\,\cap\,\cdots\,\, \cap\, W(v_k,U_k)} \, \varphi_\ell
\]
is constant.	 
\end{corollary}

\begin{proof}
Take $B$ to be the product $G_{m_1}(V)\times\cdots\times G_{m_k}(V)$ of Grassmann manifolds  and $F(i)$ to be the canonical $m_i$-dimensional bundle over the $i$th factor. 
Apply Theorem \ref{b3} with $n_i=m_i+1$ and $E(i)=\underline{\R}\oplus F(i)$. 
Indeed, since $\iota_1(\underline{\R}\oplus F(i))=n$, we have that $\iota_k(E(1),\ldots ,E(k))=\iota_k(n+1,\ldots ,n+1)$ by
Proposition \ref{prop : for k >1}.
\end{proof}

\begin{remark}\label{b5}
The previous Corollary \ref{b4} can be sharpened by restricting the base space in the following way.
Replace the Grassmann manifolds $G_{m_i}(V)$, where $V=\R^n$, by the its subspace $\PP (\R^{n-m_i+1})$, embedded by taking the direct sum of a line in $\R^{n-m_i+1}$ with $\R^{m_i-1}$ to get a subspace of $\R^n=\R^{n-m_i+1}\oplus\R^{m_i-1}$ of dimension $m_i$.
Then the vector bundle $E(i)$ restricts to $\underline{\R}^{m_i}\oplus H_i$ where $H_i$ is the Hopf line bundle $H(\R^{n-m_i+1})$.
We have $\iota_1(\underline{\R}^{m_i}\oplus H_i)=n$, because $w_{n-m_i}(-H_i)\not=0$.	
\end{remark}

\medskip
To illustrate the  given in Corollary \ref{b4} we spell out the special case $n=3$, $j=3$, $k=1$, $m_1=2$, for which $\iota_1(3+1)=3$.
Suppose that $\varphi_1, \varphi_2, \varphi_3\colon S^2=S(\R^3)\longrightarrow \R$ are continuous functions.
Then there is a wedge $W\subseteq \R^3$, specified by a plane $U$ through the origin in $\R^3$ and $(t,w)\in S(\R\oplus U)$, such that
\[
\int_{W} \phi_1=\frac{1}{2}\int_{S^2} \phi_1\, ,\qquad
\int_{W} \phi_2=\frac{1}{2}\int_{S^2} \phi_2\, ,\qquad
\int_{W} \phi_3=\frac{1}{2}\int_{S^2} \phi_3\, .
\]

\medskip
Furthermore, the $k=1$ case of Corollary \ref{b4} gives \cite[Thm.\,8]{Schnider2019}, and also the spherical version of \cite[Thm.\,1.2 and Thm.\,3.2]{SoberonTakahashi}.

\begin{corollary}
 \label{b7}
Suppose that $\varphi_1, \ldots ,\varphi_n\colon S(\R^n)\longrightarrow \R$ are continuous functions and $m$ is an integer, $2\leq m\leq n$.
Write $V=\R^n$ and $V'=\R^{m-1}\subseteq \R^{n-m+1}\oplus\R^{m-1}=V$.

\smallskip\noindent
Then there exists a vector subspace $U\subseteq V$ of dimension $m$ containing the subspace $V'$ and a vector $v\in S(\R\oplus U)$ such that
\[
\int_{W(v,U)}\varphi_\ell =
\frac{1}{2}\int_{S(V)}\phi_l =\int_{W(-v,U)}\varphi_\ell
\]
for $\ell=1,\ldots ,n$.
\end{corollary}
\begin{proof}
We just need to recall that $\iota_1(n+1)=n$.
The sharpening, to give the restriction that
 $U$ should contain $V'$, is given by Remark \ref{b5}.
\end{proof}

\medskip 
The connection between the affine and spherical cases was discussed in Section \ref{Sec 1.3}. 
We explain how \cite[Thm.\,1.2]{SoberonTakahashi} can be deduced from the case $m=2$ of our Corollary \ref{b7}.

\begin{corollary}
For an integer $n\geq 2$, suppose that $\psi_1,\ldots ,\psi_n\colon\R^{n-1}\longrightarrow \R$ are continuous
functions with compact support with the $n$ integrals $\int_{\R^{n-1}}\psi_{\ell}$, $1\leq \ell\leq n$, not all equal to zero.

\smallskip\noindent
Then there exist two distinct parallel hyperplanes in $\R^{n-1}$ such that the closed region $S$ sandwiched between them satisfies 
\[
\int_S \psi_l =\frac{1}{2}\int_{\R^{n-1}} \psi_\ell\, ,
\]
for all $1\leq \ell \leq n$.	
\end{corollary}

\noindent
(Note that if all the integrals $\int_{\R^{n-1}}\phi_\ell$ are zero, then there is a trivial statement for any two coinciding hyperplanes.)
\begin{proof}
Consider  the diffeomorphism
\[
\pi \colon \Lambda =
\{ (x,y)\in S(\R^{n-1}\oplus \R ) : y>0\} \longrightarrow \R^{n-1}\,,\qquad\qquad
 (x,y)\longmapsto \frac{x}{y}\,,
\]
which maps intersections of linear subspaces of $\R^{n-1}\oplus\R$ with $\Lambda$ to affine subspaces of $\R^{n-1}$.
Each density $\psi_\ell$ lifts to a density $\varphi_\ell$ on $S(\R^{n-1}\oplus\R)$ with support in the open upper hemisphere $\Lambda$.
(To be precise, $\varphi(x,y)=y^n\psi (x/y)$.)

\medskip
Let $U\subseteq \R^{n-1}\oplus\R =V$ be a $2$-dimensional vector subspace and $v\in S(\R\oplus U)$ a vector as provided by Corollary \ref{b7} when $m=2$.
Since some $\int_{\R^{n-1}}\psi_\ell$ is non-zero, both $S(V)-W(-v,U)$ and $S(V)-W(v,U)$ have to be non-empty.
The intersection $W(v,U)\cap W(-v,U)$ is, therefore, the union of two discs 
\[
\{ a\} * S(U^\perp )\subseteq S(\R\cdot a\oplus U^\perp)
\qquad\text{and}\qquad
\{ b\} *S(U^\perp )\subseteq S(\R\cdot b\oplus U^\perp)
\]
meeting in $S(U^\perp)$.
Here $a, b\in S(U)$.

\medskip
The image of the intersection $\pi (W(v,U)\cap W(-v,U)\cap\Lambda )$ is the union of two affine hyperplanes meeting in $\pi (S(U^\perp )\cap\Lambda )$.

\medskip
We can prescribe the subspace $V'$ in Corollary \ref{b7} to be the line $0\oplus \R \subseteq \R^{n-1}\oplus\R$.
In that case, $S(U^\perp )\cap\Lambda$ is empty, and the two hyperplanes are parallel.	
\end{proof}

%--------------------------------------------------------------------------------------%
%--------------------------------------------------------------------------------------%
%--------------------------------------------------------------------------------------%
%--------------------------------------------------------------------------------------%
\section{Concluding remarks: Real flag manifolds}
\label{sec : Concluding remarks}
%--------------------------------------------------------------------------------------%
%--------------------------------------------------------------------------------------%
%--------------------------------------------------------------------------------------%
%--------------------------------------------------------------------------------------%

In the final section we make further remarks on particular arguments used in the proofs of our results.

\medskip
For a Euclidean vector space $V$ of dimension $n$ and integers $0=n_0 <n_1 < \cdots <n_k<n$, let $B:=\flag_{n_1,\dots,n_k} (V)$ be the manifold of flags $(V_*):0=V_0\subseteq V_1\subseteq \cdots \subseteq V_k\subseteq V$ with $\dim V_i=n_i$. 
The canonical bundles of dimension $n_i$ over $B$ are denoted by $E(i)$, as in the statement of Corollary \ref{cor : B}.
Write $E$ for the trivial bundle over $B$ with fibre $V$.

\begin{proposition}
\label{c1}
The $\Ff_2$-Euler classes satisfy
$$
\prod_{i=1}^k \e(E/E(i))^{n_i-n_{i-1}}\not=0
\in H^d(B;\,\Ff_2)=\Ff_2,
$$
where $n_0=0$ and the dimension $d$ is equal to $\sum_{i=1}^k (n-n_i)(n_i-n_{i-1})$.	
\end{proposition}
\begin{proof}
Let $(U_*):0=U_0\subseteq U_1\subseteq \cdots \subseteq U_k$ be a fixed flag in $V$.
The general linear group $G=\GL (V)$ acts transitively on $B$. 
If $H\leq G$ denotes the stabilizer of $(U_*)$, we have a map $\pi\colon G\longrightarrow B$ defined  given $\pi (g) =(gU_*)$ which describes $B$ as the homogeneous space $G/H$.

\medskip
The derivative of $\pi$ at $1\in G$ is a map from the Lie algebra $\gg =\End (V)$ onto the tangent space of $B$ at $(U_*)$ with kernel the Lie algebra $\hh$ of $H$, that is, the space of endomorphisms $a$ of $V$ such that $a(U_i)\subseteq U_i$ for all $i$.
The tangent bundle of $B$ is the quotient of the trivial Lie algebra bundle $B\times \GL (V)=\GL(F)$ by the subbundle with fibre at $(V_*)\in B$ the quotient of $\End (V)$ by the Lie subalgebra, $\hh (V_*)$ of endomorphisms that preserve the flag.
Using the inner product, we can express $\hh (V_*)$ as $\bigoplus_{i=1}^k \Hom(V_i^\perp, V_i\cap (V_{i-1}^\perp))$,
which has dimension $\sum_{i=1}^k (n-n_i)(n_i-n_{i-1})$.

\medskip
Now consider the vector bundle $E'$, defined as a quotient of $B\times \End(V)$, with the fibre at $(V_*)\in B$ the quotient of $\gg =\End(V)$ by the vector subspace $\hh (V_*,U_*)$ of maps $a\colon V\longrightarrow V$ such that $a(V_i)\subseteq U_i$ for 
$i=1,\ldots ,k$.
In metric terms, 
\[
E'=\bigoplus_{i=1}^k{\rm Hom}(F(i)^\perp, U_i\cap (U_{i-1}^\perp ))\, ,
\]
and its Euler class $\e(E')$ is equal to $\prod_{i=1}^k e(E(i)^\perp )^{n_i-n_{i-1}} \in H^d(B;\, \Ff_2)$.

\medskip
The vector bundle $E'$ over the closed connected $d$-dimensional manifold $B$ has the same dimension $d$.
We shall prove that $\e(E')$ is non-zero by writing down a smooth section $s$ of $E'$ with exactly one zero and checking that the (mod $2$) degree of that zero is equal to $1$.

\medskip
The section $s$ is defined to have the value at $(V_*)$ given, modulo $\hh (V_*,U_*)$, by the
identity endomorphism $1\in \End (V)$. 
At a zero of $s$, $V_i\subseteq U_i$ for all $i$, that is, $V_*=U_*$.
At this zero, the tangent space of $B$ coincides with the fibre of $E'$, and we shall show that the derivative of $s$ is the identity endomorphism of $\gg /\hh$.

\medskip
To do this, we lift from $B=G/H$ to $G$ by the projection $\pi$.
The pullback $\pi^*E'$ is trivialized by the isomorphism
\[
G\times( \gg /\hh) \longrightarrow E'
\]
taking $(g,a+\hh )$, where $g\in\GL (V)$, $a\in\End (V)$, to $((gU_*),ag^{-1}+\hh (gU_*))$.
And the section $s$ lifts to the map
\[
G\to \gg /\hh \, : \, g\longmapsto g+\hh ,
\]
for which the derivative at $1$ is, transparently, the projection
$\gg \to \gg /\hh$.
This completes the proof.
\end{proof}

\medskip
Writing the quotient $E/E(i)=E(i)^\perp$ as the direct sum $\bigoplus_{j=i}^k(E(j+1)/E(j))$,
where $E(k+1)=E$, we can reformulate Proposition \ref{c1} as follows.

\begin{corollary}\label{c2}
The product of Euler classes
\[
\prod_{i=1}^k \e(E(i+1)/E(i))^{n_i}\in H^d(B;\,\Ff_2)
\]
is non-zero.	
\end{corollary}

\medskip
The previous Corollary \ref{c2} connects with previously given arguments in the following way:
\begin{compactitem}[\qquad --]
\item The case $k=1$, shows that $\e(E(1)^\perp)^{n-n_1}\not=0$, and in particular $\e(E(1)^\perp )\not=0$, as used in the proof of Corollary \ref{cor : ham-sandwich}.
\item The statement $\e(E(1)^\perp)^{n-n_1}\not=0$ is the result needed in Section \ref{sub : proof of cor: A} for the proof of Corollary \ref{cor : A}.
\item For general $k$, we have in particular that $\prod_{i=1}^k \e(E(i)^\perp )\not=0$. This is what is required in Section \ref{subsec : cor : flag}  to prove Corollary \ref{cor : B}.
\item If $n_i=n-k +i-1$ (that is, $n_1=n-k$, $n_2=n-k+1$, $\ldots$, $n_k=n-1$), then $\e(E(1)^\perp)^k\e(E(2)^\perp)^{k-1}\cdots \e(E(k)^\perp)^1\not=0$. This is what is needed in Section \ref{subsec : cor : flag 03} to prove Theorem \ref{cor : F}. It shows directly that $\e(E_{k+1})^k\cdots \e(E_{d+1})^d\not=0$. The permutation symmetry of the cohomology then gives
\[\e(E_{k+1})^{j_k}\cdots \e(E_{d+1})^{j_d}\not=0.\]
\end{compactitem}

\medskip
Thus, different arguments we offered in the proofs can be seen as a direct consequences of Corollary \ref{c2}. 
To best of our knowledge these implications were not known until now, and we believe it was worth explaining these connections.

%--------------------------------------------------------------------------------------%
\subsection*{Acknowledgements.}
%--------------------------------------------------------------------------------------%
The authors are grateful to the referee of the paper for careful reading and many useful suggestions which guided our revision of the original manuscript.
Additionally, we would like to thank Aleksandra Dimitrijevi\'c Blagojevi\'c and Matija Blagojevi\'c for useful discussions and many improvements of the manuscript.

%%--------------------------------------------------------------------------------------%
%\small
%\bibliography{references-mass-assignments}{}
%\bibliographystyle{amsplain}
%%--------------------------------------------------------------------------------------%

\end{document}